\pdfoutput=1

\documentclass[a4paper, twoside]{article}

\usepackage{amssymb}
\usepackage{amsmath}
\usepackage{amsthm}

\usepackage[utf8]{inputenc}
\usepackage[T1]{fontenc}   
\usepackage{lmodern}       
\usepackage{mathrsfs}      
\usepackage{upgreek}       
\usepackage{microtype}     
\usepackage{xurl}          
\usepackage{hyperref}      
\usepackage{accents}       
\usepackage[all]{xy}       

\usepackage[nottoc,numbib]{tocbibind}

\newcounter{enumi_memory}  

\theoremstyle{definition}
\newtheorem{step}{Step}    

\swapnumbers               
\theoremstyle{plain}
\newtheorem{theorem}{Theorem}[section]
\newtheorem{lemma}[theorem]{Lemma}
\newtheorem{corollary}[theorem]{Corollary}

\theoremstyle{remark}
\newtheorem{remark}[theorem]{Remark}
\newtheorem{example}[theorem]{Example}

\theoremstyle{definition}
\newtheorem{definition}[theorem]{Definition}
\newtheorem{miniremark}[theorem]{}

\newcommand{\restrict}{\mathop{\llcorner}}

\newcommand{\noopsort}[1]{} 

\newcommand{\with}{\!:\!}

\newcommand{\vect}[1]{\accentset{\rightharpoonup}{#1}}

\DeclareMathOperator{\ap}{ap}               
\DeclareMathOperator*{\aplim}{ap\,lim}      
\DeclareMathOperator{\boundary}{\partial}   
\DeclareMathOperator{\card}{card}           
\DeclareMathOperator{\diam}{diam}           
\DeclareMathOperator{\dist}{dist}           
\DeclareMathOperator{\dmn}{dmn}             
\DeclareMathOperator{\im}{im}               
\DeclareMathOperator{\spt}{spt}             
\DeclareMathOperator{\without}{\sim}        
\DeclareMathOperator{\Bdry}{Bdry}           
\DeclareMathOperator{\Clos}{Clos}           
\DeclareMathOperator{\Der}{D}               
\DeclareMathOperator{\Hom}{Hom}             
\DeclareMathOperator{\Int}{Int}             
\DeclareMathOperator{\Jac}{J}               
\DeclareMathOperator{\Lip}{Lip}             
\DeclareMathOperator{\Tan}{Tan}             

\newcommand{\ud}{\,\mathrm{d}}

\title{A priori bounds for geodesic diameter. Part I. Integral chains with
coefficients in a complete normed commutative group}

\author{Ulrich Menne \and Christian Scharrer}

\begin{document}

\maketitle 

\begin{abstract}
	As service to the community, we provide---for Euclidean space---a
	basic treatment of locally rectifiable chains and of the complex of
	locally integral chains.  In this setting, we may beneficially develop
	the idea of a complete normed commutative group bundle over the
	Grassmann manifold whose fibre is the coefficient group of the chains.
	Our exposition also sheds new light on some algebraic aspects of the
	theory.  Finally, we indicate an extension to a geometric approach to
	locally flat chains centring on locally rectifiable chains rather than
	completion procedures.
\end{abstract}

\paragraph{MSC-classes 2020}
	49Q15 (Primary); 49Q20 (Secondary).


\tableofcontents

\section{Introduction}

Throughout the introduction, \emph{$m$ is a nonnegative integer, $n$ and $d$
are positive integers, $U$ is an open subset of $\mathbf R^n$, and $G$ is a
complete normed commutative group.}  Our notation is based on H.\ Federer's
treatise \cite{MR41:1976}, see Section \ref{section:notation}; in particular,
$\mathbf I_m^{\textup{loc}} (U)$, $\mathscr R_m^{\textup{loc}} ( U)$, and
$\mathscr F_m^{\textup{loc}} (U)$ denote the commutative groups of those $m$
dimensional currents in $U$ which locally are integral, rectifiable, and flat,
respectively.

\subsection{Overview}

The primary goal of this first paper of our series is to provide a
self-contained exposition with complete proofs of all basic facts for locally
rectifiable chains and locally integral chains in $U$ with coefficients in $G$
to be employed in the third and final paper (see \cite{MenneScharrer2-3}).
The Euclidean setting allows us to beneficially employ the concept of $m$
dimensional approximate tangent planes in $\mathbf R^n$ through the usage of
rectifiable varifolds and the study of the \emph{complete normed commutative
group bundle} $\mathbf G (n,m,G)$ over the Grassmann manifold $\mathbf G(n,m)$
with fibre $G$; the latter is an idea originating from F.\ Almgren \cite[2.4,
2.6\,(d)]{MR0225243} and T.\ De Pauw and R.\ Hardt \cite[3.6]{MR2876138} which
is not developed in those works.

In comparison to W.\ Fleming \cite{MR185084} and T.\ De Pauw and R.\ Hardt
\cite{MR2876138,MR3206697}, five distinctive features of our approach may be
summarised as follows: Firstly, all our classes of chains are based on
\emph{local chains}.  Secondly, we freely use algebraic properties of
commutative groups and topological properties of normed commutative groups
following N.\ Bourbaki.  Thirdly, the \emph{closure theorem}---or, the
boundary rectifiability theorem in the terminology of L.\ Simon---in the
context of integer coefficients (i.e., for rectifiable currents) plays a
central role in our construction of integral chains with coefficients in a
general complete normed commutative group.  Fourthly, we construct the chain
complex of \emph{simple locally integral $G$ chains} as starting point for a
closure procedure---based on pairs of locally rectifiable $G$ chains---leading
to locally integral $G$ chains.  By a simple locally integral $G$ chain, we
mean a locally rectifiable $G$ chain which is expressible as finite sum of
products of locally integral $\mathbf Z$ chains (isomorphically, locally
integral currents) with elements of $G$.  Our choice is motivated by the
favourable closedness properties of this chain complex under restriction, push
forward, and slicing; traditionally, polyhedral chains or Lipschitz chains
serve as starting point to construct flat $G$ chains by means of completion.
Fifthly, we indicate how the concepts of locally integral $G$ chain and
locally rectifiable $G$ chain can be used to construct the chain complex of
locally flat $G$ chains by taking a suitable quotient---again based on pairs
of locally rectifiable $G$ chains; previous approaches firstly define the
chain complex of flat $G$ chains and obtain integral $G$ chains as a
subcomplex.  Thus, in our proposed treatment, locally rectifiable $G$ chains
are central for both constructions: that of locally integral $G$ chains and
that of locally flat $G$ chains.

The group of rectifiable $G$ chains as defined by T.\ De Pauw and R.\ Hardt in
\cite[3.6]{MR2876138} is isomorphic to ours of locally rectifiable $G$ chain
with finite mass, see \ref{miniremark:def-rectifiable-G-chains}.  For $G =
\mathbf R$, our concepts are isomorphic to those of H.\ Federer
\cite{MR833403}, see \ref{example:real-coefficients} and
\ref{example:locally-flat-real-chains}.  If $G = \mathbf Z / d \mathbf Z$,
then our integral $G$ chains are isomorphic to integral chains modulo $d$ as
defined by H.\ Federer in \cite[4.2.26]{MR41:1976}, see
\ref{example:integers-mod-lambda}.

\subsection{Outline by section}

\subsubsection*{Preliminaries}

In this section, we gather six strings of preparations: Firstly, we summarise
basic properties of normed commutative groups in
\ref{def:normed-commutative-group}--\ref{remark:summation}; in particular, we
recall that it may happen that $G$ is isomorphic to $\mathbf Z$ as commutative
group but not so as normed commutative group, see \ref{example:torus-subgroup}.
Secondly, we make some measure-theoretic preparations in
\ref{def:restriction}--\ref{remark:push_forward_borel_regular}.  Thirdly, we
construct two auxiliary functions of class $\infty$ in
\ref{thm:prescribed-zero-set}--\ref{corollary:separation-smooth-Tietze}.  The
second string and the third string will be employed throughout this series of
papers, see \ref{remark:measure-preparations} and
\ref{remark:auxiliary-functions}.  Fourthly, in analogy to
\begin{equation*}
	\mathscr F_m (U) = \{ Q + \boundary R \with Q \in \mathscr R_m (U), R
	\in \mathscr R_{m+1} (U) \},
\end{equation*}
we represent $\mathbf F_m^{\textup{loc}} ( \mathbf R^n )$ as quotient vector
space in
\ref{thm:structure-flat-chains}--\ref{remark:classical-real-flat-chains}.
Fifthly, in \ref{lemma:cutting-locally-integral-flat-chains}%
--\ref{remark:Alberti-Marchese-flat}, we similarly exhibit $\mathscr
F_m^{\textup{loc}} ( U )$ as quotient commutative group.  Sixthly, we study
weights of rectifiable $m$ varifolds in
\ref{thm:rectifiability}--\ref{remark:product-rect-var}; this includes the
area formula---including a version for $G$ valued functions---in
\ref{thm:push-forward-rectifiable-varifolds} and
\ref{corollary:push-forward-rectifiable-varifolds}, the coarea formula in
\ref{thm:coarea-rectifiable-varifolds}, and the Cartesian product in
\ref{thm:product-rect-var}.

\subsubsection*{Rectifiable chains}

We begin by defining and studying the complete normed commutative group bundle
$\mathbf G (n,m,G)$ over $\mathbf G(n,m)$ with fibre $G$ in
\ref{miniremark:normed-group-bundle}%
--\ref{miniremark:product-grassmann-bundle}.  This allows to introduce the
necessary operations for locally rectifiable $G$ chains---addition,
right-multiplication on $\mathbf G(n,m,\mathbf Z)$ with members of $G$, push
forward, slicing, and Cartesian product---firstly on the level of the bundle.
Then, the complete normed commutative group
\begin{equation*}
	\mathscr R_m^{\textup{loc}} ( U, G)
\end{equation*}
of $m$ dimensional \emph{locally rectifiable $G$ chains} in $U$ is defined
using equivalence classes of certain $\mathscr H^m$ measurable $\mathbf
G(n,m,G)$ valued functions in \ref{miniremark:def-rectifiable-G-chains}.  To
each $S$ in $\mathscr R_m^{\textup{loc}} (U,G)$ correspond the weight $\| S
\|$ of an $m$ dimensional rectifiable varifold in $U$ and a representing
function (i.e., a member of the equivalence class $S$)
\begin{equation*}
	\vect S;
\end{equation*}
the role of the latter is analogous to the product $\boldsymbol \Uptheta^m (
\| Q \|, \cdot ) \vec Q$ for $Q \in \mathscr R_m^{\textup{loc}} ( U )$.
Whenever $\| S \|$ is absolutely continuous with respect to the weight $\phi$
of some $m$ dimensional rectifiable varifold in $U$, there exists a
representing function of $S$ which is \emph{adapted} to $\phi$, see
\ref{miniremark:def-rectifiable-G-chains}; for instance, $\vect S$ is adapted
to $\| S \|$.  This concept allows to combine the results on the bundle
$\mathbf G(n,m,G)$ with those on rectifiable varifolds to study the
afore-mentioned operations on $\mathscr R_m^{\textup{loc}} (U,G)$ in
\ref{miniremark:push-forward-G-chain}--\ref{miniremark:slicing-G-chains}.

\subsubsection*{Integral chains}

We construct the complete normed commutative group $\mathbf I_m^{\textup{loc}}
( U, G )$ of $m$ dimensional locally integral $G$ chains in $U$ and the
corresponding boundary operator $\boundary_G$ in six steps.

\begin{step} [integer coefficients] \label{step:integer-coefficients}
	To define the subgroup $\mathbf I_m^{\textup{loc}} ( U, \mathbf Z )$
	of $\mathscr R_m^{\textup{loc}} ( U,\mathbf Z )$ and the boundary
	operator $\boundary_{\mathbf Z}$ corresponding to $\mathbf
	I_m^{\textup{loc}} ( U, \mathbf Z )$, we employ the canonical
	isomorphism of commutative groups $\mathscr R_m^{\textup{loc}} (U)
	\simeq \mathscr R_m^{\textup{loc}} ( U, \mathbf Z )$ in
	\ref{miniremark:isomorphism-for-integers}.
\end{step}

\begin{step} [algebra lemma] \label{step:algebra-lemma}
	For homomorphisms $i : A \to B$ of commutative groups, we establish an
	equivalent condition to $i \otimes \mathbf 1_H$ being univalent (i.e.,
	injective) for every commutative group $H$ in \ref{lemma:tensor-G};
	this is accomplished by expressing $H$ as inductive limit of its
	finitely generated subgroups and employing the structure theorem for
	finitely generated commutative groups.  We also recall that---contrary
	to the category of vector spaces---univalentness of $i$ does not imply
	the same for the homomorphism $i \otimes \mathbf 1_H$, see
	\ref{remark:tensor-G}.
\end{step}

\begin{step} [application of the closure theorem for rectifiable currents]
	\label{step:closure-theorem}
	We verify the condition obtained in Step \ref{step:algebra-lemma} for
	the inclusion map $i$ of $\mathbf I_m^{\textup{loc}} ( U, \mathbf Z )$
	into $\mathscr R_m^{\textup{loc}} ( U, \mathbf Z )$ by means of the
	\emph{closure theorem} for rectifiable currents (see
	\cite[4.2.16\,(2)]{MR41:1976} or \cite[30.3]{MR756417}) in
	\ref{example:monomorphisms}.  This amounts to verifying, for every
	positive integer $d$, we have $Q \in \mathbf I_m^{\textup{loc}} ( U)$
	whenever $Q \in \mathscr R_m^{\textup{loc}} ( U )$ and $dQ \in \mathbf
	I_m^{\textup{loc}} (U)$.
\end{step}

\begin{step} [simple locally rectifiable $G$ chains]
	\label{step:simple-rectifiable-chains}
	We use the canonical multiplication $\cdot : \mathscr
	R_m^{\textup{loc}} (U,\mathbf Z ) \times G \to \mathscr
	R_m^{\textup{loc}} (U,G)$, to obtain the induced homomorphism
	\begin{equation*}
		\rho_{U,m,G} : \mathscr R_m^{\textup{loc}} ( U, \mathbf Z )
		\otimes G \to \mathscr R_m^{\textup{loc}} ( U, G ),
	\end{equation*}
	whose image is dense in $\mathscr R_m^{\textup{loc}} (U,G)$, in
	\ref{miniremark:bilinear}.  Based on the structure theorem for
	finitely generated commutative groups, we next prove that
	\begin{equation*}
		\text{$\rho_{U,m,G}$ is univalent}
	\end{equation*}
	in \ref{thm:rho-mono}.  In case $G$ is finite, we deduce $\mathscr
	R_m^{\textup{loc}} (U, \mathbf Z ) \otimes G \simeq \mathscr
	R_m^{\textup{loc}} ( U, G )$, see
	\ref{remark:isomorphisms-finite-group}.
\end{step}

\begin{step} [simple locally integral $G$ chains]
	\label{step:simple-integral-chains}
	By Steps \ref{step:closure-theorem} and
	\ref{step:simple-rectifiable-chains}, the composition of homomorphisms
	\begin{equation*}
		\begin{xy}
			\xymatrix{
				\mathbf I_m^{\textup{loc}} (U,\mathbf Z)
				\otimes G \ar[rr]^{\text{$i \otimes \mathbf
				1_G$}} && \mathscr R_m^{\textup{loc}}
				(U,\mathbf Z) \otimes G \ar[rr]^{\rho_{U,m,G}}
				&& \mathscr R_m^{\textup{loc}} (U,G), }
		\end{xy}
	\end{equation*}
	where $i : \mathbf I_m^{\textup{loc}} ( U, \mathbf Z ) \to \mathscr
	R_m^{\textup{loc}} ( U, \mathbf Z )$ is the inclusion, is univalent;
	its image consists, by definition, of all $m$ dimensional \emph{simple
	locally integral $G$ chains} in $U$.  By univalentness, the boundary
	operator $\boundary_G$ of this chain complex may be defined by
	\begin{equation*}
		\boundary_G ( S \cdot g ) = ( \boundary_{\mathbf Z} S ) \cdot
		g \quad \text{for $S \in \mathbf I_m^{\textup{loc}} ( U,
		\mathbf Z )$ and $g \in G$}
	\end{equation*}
	in \ref{corollary:boundary-operator-dense-subset}; this is in
	accordance with the previous definition in case $G = \mathbf Z$.
\end{step}

\begin{step} [closure operation] \label{step:closure-operation}
	We let $\mathbf I_0^{\textup{loc}} (U,G) = \mathscr R_0^{\textup{loc}}
	(U,G)$.  Whenever $m \geq 1$, the complete normed commutative group
	$\mathbf I_m^{\textup{loc}} ( U, G )$ is defined in
	\ref{definition:loc-integral-chain} as closure of the subgroup of
	\begin{equation*}
		\mathscr R_m^{\textup{loc}} (U,G) \times \mathscr
		R_{m-1}^{\textup{loc}} (U,G)
	\end{equation*}
	consisting of all pairs $(S,\boundary_G S)$ corresponding to $m$
	dimensional simple locally integral $G$ chains $S$ in $U$; the
	boundary operator
	\begin{equation*}
		\boundary_G : \mathbf I_m^{\textup{loc}} (U,G) \to \mathbf
		I_{m-1}^{\textup{loc}} (U,G)
	\end{equation*}
	is then induced by the shift operator mapping $(S,T) \in \mathscr
	R_m^{\textup{loc}} (U,G) \times \mathscr R_{m-1}^{\textup{loc}} (U,G)$
	onto $(T,0)$ if $m \geq 2$ and onto $T$ if $m=1$.  Clearly,
	$\boundary_G$ is continuous.  We then show in
	\ref{thm:integral-chains} that the canonical projection of $\mathscr
	R_m^{\textup{loc}} ( U,G ) \times \mathscr R_{m-1}^{\textup{loc}}
	(U,G)$ onto its first factor, restricted to $\mathbf
	I_m^{\textup{loc}} (U,G)$, is univalent; this allows us to
	subsequently
	\begin{quote}
		\emph{identify $\mathbf I_m^{\textup{loc}} (U,G)$ with a dense
		subgroup of $\mathscr R_m^{\textup{loc}} (U,G)$ so that
		$\boundary_G$ extends the boundary operator on simple locally
		integral $G$ chains}.
	\end{quote}
	In this process, establishing that we have $T=0$ whenever $(0,T) \in
	\mathbf I_m^{\textup{loc}} (U,G)$ is ultimately reduced to the case
	that $U = \mathbf R^n$ and that $\spt \| T \|$ is a compact subset of
	an $m-1$ dimensional vector subspace.
\end{step}

During Steps \ref{step:integer-coefficients}--\ref{step:closure-operation}, we
keep track in \ref{miniremark:isomorphism-for-integers},
\ref{miniremark:bilinear}, and \ref{remark:boundary-dense-subset} of how the
operations push forward, Cartesian product, and slicing on rectifiable $G$
chains interact with the intermediately constructed boundary operators.  This
is crucial: firstly, in the identification of $\mathbf I_m^{\textup{loc}} (
U,G )$ with a subgroup of $\mathscr R_m^{\textup{loc}} (U,G)$ carried out in
Step \ref{step:closure-operation}, whence the properties of these operations
for $\mathbf I_m^{\textup{loc}} (U,G)$ in \ref{thm:integral-chains} and the
homotopy formula in \ref{corollary:integral-chains:homotopy-formula} follow,
and secondly in proving in \ref{thm:restriction-homomorphism} that the
restriction operators $r_m : \mathscr R_m^{\textup{loc}} (V,G) \to \mathscr
R_m^{\textup{loc}} (U,G)$ satisfy
\begin{equation*}
	r_m \big [ \mathbf I_m^{\textup{loc}} (V,G) \big ] \subset \mathbf
	I_m^{\textup{loc}} (U,G)
\end{equation*}
and commute with $\boundary_G$, whenever $U \subset V \subset \mathbf R^n$ and
$V$ is open.

To compare with classical examples, we define (see
\ref{miniremark:def-rectifiable-G-chains} and \ref{definition:integral-chain})
the subgroups
\begin{align*}
	\mathscr R_m ( U,G ) & = \mathscr R_m^{\textup{loc}} (U,G) \cap \{ S
	\with \text{$\spt \| S \|$ is compact} \}, \\
	\mathbf I_m ( U,G ) & = \mathbf I_m^{\textup{loc}} (U,G) \cap \{ S
	\with \text{$\spt \| S \|$ is compact} \}.
\end{align*}
Moreover, Steps \ref{step:integer-coefficients}, \ref{step:algebra-lemma}, and
\ref{step:simple-rectifiable-chains} allow to define the subgroup $\mathscr
P_m (U,G)$ of $\mathscr R_m^{\textup{loc}} (U,G)$ consisting of all $m$
dimensional \emph{polyhedral $G$ chains} in $U$, see
\ref{miniremark:bilinear}.

\subsubsection*{Classical coefficient groups}

We compare our treatment of rectifiable and integral $G$ chains with that of
the classical cases $G = \mathbf R$ and $G = \mathbf Z/ d\mathbf Z$ in
\cite{MR833403} and \cite[4.2.26]{MR41:1976}, respectively.  In
\ref{example:real-coefficients}, we provide canonical isomorphisms---commuting
with the boundary operators, restriction, push forward, Cartesian product, and
slicing---showing that
\begin{equation*}
	\mathscr R_m^{\textup{loc}} (\mathbf R^n,\mathbf R) \simeq \mathbf
	F_m^{\textup{loc}} ( \mathbf R^n ) \cap \{ Q \with \text{$Q$ has
	positive densities} \}.
\end{equation*}
This isomorphism maps $\mathbf I_m^{\textup{loc}} ( \mathbf R^n, \mathbf R )
\cap \{ S \with \text{$S$ is simple} \}$ onto the
\begin{equation*}
	\text{real linear span of $\mathbf I_m^{\textup{loc}} ( \mathbf R^n )$
	in $\mathbf F_m^{\textup{loc}} ( \mathbf R^n )$}.
\end{equation*}
For $m \geq 1$, we determine the closure of these groups to obtain
\begin{equation*}
	\mathbf I_m^{\textup{loc}} (\mathbf R^n,\mathbf R) \simeq \mathbf
	F_m^{\textup{loc}} ( \mathbf R^n ) \cap \{ Q \with \text{$Q$ and
	$\boundary Q$ have positive densities} \};
\end{equation*}
this is based on the deformation theorem for members of the group on the right
from \cite{MR833403} and the resulting approximation theorem by push forwards
of $m$ dimensional real polyhedral chains in $\mathbf R^n$ by diffeomorphisms
of class $1$ of $\mathbf R^n$.

Similarly, relying on the approximation theorem for integral chains modulo $d$
from \cite[4.2.26]{MR41:1976}, we construct canonical isomorphisms
\begin{gather*}
	\mathscr R_m ( \mathbf R^n, \mathbf Z/d\mathbf Z ) \simeq \mathscr
	R_m^d ( \mathbf R^n ), \quad \mathbf I_m ( \mathbf R^n, \mathbf
	Z/d\mathbf Z ) \simeq \mathbf I_m^d ( \mathbf R^n ), \\
	\mathbf I_m ( \mathbf R^n, \mathbf Z/d\mathbf Z ) \cap \{ S \with
	\text{$S$ is simple} \} \simeq \big \{ ( Q )^d \with Q \in \mathbf I_m
	( \mathbf R^n ) \big \}
\end{gather*}
in \ref{example:integers-mod-lambda};  in particular, R.\ Young's structural
result
\begin{equation*}
	\mathbf I_m^d (\mathbf R^n) = \{ (Q)^d \with Q \in \mathbf I_m
	(\mathbf R^n) \}
\end{equation*}
in \cite[Corollary 1.5]{MR3743699} may be restated by saying that every $S \in
\mathbf I_m ( \mathbf R^n, \mathbf Z/d\mathbf Z)$ is simple.  To clarify the
literature regarding the impossibility of an analogous structural result for
$\mathbf I_{m,K}^d ( \mathbf R^n )$ for general compact subsets $K$ of
$\mathbf R^n$, we include in \ref{remark:Federer-correction} an unpublished
correction listed by H.\ Federer.

\subsubsection*{Constancy theorem}

To construct an example of a one-dimensional indecomposable integral $G$ chain
whose associated rectifiable varifold is decomposable in the second paper of
our series (see \cite{arXiv:2209.05955v1}), we provide a constancy theorem for
$m$ dimensional locally integral $G$ chains whose boundary lies outside of an
$m$ dimensional connected orientable submanifold $M$ of class $1$ of $U$ in
\ref{thm:constancy-theorem}.  In the model case that $m = n$ and
\begin{equation*}
	M = \mathbf R^m \cap \{ x \with \text{$a_i < x_i < b_i$ for $i = 1,
	\ldots, m$} \},
\end{equation*}
where $- \infty < a_i < b_i < \infty$ for $i = 1, \ldots, m$, the constancy
theorem yields that $T$ in $\mathbf I_m ( \mathbf R^m, G )$, satisfying $\spt
\| \boundary_G T \| \subset \Bdry M$, equals $Q \cdot g$, for some $g \in G$,
where $Q \in \mathbf I_m ( \mathbf R^m, \mathbf Z )$ corresponds to $(
\mathscr L^m \restrict M ) \wedge \mathbf e_1 \wedge \cdots \wedge \mathbf e_m
\in \mathbf I_m ( \mathbf R^m)$.  We prove the general case and the model case
by simultaneous induction on $m$.  This is mostly based on H.\ Federer's
arguments for the classical coefficient groups in \cite[4.1.31\,(2),
4.2.3]{MR41:1976}---see also \cite[4.2.26, p.\ 432]{MR41:1976}---which we
adapt and merge by means of our restriction operators $r_m$, see
\ref{remark:constacy-theorem-model} and
\ref{remark:constancy-theorem-discussion}.

\subsubsection*{Flat chains}

To conclude the development of the present paper, we indicate in
\ref{miniremark:locally-flat-G-chains} how to extend our approach to include a
chain complex of locally flat $G$ chains.  Namely, we construct complete
normed commutative groups $\mathscr F_m^{\textup{loc}} (U,G)$ together with
continuous boundary operators $\boundary_G$ such that $\mathscr P_m (U,G)$ is
dense in $\mathscr F_m^{\textup{loc}} (U,G)$ and
\begin{equation*}
	\mathscr F_m^{\textup{loc}} (U,G) = \big \{ S + \boundary_G T \with S
	\in \mathscr R_m^{\textup{loc}} (U,G), T \in \mathscr
	R_{m+1}^{\textup{loc}} (U,G) \big \}.
\end{equation*}
In fact, $\mathscr F_m^{\textup{loc}} (U,G)$ is defined to be the quotient
$\big ( \mathscr R_m^{\textup{loc}} (U,G) \times \mathscr
R_{m+1}^{\textup{loc}} (U,G) \big ) \big / H_m$, where the \emph{closed}
subgroup $H_m$ of $\mathscr R_m^{\textup{loc}} (U,G) \times \mathscr
R_{m+1}^{\textup{loc}} (U,G)$ is given by
\begin{equation*}
	H_m = \big ( \mathbf I_m^{\textup{loc}} (U,G) \times \mathbf
	I_{m+1}^{\textup{loc}} (U,G) \big ) \cap \{ (S,T) \with S +
	\boundary_G T = 0 \}.
\end{equation*}
Finally, we obtain (in \ref{example:locally-integral-flat-chains} and
\ref{example:locally-flat-real-chains}) canonical isomorphisms
\begin{equation*}
	\mathscr F_m^{\textup{loc}} ( U, \mathbf Z ) \simeq \mathscr
	F_m^{\textup{loc}} (U), \quad \mathscr F_m^{\textup{loc}} ( \mathbf
	R^n, \mathbf R ) \simeq \mathbf F_m^{\textup{loc}} ( \mathbf R^n );
\end{equation*}
these are based on the representations of $\mathscr F_m^{\textup{loc}} (U)$
and $\mathbf F_m^{\textup{loc}} ( \mathbf R^n )$ recorded earlier.

\subsection{Remarks}

\paragraph{Constancy theorem}  The case that $M$, for some cubical subdivision
of $\mathbf R^n$, equals the $m$ skeleton minus the $m-1$ skeleton (e.g., $M =
\mathbf W_m' \without \mathbf W_{m-1}'$) is a basic ingredient for deformation
theorems.  For general $G$ and such $M$, a constancy theorem was first
formulated by W.\ Fleming in \cite[(7.2)]{MR185084} for compactly supported
normal $G$ chains; however, similar to H.\ Federer for classical coefficient
groups in \cite[4.2.3]{MR41:1976} and T.\ De Pauw and R.\ Hardt for general
$G$ in \cite[6.3]{MR3206697}, we avoid unspecific references to topology in
our argument.

\paragraph{Possible continuation of these notes}  Besides extending the
various operations studied for $\mathscr R_m^{\textup{loc}} (U,G)$ to
$\mathscr F_m^{\textup{loc}} (U,G)$, one would surely intend to add notions
for a Borel regular measure $\| S \|$ and for the support of $S$ associated
with $S$ in $\mathscr F_m^{\textup{loc}} (U,G)$ and to obtain suitable
deformation theorems.  In this regard, we would expect the representation
\begin{equation*}
	\mathscr F_m^{\textup{loc}} (U,G) = \big \{ S + \boundary_G T \with S
	\in \mathscr R_m^{\textup{loc}} (U,G), T \in \mathscr
	R_{m+1}^{\textup{loc}} (U,G) \big \}
\end{equation*}
to be particularly expedient.  For $G = \mathbf Z$ or $G = \mathbf R$, the
deformation theorems of \cite[4.1.9]{MR41:1976} and \cite[\S\,4]{MR833403},
respectively, form the key ingredients in proving the above isomorphisms for
$\mathscr F_m^{\textup{loc}} (\mathbf R^n, G)$ which, for these $G$, yield an
affirmative answer to the following question.  Assuming $m \geq 1$,
\begin{quote}
	\emph{is $\mathbf I_m^{\textup{loc}} (\mathbf R^n, G)$ equal to
	$\mathscr R_m^{\textup{loc}} (\mathbf R^n, G) \cap \big \{ S \with
	\boundary_G S \in \mathscr R_{m-1}^{\textup{loc}} ( \mathbf R^n, G)
	\big \}$?}
\end{quote}
If successful, these extensions would yield a geometric approach to $\mathscr
F_m^{\textup{loc}} (U,G)$ with $\mathscr R_m^{\textup{loc}} (U,G)$ and the
complete normed commutative group bundle $\mathbf G(n,m,G)$ taking the centre
stage instead of functional analytic completion procedures.

\paragraph{Background}  A notion of flat $G$ chains in Euclidean space was
first introduced by W.\ Fleming in \cite{MR185084}; for the special case $G =
\mathbf Z/d\mathbf Z$, H.\ Federer provided an alternative approach to W.\
Fleming's theory in \cite[4.2.26]{MR41:1976}.  Returning to general $G$, the
first six sections of \cite{MR185084} form the foundation for B.\ White's
improved deformation theorem and his subsequent rectifiability theorem of flat
chains in \cite{MR1738045} and \cite{MR1715323}.  These developments are
comprised in \cite{MR2876138} where T.\ De Pauw and R.\ Hardt extended them to
general metric spaces.

\paragraph{Development of these notes}  Originally, we intended to draw from
the most general and self-contained account \cite{MR2876138} of T.\ De Pauw
and R.\ Hardt for our applications in the third paper of our series (see
\cite{MenneScharrer2-3}). The present notes then grew out of an attempt to
provide to the reader---with the due simplifications entailed by the Euclidean
setting---the relevant definitions from \cite{MR2876138}.  Focusing on local
chains and employing the bundle $\mathbf G(n,m,G)$ appeared to be natural
choices in approaching rectifiable chains in this context.  Defining the
relevant operations on $\mathscr R_m^{\textup{loc}} ( U,G )$ then also
entailed the inclusion of some seemingly well-known but hard-to-cite
properties of rectifiable varifolds.  Next, the ambition to provide a direct
route to locally integral $G$ chains---without prior construction of $\mathscr
F_m^{\textup{loc}}(U,G)$ by completion---raised the question whether (see
Steps \ref{step:algebra-lemma}--\ref{step:simple-integral-chains}) certain
canonical homomorphisms were univalent and whether (see Step
\ref{step:closure-operation}) the group $\mathbf I_m^{\textup{loc}} (U,G)$
constructed could in fact be identified with a subgroup of $\mathscr
R_m^{\textup{loc}} (U,G)$.  The proof of consistency with previous work on
chains with classical coefficient groups and the related correction of
\cite[4.2.26]{MR41:1976} dutifully followed.  The simplification of an example
in the second paper of our series (see \cite{arXiv:2209.05955v1}) then gave rise
to adding the constancy theorem in the submanifold setting; thereby, the
adaptation of the existing proof strategies to our context led to the study of
the restriction operators.  Finally---with our primary goal obtained---, we
realised that our approach could be extended to yield a viable definition for
$\mathscr F_m^{\textup{loc}} (U,G)$.  Thus, we decided to indicate this
direction which entailed documenting the seemingly well-known but hard-to-cite
representations of $\mathscr F_m^{\textup{loc}} (U)$ and $\mathbf
F_m^{\textup{loc}} ( \mathbf R^n )$.

\paragraph{W.\ Fleming's approach}  As a possible alternative to constructing
the afore-mentioned direct route, we also considered to simply draw from W.\
Fleming's original theory in \cite{MR185084} which is formulated in the
Euclidean setting.  However, studying the first four sections thereof, we
found that some parts of the treatment required to be formalised, expanded,
and---at times---corrected to become entirely satisfactory. We accordingly
deemed it advisable to avoid just referring the reader to \cite{MR185084} for
proofs.  Nonetheless, we have been inspired by W.\ Fleming's work---for
instance, regarding how to identify $\mathbf I_m^{\textup{loc}} ( U,G)$ with a
subgroup of $\mathscr R_m^{\textup{loc}} ( U,G )$, see
\ref{remark:Fleming-inspiration}---and we believe that our notes could in fact
be partially of assistance to readers intending to study the paper
\cite{MR185084}.

\paragraph{H.\ Federer's approach}  H.\ Federer's treatment of the case $G =
\mathbf Z/d\mathbf Z$ and ours of general $G$ share the essential role of the
case of integer coefficients.  Noting $d \mathbf I_m ( \mathbf R^n ) \subset d
\mathscr F_m ( \mathbf R^n ) \subset \mathscr F_m ( \mathbf R^n ) \cap \{ T
\with T \equiv 0 \mod d \}$, his quotient approach to flat chains modulo $d$
leads to the following commutative diagram.
\begin{equation*}
	\begin{xy}
		\xymatrix{
			\mathbf I_m ( \mathbf R^n) / d \mathbf I_m ( \mathbf
			R^n ) \ar[r] \ar[d]_-{\textup{univalent}} & \mathscr
			R_m ( \mathbf R^n) / d \mathscr R_m ( \mathbf R^n )
			\ar[r] \ar[d]^-{\simeq} & \mathscr F_m ( \mathbf R^n )
			/ d \mathscr F_m ( \mathbf R^n )
			\ar[d]^-{\textup{onto}} \\
			\mathbf I_m^d ( \mathbf R^n ) \ar[r]^-{\subset} &
			\mathscr R_m^d ( \mathbf R^n ) \ar[r]^-{\subset} &
			\mathscr F_m^d ( \mathbf R^n) }
	\end{xy}
\end{equation*}
The two horizontal arrows in the top row are univalent by the closure theorem
(see \cite[4.2.16\,(2)\,(3)]{MR41:1976}); the two horizontal arrows in the
bottom row are inclusions; the middle vertical arrow is an isomorphism by
\cite[4.2.26, p.\,430]{MR41:1976} which corresponds to our isomorphism
$\rho_{\mathbf R^n,m,\mathbf Z/d\mathbf Z}$;%
\begin{footnote}%
	{Correspondence refers to the commutative diagram below, where $i :
	\mathscr R_m ( \mathbf R^n ) \to \mathscr R_m^{\textup{loc}} ( \mathbf
	R^n)$ is the inclusion, and $\iota_{\mathbf R^n,m}$ and $\mu_{\mathbf
	R^n,m,d}$ are the isomorphisms of
	\ref{miniremark:isomorphism-for-integers} and
	\ref{example:integers-mod-lambda}, respectively.
	\begin{equation*}
		\begin{xy}
			\xymatrix{
				\mathscr R_m ( \mathbf R^n ) / d \mathscr R_m
				( \mathbf R^n ) \ar[dd]^-{\simeq}
				\ar[r]_-{\simeq} & \mathscr R_m (\mathbf R^n)
				\otimes ( \mathbf Z / d \mathbf Z )
				\ar[rr]_-{\textup{univalent}}^-{i \otimes
				\mathbf 1_{\mathbf Z/ d\mathbf Z}} && \mathscr
				R_m^{\textup{loc}} ( \mathbf R^n ) \otimes (
				\mathbf Z /d \mathbf Z )
				\ar[d]_-{\simeq}^-{\iota_{\mathbf R^n,m}
				\otimes \mathbf 1_{\mathbf Z /d \mathbf Z}} \\
				&&& \mathscr R_m^{\textup{loc}} ( \mathbf R^n,
				\mathbf Z) \otimes ( \mathbf Z/d\mathbf Z)
				\ar[d]_-{\simeq}^-{\rho_{\mathbf R^n,m,\mathbf
				Z/d\mathbf Z}} \\
				\mathscr R_m^d ( \mathbf R^n)
				\ar[r]_-{\simeq}^-{\mu_{\mathbf R^n,m,d}} &
				\mathscr R_m ( \mathbf R^n, \mathbf Z/d\mathbf
				Z ) \ar[rr]^-{\subset} && \mathscr
				R_m^{\textup{loc}} ( \mathbf R^n, \mathbf
				Z/d\mathbf Z)}
		\end{xy}
	\end{equation*}}
\end{footnote}%
hence, the left vertical arrow is univalent; and the right vertical arrow is
onto by definition of $\mathscr F_m^d (\mathbf R^n)$.

\paragraph{R.\ Young's structural results} In \cite[Corollary 1.6]{MR3743699},
R.\ Young then established that the left and right vertical arrows in the
preceding commutative diagram are isomorphisms.  As the isomorphisms $A
\otimes ( \mathbf Z/d\mathbf Z) \simeq A/dA$, corresponding to commutative
groups $A$, form a natural transformation, the following commutative
diagram---in which all horizontal arrows are univalent---results.
\begin{equation*}
	\begin{xy}
		\xymatrix{
			\mathbf I_m ( \mathbf R^n) \otimes (\mathbf Z/ d
			\mathbf Z ) \ar[r] \ar[d]^-{\simeq} & \mathscr R_m (
			\mathbf R^n) \otimes ( \mathbf Z/ d \mathbf Z ) \ar[r]
			\ar[d]^-{\simeq} & \mathscr F_m ( \mathbf R^n )
			\otimes ( \mathbf Z /d \mathbf Z ) \ar[d]^-{\simeq} \\
			\mathbf I_m^d ( \mathbf R^n ) \ar[r]^-{\subset} &
			\mathscr R_m^d ( \mathbf R^n ) \ar[r]^-{\subset} &
			\mathscr F_m^d ( \mathbf R^n) }
	\end{xy}
\end{equation*}

\subsection{Acknowledgements} U.M.\ is grateful to Professor Jaigyoung Choe
for making H.\ Federer's correction list available, to Professor Thierry De
Pauw for sharing his insights on flat $G$ chains, and to Professor Salvatore
Stuvard for pointing out \cite{MR3743699} and for discussing aspects of
\cite{MR3819529,MR3819529-corrigendum}.  Some preliminaries of this paper were
drafted while U.M.\ worked at the Max Planck Institutes for Gravitational
Physics (Albert Einstein Institute) and for Mathematics in the Sciences and
the Universities of Potsdam and Leipzig, whereas the bulk was written while
U.M.\ was supported in Taiwan (R.O.C.) by means of the grants with Nos.\ MOST
108-2115-M-003-016-MY3, MOST 110-2115-M-003-017 -, and MOST 111-2115-M-003-014
- by the National Science and Technology Center (formerly termed Ministry of
Science and Technology) and as Center Scientist at the National Center for
Theoretical Sciences.  Parts of this paper were written while C.S.\ was
supported by the EPSRC as part of the MASDOC DTC at the University of Warwick,
Grant No.\ EP/HO23364/1.

\section{Notation} \label{section:notation}

Our notation follows \cite{MR3528825}; thus, we are largely consistent with
H.\ Federer's terminology in geometric measure theory (see
\cite[pp.\,669--676]{MR41:1976}) and W.\ Allard's notation for varifolds (see
\cite{MR0307015}).  We mention two exceptions: Whenever $f$ is a relation, we
employ $f[A]$ to mean $\{ y \with \text{$(x,y) \in f$ for some $x \in A$} \}$
and, whenever $T$ is an $m$ dimensional vector subspace of $\mathbf R^n$, the
canonical projection of $\mathbf R^n$ onto $T$ is denoted by $T_\natural$.
Additionally, following H.\ Federer \cite[p.\,414]{MR833403}, we say an $m$
dimensional locally flat chain $Q$ in $\mathbf R^n$ has \emph{positive
densities} if and only if $Q$ is representable by integration and $\boldsymbol
\Uptheta^{\ast m} ( \| Q \|, x ) > 0$ for $\| Q \|$ almost all $x$; by
\cite[\S\,1]{MR833403}, this concept yields the analogue for real coefficients
to that of $m$ dimensional locally rectifiable currents in $\mathbf R^n$ for
integer coefficients.

\section{Preliminaries}

\begin{definition} \label{def:normed-commutative-group}
	Suppose $G$ is a commutative group.
	
	Then, a function $\sigma : G \to \{ r \with 0 \leq r < \infty \}$ is
	termed a \emph{group norm} on $G$ if and only if $\sigma^{-1} [\{0\}]
	= \{0\}$ and, whenever $g,h \in G$, we have $\sigma (g) = \sigma (-g)$
	and $\sigma (g+h) \leq \sigma (g) + \sigma (h)$.  We associate with
	$\sigma$ the metric $\rho : G \times G \to \mathbf R$ on $G$, defined
	by $\rho (g,h)=\sigma(g-h)$ for $g,h \in G$ and often write $|g|$
	instead of $\sigma(g)$.
\end{definition}

\begin{remark} \label{remark:normed-group}
	Defining $s : G \times G \to G$ by $s(g,h) = g-h$ for $g,h \in G$, we
	see that $\Lip s \leq 1$ with respect to the metric on $G \times G$
	with value $\rho(g,g')+\rho(h,h')$ at $((g,h),(g',h')) \in ( G \times
	G)^2$.  In particular, $G$ is a topological group; it is complete as
	uniform space if and only if $\rho$ is complete.  If $H$ is a closed
	subgroup of $G$ and $p : G \to G/H$ is the quotient map, then $\dist (
	\cdot, H) \circ p^{-1}$ constitutes a group norm on $G/H$ which
	induces the quotient topology and $p$ is an open map; if $G$ is
	complete, so is $G/H$, and we have $\Lip f = \Lip ( f \circ p )$
	whenever $f$ maps $G/H$ into some metric space.  Whenever $H$ is
	another normed commutative group, we endow $G \times H$ with the group
	norm whose value at $(g,h) \in G \times H$ equals $|g| + |h| \in
	\mathbf R$.  Taking the standard group norm on $\mathbf Z$, the
	canonical bilinear map from $\mathbf Z \times G$ into $G$, mapping
	$(d,g) \in \mathbf Z \times G$ onto $d \cdot g \in G$, is Lipschitzian
	on bounded sets.
\end{remark}

\begin{example} \label{example:torus-subgroup}
	If $G = \mathbf R / \mathbf Z$, $r \in \mathbf R \without \mathbf Q$,
	and $H$ is the subgroup of $G$ generated by $\{ r + d \with d \in
	\mathbf Z \}$, then, $G$ is a complete normed group by
	\ref{remark:normed-group} and $H$ is infinite and thus dense in $G$ by
	\cite[Chapter 3, \S\,2.1, Proposition 1]{MR979294} and \cite[Chapter
	7, \S\,1.5, Corollary to Proposition 11]{MR1726872}; hence, $H$ is
	isomorphic to $\mathbf Z$ as commutative group but not so as normed
	commutative group, because $H$ has no isolated points.
\end{example}

\begin{definition}
	Suppose $G$ is a complete normed commutative group, $f$ is a function
	whose domain contains a set $A$ with values in $G$, and $\sum_{a \in
	A} |f(a)|<\infty$.
	
	Then, extending finite summation, we define the sum $\sum_A f$, also
	denoted by $\sum_{a \in A} f(a)$, in $G$ by requiring that, for
	$\epsilon > 0$, there exists a finite subset $C$ of $A$ such that
	$\big | \sum_A f - \sum_B f | \leq \epsilon$ whenever $B$ is finite
	and $C \subset B \subset A$.
\end{definition}

\begin{remark} \label{remark:summation}
	If $h : A \to Y$, then $\sum_A f = \sum_{y \in Y} \sum_{h^{-1}
	[\{y\}]} f$.
\end{remark}

\begin{definition} \label{def:restriction}
	Whenever $\phi$~measures~$X$ and $f$ is a~$\{ y \with 0 \leq y \leq
	\infty \}$~valued function whose domain contains $\phi$~almost all
	of~$X$, we define the measure $\phi \restrict f$ over~$X$ by
	\begin{equation*}
		{\textstyle ( \phi \restrict f ) (A) = \int^\ast_A f \ud \phi
		\quad \text{for $A \subset X$}}.
	\end{equation*}
\end{definition}

\begin{remark} \label{remark:restriction}
	Basic properties of this measure are listed
	in~\cite[2.4.10]{MR41:1976}.  Moreover, if $f$ is $\phi$ measurable,
	then
	\begin{equation*}
		(\phi \restrict f ) (A) = \inf \{ ( \phi \restrict f) (B)
		\with \text{$A \subset B$, $B$ is $\phi$ measurable} \} \quad
		\text{for $A \subset X$};
	\end{equation*}
	if additionally $X$ is a topological space, $\phi$ is Borel regular,
	and $\{ x \with f(x)>0 \}$ is $\phi$ almost equal to a Borel set, then
	$\phi \restrict f$ is Borel regular.
\end{remark}

\begin{remark} \label{remark:radon_measure_restriction}
	If $X$~is a locally compact Hausdorff space, $\phi$~is a Radon measure
	over~$X$, and $0 \leq f \in \mathbf L_1^\mathrm{loc} ( \phi )$, then
	$\phi \restrict f$ is a Radon measure over~$X$, provided $X$ is the
	union of a countable family of compact subsets of~$X$.  The
	supplementary hypothesis ``provided \ldots\ of~$X$'' may not be
	omitted; in fact, one may take $f$ to be the characteristic function
	of the set constructed in~\cite[9.41\,(e)]{MR0367121}.
\end{remark}

\begin{definition} \label{def:push_forward}
	Whenever $\phi$ measures~$X$, $Y$ is a topological space, and $f$ is a
	$Y$~valued function with $\dmn f \subset X$, we define the
	measure~$f_\# \phi$ over~$Y$ by
	\begin{equation*}
		f_\# \phi (B) = \phi ( f^{-1} [ B ] ) \quad \text{for $B
		\subset Y$}.
	\end{equation*}
\end{definition}

\begin{remark}
	This slightly extends~\cite[2.1.2]{MR41:1976}, where $\dmn f = X$ is
	required.
\end{remark}

\begin{lemma} \label{lemma:push_forward_borel_regular}
	Suppose $\phi$ is a Radon measure over a locally compact Hausdorff
	space~$X$, $Y$ is a separable metric space, $f$ is a $\phi$~measurable
	$Y$~valued function, and $X$ is $\phi$~almost equal to the union of a
	countable family of compact subsets of $X$.
	
	Then, $f_\# \phi$ is a Borel regular measure over~$Y$.
\end{lemma}

\begin{proof}
	By \cite[2.1.2]{MR41:1976}, all closed subsets of $Y$ are $f_\# \phi$
	measurable.  To prove the Borel regularity, we employ
	\cite[2.3.5]{MR41:1976} to reduce the problem to the case, $C = \spt
	\phi$~is compact and $f | C$~is continuous.  Then, supposing $B
	\subset Y$ and $\varepsilon > 0$, we employ \cite[2.2.5]{MR41:1976} to
	choose an open subset~$U$ of~$X$ with $f^{-1} [B] \subset U$ and $\phi
	(U) \leq \varepsilon + f_\#\phi (B)$, define an open subset~$V$ of~$Y$
	by $V = Y \without f [ C \without U ]$, and verify
	\begin{equation*}
		B \subset V, \quad f^{-1} [V] \subset U \cup ( X \without C ),
	\end{equation*}
	whence it follows $f_\# \phi (V) \leq \varepsilon + f_\# \phi (B)$.
\end{proof}

\begin{remark} \label{remark:push_forward_borel_regular}
	In the context of Radon measures and proper maps, a related statement
	is available from \cite[2.2.17]{MR41:1976}.
\end{remark}

\begin{remark} \label{remark:measure-preparations}
	Apart of \ref{thm:product-rect-var} and
	\ref{thm:push-forward-rectifiable-varifolds} below,
	\ref{remark:restriction}, \ref{remark:radon_measure_restriction}, or
	\ref{lemma:push_forward_borel_regular} will also be employed in
	\ref{miniremark:def-rectifiable-G-chains} of the present paper and
	three items of \cite{arXiv:2209.05955v1}. 
\end{remark}

\begin{theorem} \label{thm:prescribed-zero-set}
	Suppose $A$ is a closed subset of $\mathbf R^n$.

	Then, there exists a nonnegative function $f : \mathbf R^n \to \mathbf
	R$ of class $\infty$ such that
	\begin{gather*}
		A = \{ x \with f(x) = 0 \}, \quad \text{$\Der^i f(x) = 0$
		whenever $x \in A$ and $i$ a positive integer}, \\
		\text{and $\{ x \with f(x) \geq y \}$ is compact for $0 < y <
		\infty$}.
	\end{gather*}
\end{theorem}

\begin{proof}
	We abbreviate $U = \mathbf R^n \without A$, assume $U \neq
	\varnothing$, apply \cite[3.1.13]{MR41:1976} with $\Phi = \{ \mathbf
	R^n \without A \}$, arrange the elements of the resulting set $S$ in a
	univalent sequence $s_1, s_2, s_3, \ldots$, and, taking $\epsilon_i =
	\inf \{ 2^{-i}, \exp (-3/h(s_i))\} $, define $g : U \to \mathbf R$ by
	\begin{equation*}
		g(x) = \sum_{i=1}^\infty \epsilon_i v_{s_i} (x) \quad
		\text{for $x \in U$}.
	\end{equation*}
	For every positive integer $j$, we then estimate
	\begin{gather*}
		\| \Der^j g(x) \| \leq (129)^m V_j h(x)^{-j} \exp (-1/h(x))
		\quad \text{for $x \in U$}, \\
		{\textstyle \{ x \with g(x) \geq 2^{-j} \} \subset
		\bigcup_{i=1}^j \mathbf B(s_i,10h(s_i))}.
	\end{gather*}
	Therefore, we may take $f$ to be the extension of $g$ to $\mathbf R^n$
	by $0$.
\end{proof}

\begin{remark}
	A special case of the preceding theorem is employed in
	\cite[8.1\,(2)]{MR0307015} to demonstrate the sharpness of the
	regularity theorem in \cite[\S\,8]{MR0307015}.
\end{remark}

\begin{corollary} \label{corollary:separation-smooth-Tietze}
	Suppose $U$ is an open subset of $\mathbf R^n$ and $E_0$ and $E_1$ are
	disjoint relatively closed subsets of $U$.

	Then, there exists $f \in \mathscr E (U,\mathbf R)$ satisfying $E_i
	\subset \Int \{ x \with f(x)=i \}$ for $i \in \{ 0,1 \}$ and $0 \leq f
	\leq 1$.
\end{corollary}

\begin{proof}
	We choose, for $i \in \{ 0, 1 \}$, disjoint relatively closed sets
	$A_i$ with $E_i \subset \Int A_i$ and, by
	\ref{thm:prescribed-zero-set}, applied with $A$ replaced by $\mathbf
	R^n \without ( U \without A_i )$, also $g_i \in \mathscr E ( U,
	\mathbf R)$ satisfying $g_i \geq 0$ and $\{ x \with g_i (x) = 0 \} =
	A_i$, and take $f = g_0/(g_0+g_1)$.
\end{proof}

\begin{remark} \label{remark:auxiliary-functions}
	Apart of \ref{lemma:cutting-locally-integral-flat-chains} below,
	\ref{thm:prescribed-zero-set} or
	\ref{corollary:separation-smooth-Tietze} will also be employed in
	\ref{miniremark:bilinear}, \ref{thm:integral-chains}, and
	\ref{thm:restriction-homomorphism} of the present paper and three
	items of \cite{arXiv:2209.05955v1,MenneScharrer2-3}.
\end{remark}

\begin{theorem} \label{thm:structure-flat-chains}
	Suppose $m$ is a nonnegative integer, $n$ is a positive integer, $Z
	\in \mathbf F_m ( \mathbf R^n )$, $K$ is a compact subset of $\mathbf
	R^n$, and $\spt Z \subset \Int K$.

	Then, there exist $Q \in \mathbf F_{m,K} ( \mathbf R^n )$ and $R \in
	\mathbf F_{m+1,K} ( \mathbf R^n )$, both with positive densities, such
	that $Z = Q + \boundary R$.
\end{theorem}

\begin{proof}
	We choose compact subsets $B$ and $C$ of $\mathbf R^n$ with $\spt Z
	\subset \Int B$, $B \subset \Int C$, and $C \subset \Int K$ and notice
	that $Z \in \mathbf F_{m,B} ( \mathbf R^n )$ by
	\cite[4.1.12]{MR41:1976}.  Employing \cite[4.1.23]{MR41:1976} with $K$
	replaced by $B$, we construct $P_i \in \mathbf P_{m,C} ( \mathbf R^n)$
	satisfying
	\begin{equation*}
		\sum_{i=1}^\infty \mathbf F_C ( P_i ) < \infty \quad
		\text{and} \quad \sum_{i=1}^\infty P_i = Z.
	\end{equation*}
	Using \cite[4.2.23]{MR41:1976} with $V = \Int K$ and $X = P_i$, we
	pick $R_i \in \mathbf P_{m+1} ( \mathbf R^n )$ with $\spt R_i \subset
	K$ and
	\begin{equation*}
		\sum_{i=1}^\infty \big ( \mathbf M (P_i-\boundary R_i) +
		\mathbf M (R_i) \big ) < \infty.
	\end{equation*}
	We define $Q = \sum_{i=1}^\infty (P_i-\boundary R_i) \in \mathbf
	F_{m,K} ( \mathbf R^n )$ and $R = \sum_{i=1}^\infty R_i \in \mathbf
	F_{m+1,K} ( \mathbf R^n )$ with $Z = Q + \boundary R$ by means of
	$\mathbf F_K$ convergent series and \cite[4.1.12]{MR41:1976}.
	Finally, noting that $\mathbf R^n$ is countably $( \| Q \|, m)$
	rectifiable and countably $( \| R \|, m+1)$ rectifiable, $Q$ and $R$
	have positive densities by \cite[\S\,1\,(V)]{MR833403}.
\end{proof}

\begin{corollary} \label{corollary:structure-flat-chains}
	Suppose $m$ is a nonnegative integer, $n$ is a positive integer, and
	$Z \in \mathbf F_m^{\textup{loc}} ( \mathbf R^n )$.

	Then, there exist $Q \in \mathbf F_m^{\textup{loc}} ( \mathbf R^n )$
	and $R \in \mathbf F_{m+1}^{\textup{loc}} ( \mathbf R^n )$, both with
	positive densities, such that $Z = Q + \boundary R$.
\end{corollary}

\begin{proof}
	In view of \cite[4.1.12]{MR41:1976}, this follows from
	\ref{thm:structure-flat-chains} using a suitable partition of unity;
	for instance, one may apply \cite[3.1.13]{MR41:1976} with $\Phi = \{
	\mathbf R^n \}$ and \cite[3.1.12]{MR41:1976} with $U = \mathbf R^n$,
	$h(x) = \frac 1{20}$ for $x \in \mathbf R^n$, $\lambda = 0$, and
	$\alpha = \beta = 20$.
\end{proof}

\begin{remark} \label{remark:classical-real-flat-chains}
	Defining the linear map $L$ from $F_{n,m}$ onto $\mathbf
	F_m^{\textup{loc}} ( \mathbf R^n )$, where
	\begin{equation*}
		F_{n,m} = \big (\mathbf F_m^{\textup{loc}} ( \mathbf R^n )
		\times \mathbf F_{m+1}^{\textup{loc}} ( \mathbf R^n ) \big )
		\cap \{ (Q,R) \with \text{$Q$ and $R$ have positive densities}
		\},
	\end{equation*}
	by $L ( Q,R ) = Q + \boundary R$ for $(Q,R) \in F_{n,m}$, we obtain a
	vector space isomorphism
	\begin{equation*}
		F_{n,m} / \ker L \simeq \mathbf F_m^{\textup{loc}} ( \mathbf
		R^n ).
	\end{equation*}
\end{remark}

\begin{lemma} \label{lemma:cutting-locally-integral-flat-chains}
	Suppose $m$ is a nonnegative integer, $n$ is a positive integer, $U$
	is an open subset of $\mathbf R^n$, $A$ is a relatively closed subset
	of $U$, $S$ is an open subset of $U$, $Z \in \mathscr
	F_m^{\textup{loc}} ( U )$, $W \in \mathscr F_m ( U )$, and
	\begin{equation*}
		A \cap \spt Z \subset S, \quad A \cap \spt ( Z - W ) =
		\varnothing.
	\end{equation*}

	Then, there exist $Q \in \mathscr R_m ( U )$ and $R \in \mathscr
	R_{m+1} ( U )$ such that
	\begin{equation*}
		\spt Q \cup \spt R \subset S, \quad A \cap \spt ( Z - Q -
		\boundary R ) = \varnothing.
	\end{equation*}
\end{lemma}

\begin{proof}
	We pick $X \in \mathscr R_m ( U )$ and $Y \in \mathscr R_{m+1} ( U )$
	with $W = X + \boundary Y$, observe that we may choose an open subset
	$T$ of $U$ with
	\begin{equation*}
		A \subset T, \quad ( \Clos T ) \cap \spt Z \subset S, \quad T
		\cap \spt ( Z - W ) = \varnothing,
	\end{equation*}
	and obtain a locally Lipschitzian function $f : U \to \mathbf R$ with
	\begin{equation*}
		\text{$f(x) \leq 0$ for $x \in T \cap \spt Z$}, \quad
		\text{$f(x) \geq 1$ for $x \in U \without S$}
	\end{equation*}
	from \ref{corollary:separation-smooth-Tietze}.  Noting $T \cap \spt (
	X + \boundary Y ) \subset \spt Z$, we employ \cite[4.2.1, 4.3.1,
	4.3.4, 4.3.6]{MR41:1976} to select $0 < y < 1$ with $\langle Y,f,y
	\rangle \in \mathscr R_m ( U)$ and
	\begin{equation*}
		T \cap \spt \big ( X \restrict \{ x \with f(x) > y \} +
		\boundary ( Y \restrict \{ x \with f(x) > y \} ) - \langle
		Y,f,y \rangle \big ) \subset \spt Z.
	\end{equation*}
	Thus, we may take $Q = X \restrict \{ x \with f(x) \leq y \} + \langle
	Y,f,y \rangle$ and $R = Y \restrict \{ x \with f(x) \leq y \}$ because
	$T \cap \spt ( X+\boundary Y - Q - \boundary R ) \subset T \cap \{ x
	\with f(x) \geq y \} \cap \spt Z = \varnothing$.
\end{proof}

\begin{theorem} \label{thm:locally-integral-flat-chains}
	Suppose $m$ and $n$ are integers, $m \geq 0$, $n \geq 1$, $U$ is an
	open subset of $\mathbf R^n$, and $Z \in \mathscr F_m^{\textup{loc}} (
	U )$.

	Then, there exist $Q \in \mathscr R_m^{\textup{loc}} ( U )$ and $R \in
	\mathscr R_{m+1}^{\textup{loc}} ( U )$ such that $Z = Q + \boundary
	R$.
\end{theorem}

\begin{proof}
	Assume $U \neq \varnothing$.  Let $\Phi$ denote the class of all open
	subsets $T$ of $U$ such that, for some $W \in \mathscr F_m ( U )$, we
	have $T \cap \spt ( Z - W ) = \varnothing$; hence $\bigcup \Phi = U$.
	We define $h : U \to \{ r \with 0 < r < \infty \}$ by
	\begin{equation*}
		h(x) = {\textstyle\frac 1{20}} \sup \{ \inf \{ 1, \dist (x,
		\mathbf R^n \without T ) \} \with T \in \Phi \} \quad
		\text{for $x \in U$}.
	\end{equation*}
	Applying \cite[3.1.13]{MR41:1976}, we obtain a set $S$ such that,
	arranging its elements into an univalent sequence $s_1, s_2, s_3,
	\ldots$ in $U$, we have $U = \bigcup_{i=1}^\infty \mathbf B
	(s_i,5h(s_i))$ and
	\begin{equation*}
		\card \{ i \with \mathbf B (x,10h(x)) \cap \mathbf
		B(s_i,10h(s_i)) \neq \varnothing \} \leq (129)^n \quad
		\text{for $x \in U$}.
	\end{equation*}

	Next, we construct $R_1, R_2, R_3, \ldots$ in $\mathscr R_m ( U )$ and
	$Q_1, Q_2, Q_3, \ldots$ in $\mathscr R_{m+1} ( U )$ satisfying $\spt
	Q_i \cup \spt R_i \subset \mathbf B (s_i,10h(s_i))$ for every positive
	integer $i$ and
	\begin{equation*}
		K_i \cap \spt Z_i = \varnothing \quad \text{for every
		nonnegative integer $i$},
	\end{equation*}
	where we abbreviated $K_i = \bigcup_{j=1}^i \mathbf B ( s_j, 5
	h(s_j))$ and $Z_i = Z - \sum_{j=1}^i ( Q_j + \boundary R_j)$; in fact,
	this is trivial for $i = 0$ and, if $R_1, \ldots, R_{i-1}$ and $Q_1,
	\ldots, Q_{i-1}$ with these properties have been constructed for some
	positive integer $i$, then, noting that $K_i \cap \spt Z_{i-1} \subset
	\mathbf B(s_i,5h(s_i))$, we may take $Q_i$ and $R_i$ to be the
	currents furnished by applying
	\ref{lemma:cutting-locally-integral-flat-chains} with $A$, $S$ and $Z$
	replaced by $K_i \cap \mathbf B (s_i,10h(s_i))$, $\mathbf U
	(s_i,10h(s_i))$, and $Z_{i-1}$ because $( K_i \without \mathbf B
	(s_i,10h(s_i)) ) \cap \spt Z_i \subset K_{i-1} \cap \spt Z_{i-1} =
	\varnothing$.

	Let $Q = \sum_{i=1}^\infty Q_i$ and $R = \sum_{i=1}^\infty R_i$.  If
	$T$ is open and $\Clos T$ is a compact subset of $U$, then $T \subset
	K_i$ for some $i$, hence $T \cap \spt ( Z - Q - \boundary R ) =
	\varnothing$.
\end{proof}

\begin{remark} \label{remark:classical-integral-flat-chains}
	Defining the homomorphism $\eta$ from $\mathscr R_m^{\textup{loc}} (U)
	\times \mathscr R_{m+1}^{\textup{loc}} (U)$ onto $\mathscr
	F_m^{\textup{loc}} ( U )$ by $\eta ( Q,R ) = Q + \boundary R$ for
	$(Q,R) \in \mathscr R_m^{\textup{loc}} (U) \times \mathscr
	R_{m+1}^{\textup{loc}} (U)$, we obtain an isomorphism
	\begin{equation*}
		\big ( \mathscr R_m^{\textup{loc}} ( U ) \times \mathscr
		R_{m+1}^{\textup{loc}} (U ) \big ) \big / \ker \eta \simeq
		\mathscr F_m^{\textup{loc}} ( U).
	\end{equation*}
\end{remark}

\begin{remark} \label{remark:Alberti-Marchese-flat}
	For the flat $G$ chains of \cite{MR185084}, a representation analogous
	to those of \ref{thm:structure-flat-chains},
	\ref{corollary:structure-flat-chains}, and
	\ref{thm:locally-integral-flat-chains} was obtained in
	\cite[Proposition 2.1]{CVGMT-AlbMar22}.
\end{remark}

\begin{theorem} \label{thm:rectifiability}
	Suppose $\phi$ is a Radon measure over an open subset $X$ of $\mathbf
	R^n$, $m$ is a nonnegative integer, $X$ is countably $( \phi, m)$
	rectifiable, and
	\begin{equation*}
		\boldsymbol \Uptheta^{\ast m} ( \phi, a ) < \infty \quad
		\text{for $\phi$ almost all $a$}.
	\end{equation*}
	
	Then, $\phi$ is the weight of some member of $\mathbf{RV}_m ( X )$
	and, whenever $R$ is a compact $m$ rectifiable subset of $X$ and $f$
	maps a subset of $X$ into $\mathbf R^\nu$, we have that, for $\mathscr
	H^m$ almost all $a \in R$ with $\boldsymbol \Uptheta^m ( \phi, a ) >
	0$,
	\begin{gather*}
		\text{$\Tan^m ( \phi,a ) = \Tan^m ( \mathscr H^m \restrict R,
		a )$ is an $m$ dimensional vector space}, \\
		( \phi, m ) \ap \Der f (a) = ( \mathscr H^m \restrict R, m )
		\ap \Der f (a).
	\end{gather*}
\end{theorem}

\begin{proof}
	Whenever $R$ is a compact $m$~rectifiable subset of $X$, noting
	$\mathscr H^m (R) < \infty$, we infer $\phi ( R \cap \{ a \with
	\boldsymbol \Uptheta^m ( \phi, a ) = 0 \}) = 0$ from
	\cite[2.10.19\,(1)]{MR41:1976}.  It follows that
	\begin{equation*}
		\boldsymbol \Uptheta^{\ast m} ( \phi, a ) > 0 \quad \text{for
		$\phi$ almost all $a$}.
	\end{equation*}
	
	Next, suppose $f$ maps a subset of $X$ into $\mathbf R^\nu$, $R$ is a
	compact $m$ rectifiable subset of $X$, and the Borel sets $R_i$ are
	defined by
	\begin{equation*}
		R_i = R \cap \{ a \with 1/i < \boldsymbol \Uptheta^{\ast m} (
		\phi, a ) < i \}
	\end{equation*}
	whenever $i$ is a positive integer; hence, there holds
	\begin{equation*}
		\mathscr H^m \restrict R_i \leq \phi \restrict R_i \leq 2^m i
		\, \mathscr H^m \restrict R_i
	\end{equation*}
	by \cite[2.10.19\,(1)\,(3)]{MR41:1976}, and
	\begin{gather*}
		\boldsymbol \Uptheta^m ( \phi \restrict X \without R_i, a ) =
		0, \quad \boldsymbol \Uptheta^m ( \mathscr H^m \restrict R
		\without R_i, a ) = 0,
	\end{gather*}
	fo  $\mathscr H^m$ almost all $a \in R_i$
	by \cite[2.10.19\,(4)]{MR41:1976}, whence we infer
	\begin{align*}
		\Tan^m ( \phi, a ) & = \Tan^m ( \phi \restrict R_i, a) \\
		& = \Tan^m ( \mathscr H^m \restrict R_i, a ) = \Tan^m (
		\mathscr H^m \restrict R, a ) \in \mathbf G(n,m), \\
		( \phi, m ) \ap \Der f (a) & = ( \phi \restrict R_i, m ) \ap
		\Der f (a) \\
		& = ( \mathscr H^m \restrict R_i, m ) \ap \Der f(a) = (
		\mathscr H^m \restrict R, m ) \ap \Der f (a)
	\end{align*}
	for $\mathscr H^m$ almost all $a \in R_i$ by \cite[3.2.19]{MR41:1976}.
	The function mapping $\phi$ almost all $a \in R_i$ onto $\Tan^m (
	\phi, a ) \in \mathbf G(n,m)$ hence is $\phi \restrict R_i$ measurable
	by \cite[3.2.25, 3.2.28\,(2)\,(4)]{MR41:1976}.  Thus, defining $V \in
	\mathbf V_m ( X )$ with $\| V \| = \phi$ by
	\begin{equation*}
		V(k) = {\textstyle\int} k(x,\Tan^m ( \phi, x)) \ud \phi \, x
		\quad \text{for $k \in \mathscr K ( X \times \mathbf
		G(n,m))$},
	\end{equation*}
	we notice that, for $\phi$ almost all $a$, we have
	\begin{equation*}
		V^{(a)} ( \beta ) = \beta ( \Tan^m ( \phi, a ) ) \quad
		\text{for $\beta \in \mathscr K ( \mathbf G (n,m))$}
	\end{equation*}
	by \cite[2.8.18, 2.9.13]{MR41:1976}.  Since $\phi = \mathscr H^m
	\restrict \boldsymbol \Uptheta^m ( \phi, \cdot )$ by
	\cite[2.8\,(5)]{MR0307015}, we have $V \in \mathbf{RV}_m ( X)$ by
	\cite[3.5\,(1a)]{MR0307015} and the conclusion follows.
\end{proof}

\begin{remark} \label{remark:rectifiable-varifolds}
	Recalling \cite[3.5\,(1b)]{MR0307015}, we infer the following
	assertion: If also $\chi$ satisfies the hypotheses
	of~\ref{thm:rectifiability} with $\phi$ replaced by $\chi$, then, we
	have that, for $\mathscr H^m$ almost all $a$ with $\boldsymbol
	\Uptheta^m ( \phi, a ) > 0$ and $\boldsymbol \Uptheta^m ( \chi, a ) >
	0$,
	\begin{gather*}
		\text{$\Tan^m ( \phi, a ) = \Tan^m ( \chi, a )$ is an $m$
		dimensional vector space} \\
		\text{and $( \phi, m ) \ap \Der f(a) = ( \chi, m ) \ap \Der f
		(a)$.}
	\end{gather*}
\end{remark}

\begin{theorem} \label{thm:push-forward-rectifiable-varifolds}
	Suppose $X$ and $Y$ are open subsets of $\mathbf R^n$ and $\mathbf
	R^\nu$, respectively, $f : X \to Y$ is locally Lipschitzian, $m$ is a
	nonnegative integer with $m \leq n$ and $m \leq \nu$, $V \in
	\mathbf{RV}_m ( X )$, $f | \spt \| V \|$ is proper, and $g : Y \to
	\mathbf R^\mu$ is a locally Lipschitzian map.
	
	Then, a member $W$ of $\mathbf{RV}_m ( Y )$ may be defined by
	\begin{equation*}
		W(k) = {\textstyle\int_A} k (f(x), \im ( \| V \|, m ) \ap \Der
		f(x)) ( \| V \|, m ) \ap \Jac_m f(x) \ud \| V \| \, x
	\end{equation*}
	whenever $k \in \mathscr K ( Y \times \mathbf G (\nu,m))$, where $A =
	\{ x \with (\| V \|, m ) \ap \Jac_m f(x)>0 \}$ and $(\| V \|, m ) \ap
	\Jac_m f = \big \| \bigwedge_m ( \| V \|, m ) \ap \Der f \big \|$; in
	particular, we have
	\begin{equation*}
		\| W \| = f_\# ( \| V \| \restrict (\|V\|,m) \ap \Jac_m f ).
	\end{equation*}
	Moreover, for $\mathscr H^m$ almost all $y \in Y$, there holds
	\begin{equation*}
		\boldsymbol \Uptheta^m ( \| W \|, y ) = \sum_{x \in f^{-1} [
		\{ y \} ]} \boldsymbol \Uptheta^m ( \| V \|, x ),
	\end{equation*}
	and, if $f(x) = y$ and $\boldsymbol \Uptheta^m ( \| V \|, x ) > 0$,
	then
	\begin{equation*}
		\text{$\im \, (\| V \|,m) \ap \Der f(x) = \Tan^m ( \| W \|,
		y)$ is an $m$ dimensional vector space},
	\end{equation*}
	$g$ is $(\| W \|, m)$ approximately differentiable at $y$, and
	\begin{equation*}
		( \| V \|, m ) \ap \Der (g \circ f ) (x) = ( \| W \|, m ) \ap
		\Der g (y) \circ ( \| V \|, m ) \ap \Der f (x).
	\end{equation*}
\end{theorem}

\begin{proof}
	As $\big \| \bigwedge_m L \big \| = \big \| \bigwedge_m ( L \circ
	T_\natural ) \big \|$ for $T \in \mathbf G (n,m)$ and $L \in \Hom ( T,
	\mathbf R^\nu)$, the legitimacy of the definition of $W$ and the
	equation for $\| W \|$ follows from \cite[4.5\,(1)\,(2)]{MR2898736},
	\ref{remark:radon_measure_restriction},
	\ref{lemma:push_forward_borel_regular}, and \cite[2.2.3,
	3.2.28\,(4)]{MR41:1976}.  To prove the remaining conclusions, we
	firstly consider the special case that $V$ is associated with a
	compact subset $K$ of $X$ which is contained in an $m$ dimensional
	submanifold of class $1$ of $\mathbf R^n$, and
	\begin{align*}
		& \text{either $( \| V \|, m ) \ap \Jac_m f(x) = 0$ for
		$\mathscr H^m$ almost all $x \in K$}, \\
		& \text{or $f|K$ is univalent and $(f|K)^{-1}$ is
		Lipschitzian}.
	\end{align*}
	If the second alternative holds, then, noting that $( \| V \|, m ) \ap
	\Jac_mf(x) > 0$ for $\mathscr H^m$ almost all $x \in K$ by
	\cite[2.10.11, 3.2.20]{MR41:1976}, we infer from
	\cite[3.2.17]{MR41:1976} in conjunction with \cite[2.10.19\,(4),
	3.1.6, 3.1.19\,(4)]{MR41:1976} that
	\begin{gather*}
		\Tan^m ( \mathscr H^m \restrict f[K], f(x)) = \im \, (
		\mathscr H^m \restrict K, m ) \ap \Der f (x) \in \mathbf
		G(n,m), \\
		\begin{aligned}
			& ( \mathscr H^m \restrict K, m ) \ap \Der ( g \circ
			f) (x) \\
			& \qquad = ( \mathscr H^m \restrict f[K], m ) \ap \Der
			g (f(x)) \circ ( \mathscr H^m \restrict K, m ) \ap
			\Der f (x)
		\end{aligned}
	\end{gather*}
	for $\mathscr H^m$ almost all $x \in K$.  Therefore, for both
	alternatives, $W$ is the rectifiable varifold associated with $f[K]$
	by \cite[3.2.20]{MR41:1976}.  As the characteristic functions of $K$
	and $f[K]$ are $\mathscr H^m$ almost equal to $\boldsymbol \Uptheta^m
	( \| V \|, \cdot )$ and $\boldsymbol \Uptheta^m ( \| W \|, \cdot )$,
	respectively, by \cite[2.10.19\,(4), 3.2.19]{MR41:1976}, the
	conclusion in the special case now follows by
	\cite[2.10.11]{MR41:1976} and \cite[3.5\,(1b)]{MR0307015}.  In the
	general case, we use \cite[3.5\,(1)]{MR3625810}, \cite[2.2.3,
	3.2.2]{MR41:1976}, and \ref{remark:rectifiable-varifolds} to express
	$V = \sum_{i=1}^\infty c_i V_i$ for some $0 < c_i < \infty$ and some
	varifolds $V_i$ satisfying the conditions of the special case, whence
	we infer the conclusion by \cite[3.5\,(2)]{MR0307015} in conjunction
	with \cite[2.10.11]{MR41:1976}, \cite[3.5\,(1b)]{MR0307015}, and
	\ref{remark:rectifiable-varifolds}.
\end{proof}

\begin{remark} \label{remark:push-forward-rectifiable-varifolds}
	If $f$ is of class $\infty$, then $W = f_\# V$ as defined in
	\cite[3.2]{MR0307015}; thus, we extend the definition of $f_\# V$ to
	encompass the presently considered maps $f$.
\end{remark}

\begin{remark} \label{remark:functorial_varifold_push}
	If additionally $\mu \geq m$, $Z$ is an open subset $\mathbf R^\mu$,
	$g : Y \to Z$ is locally Lipschitzian, and $g \circ f | \spt \| V \|$
	is proper as map into $Z$, then
	\begin{equation*}
		( g \circ f )_\# V = g_\# ( f_\# V ) \in \mathbf{RV}_m ( Z).
	\end{equation*}
	Allowing for the case $m > \nu$, we leave $f_\# V$ undefined but we
	are still assured that $( g \circ f)_\# V = 0$ because $\mathscr H^m (
	\im f ) = 0$ implies $\mathscr H^m ( \im ( g \circ f ) ) = 0$.
\end{remark}

\begin{corollary} \label{corollary:push-forward-restricted-rect-var}
	If additionally $u$ is $\| V \| \restrict ( \| V \|, m ) \ap \Jac_m f$
	integrable, then
	\begin{equation*}
		\big ( \| V \| \restrict ( \| V \|, m ) \ap \Jac_m f ) \big )
		(u) = {\textstyle\int \sum_{M \cap f^{-1} [ \{ y \} ] }
		\boldsymbol \Uptheta^m ( \| V \|, \cdot ) u \ud \mathscr H^m
		\, y},
	\end{equation*}
	where $M = \{ x \with \boldsymbol \Uptheta^m ( \| V \|, x) > 0 \}$.
\end{corollary}

\begin{proof}
	If $u$ is the characteristic function of a $\| V \|$ measurable set
	$B$ over $X$, then
	\begin{equation*}
		\boldsymbol \Uptheta^m ( \| V \| \restrict B, x ) =
		\boldsymbol \Uptheta^m ( \| V \|,x ) u (x) \quad \text{for
		$\mathscr H^m$ almost all $x \in X$}
	\end{equation*}
	by \cite[2.8.18, 2.9.11]{MR41:1976} and \cite[3.5\,(1b)]{MR0307015};
	whence, recalling \cite[2.10.11]{MR41:1976} and
	\ref{remark:rectifiable-varifolds}, we infer the conclusion from
	\ref{thm:push-forward-rectifiable-varifolds} with $V$ replaced by $V
	\restrict B \times \mathbf G(n,m)$.  Considering the subcase $A=B$,
	the assertion extends to $\| V \| \restrict ( \| V \|, m ) \ap \Jac_m
	f$ measurable sets $B$.  Therefore, the general case follows by means
	of the usual approximation procedure, see \cite[2.1.1\,(6)\,(10),
	2.3.3, 2.4.4\,(6), 2.4.8]{MR41:1976}.
\end{proof}

\begin{corollary} \label{corollary:push-forward-rectifiable-varifolds}
	Suppose additionally $G$ is a complete separable normed commutative
	group, $v$ is a $G$ valued $\| V \| \restrict ( \| V \|, m ) \ap
	\Jac_m f$ measurable function,
	\begin{equation*}
		| v(x) | \leq \boldsymbol \Uptheta^m ( \| V \|, x ) \quad
		\text{for $\| V \| \restrict ( \| V \|, m ) \ap \Jac_m f$
		almost all $x$},
	\end{equation*}
	and $M = \{ x \with \boldsymbol \Uptheta^m ( \| V \|, x ) > 0 \}$.
	
	Then, an $\mathscr H^m$ measurable function $\xi$ may be defined by
	\begin{equation*}
		\xi(y) = \sum_{x \in M \cap f^{-1} [ \{ y \} ]} v(x) \in G
		\quad \text{whenever $y \in Y$},
	\end{equation*}
	where $\mathscr H^m$ refers to the $m$ dimensional Hausdorff measure
	over $Y$.
\end{corollary}

\begin{proof}
	We recall \cite[2.10.11]{MR41:1976} and \cite[3.5\,(1b)]{MR0307015}.
	As the set $M$ is countably $( \mathscr H^m, m )$ rectifiable, the
	function $N(f|B \cap M, \cdot)$ is $\mathscr H^m$ measurable whenever
	$B$ is a Borel subset of $X$ by \cite[3.2.20]{MR41:1976}.  This
	implies the conclusion in the special case that $v : M \to G$ is a
	Borel function with finite image and $|v(x)| \leq \boldsymbol
	\Uptheta^m ( \| V \|, x )$ for $x \in M$.  In the general case, there
	exists a sequence $w_1, w_2, w_3, \ldots$ of functions satisfying the
	conditions of the special case and
	\begin{equation*}
		\lim_{i \to \infty} w_i(x)=v(x) \quad \text{for $\| V \|
		\restrict ( \| V \|, m ) \ap \Jac_m f$ almost all $x$}.
	\end{equation*}
	Since $\mathscr H^m \big ( f [ M \cap \{ x \with ( \| V \|, m ) \ap
	\Jac_m f(x) = 0 \} ] \big ) = 0$, the conclusion follows.
\end{proof}

\begin{miniremark} \label{miniremark:kernel}
	If $T \in \mathbf G (n,m)$ is associated with the simple $m$ vector
	$\zeta \in \bigwedge_m \mathbf R^n$ and $h \in \Hom (T, \mathbf
	R^\kappa )$ satisfies $\bigwedge_\kappa h \neq 0$, then, cf.\
	\cite[4.3.8\,(3)]{MR41:1976}, $\ker h$ is associated with
	\begin{equation*}
		{\textstyle \zeta \restrict \bigwedge^\kappa ( h \circ
		T_\natural ) ( \omega ) \quad \text{whenever $0 \neq \omega
		\in \bigwedge^\kappa \mathbf R^\kappa$}.}
	\end{equation*}
\end{miniremark}

\begin{theorem} \label{thm:coarea-rectifiable-varifolds}
	Suppose $X$ is an open subset of $\mathbf R^n$, $f : X \to \mathbf
	R^\kappa$ is locally Lipschitzian, $m$ is an integer, $\kappa \leq m
	\leq n$, $V \in \mathbf{RV}_m ( X )$, the prefix $\ap$ denotes $( \| V
	\|, m)$ approximate differentiation, and $\ap \Jac_\kappa f = \big \|
	\bigwedge_\kappa \ap \Der f \big \|$.  Moreover, suppose $W$ is a
	function such that $y \in \dmn W$ if and only if $y \in \mathbf
	R^\kappa$ and
	\begin{equation*}
		\big ( \mathscr H^{m-\kappa} \restrict f^{-1} [ \{ y \} ]
		\big) \restrict \boldsymbol \Uptheta^m ( \| V \|, \cdot )
	\end{equation*}
	is the weight of some $Z \in \mathbf{RV}_{m-\kappa} ( X )$ and such
	that in this case $W(y) = Z$.
	
	Then, the following two statements hold.
	\begin{enumerate}
		\item \label{item:coarea-rectifiable-varifolds:structure} The
		function $W$ is $\mathscr L^\kappa$ measurable and, for
		$\mathscr L^\kappa$ almost all $y$, we have
		\begin{equation*}
			\Tan^{m-\kappa} ( \| W(y) \|, x ) = \ker \ap \Der f(x)
			\quad \text{for $\| W(y) \|$ almost all $x$}.
		\end{equation*}
		\item \label{item:coarea-rectifiable-varifolds:int} If $g$ is
		$\| V \| \restrict \ap \Jac_\kappa f$ integrable, then
		\begin{equation*}
			{\textstyle\int g \ud ( \| V \| \restrict \ap
			\Jac_\kappa f ) = \iint g \ud \| W(y) \| \ud \mathscr
			L^\kappa \, y}.
		\end{equation*}
	\end{enumerate}
\end{theorem}

\begin{proof}
	Denoting the statement that results from replacing $\| W (y) \|$ in
	\eqref{item:coarea-rectifiable-varifolds:int} by
	\begin{equation*}
		\phi(y) = \big ( \mathscr H^{m-\kappa} \restrict f^{-1} [ \{ y
		\} ] \big ) \restrict \boldsymbol \Uptheta^m ( \| V \|, \cdot)
	\end{equation*}
	by \eqref{item:coarea-rectifiable-varifolds:int}$'$ and recalling
	\cite[2.10.25]{MR41:1976} and \cite[3.5\,(1b)]{MR0307015}, we readily
	infer \eqref{item:coarea-rectifiable-varifolds:int}$'$ from
	\cite[3.5\,(2)]{MR3625810}, whence, in conjunction with
	\cite[2.10.19\,(3), 3.2.22\,(2)]{MR41:1976}, we infer that $\phi(y) =
	\| W(y) \|$ for $\mathscr L^\kappa$ almost all $y$ by
	\ref{thm:rectifiability}.  Using \cite[3.2.25,
	3.2.28\,(2)\,(4)]{MR41:1976} and \ref{miniremark:kernel}, we deduce
	that the function mapping $a \in A$ onto
	\begin{equation*}
		\ker \ap \Der f (a) \in \mathbf G (n,m-\kappa)
	\end{equation*}
	is $\| V \| \restrict A$ measurable from
	\cite[4.5\,(1)\,(2)]{MR2898736}, where $A = \{ a \with \ap \Jac_\kappa
	f (a) > 0 \}$.  In view of \cite[2.23]{MR3528825} and
	\eqref{item:coarea-rectifiable-varifolds:int}$'$, we may therefore
	define an $\mathscr L^\kappa$ measurable $\mathbf V_m ( X )$ valued
	function $Z$ such that, for $\mathscr L^\kappa$ almost all $y$,
	\begin{equation*}
		Z(y)(k) = {\textstyle\int} k(x,\ker \ap \Der f(x)) \ud \phi
		(y) \, x \quad \text{whenever $k \in \mathscr K ( X \times
		\mathbf G (n,m-\kappa) )$}
	\end{equation*}
	and it remains to show that $Z(y)$ is rectifiable for $\mathscr
	L^\kappa$ almost all $y$.  Recalling
	\ref{remark:rectifiable-varifolds}, this follows from
	\cite[3.5\,(1)]{MR3625810} in the special case that $f$ is of class
	$1$ to which the general case may be reduced by means of
	\cite[11.1]{MR3528825}.
\end{proof}

\begin{remark} \label{remark:coarea-rectifiable-varifolds}
	If $C$ is an $\mathscr L^\kappa$ Vitali relation, then, for $\mathscr
	L^\kappa$ almost all $y$,
	\begin{equation*}
		W(y)(k) = (C) \lim_{S \to y} \mathscr L^\kappa (S)^{-1}
		{\textstyle\int_{f^{-1} [S]}} k ( x, \ker \ap \Der f(x) ) \ud
		( \| V \| \restrict \ap \Jac_\kappa f ) \, x
	\end{equation*}
	whenever $k \in \mathscr K ( U \times \mathbf G (n,m-\kappa))$ by
	\cite[2.9.8]{MR41:1976} in conjunction with
	\cite[3.5\,(1b)]{MR0307015} and \cite[2.23]{MR3528825}.  Thus, in case
	$\kappa = 1$ and $\| \updelta V \|$ is a Radon measure, we recall
	\cite[2.8.17]{MR41:1976} and \cite[8.5, 8.7, 8.29, 12.1]{MR3528825} to
	similarly conclude
	\begin{equation*}
		\| V \boundary \{ x \with f(x)>y \} \| = \big ( \mathscr
		H^{m-1} \restrict f^{-1} [ \{ y \} ] \big ) \restrict
		\boldsymbol \Uptheta^m ( \| V \|, \cdot ) \quad \text{for
		$\mathscr L^1$ almost all $y$}.
	\end{equation*}
\end{remark}

\begin{remark} \label{remark:comparison-coarea}
	For $X = \mathbf R^n$, the statement of
	\ref{thm:coarea-rectifiable-varifolds} is analogous to those for flat
	chains with positive densities in \cite[4.3.2\,(2),
	4.3.8\,(2)\,(3)]{MR41:1976} with $U = \mathbf R^n$.
\end{remark}

\begin{theorem} \label{thm:product-rect-var}
	Suppose $X$ and $Y$ are open subsets of $\mathbf R^n$ and $\mathbf
	R^\nu$, respectively, $m$ and $\mu$ are nonnegative integers, $m \leq
	n$, $\mu \leq \nu$, $\phi$ and $\chi$ are the weights of some members
	of $\mathbf{RV}_m ( X )$ and $\mathbf{RV}_\mu ( Y )$, respectively,
	and
	\begin{equation*}
		M = \{ x \with \boldsymbol \Uptheta^m ( \phi, x ) > 0 \},
		\quad N = \{ y \with \boldsymbol \Uptheta^\mu ( \chi, y ) > 0
		\}.
	\end{equation*}
	
	Then, there holds
	\begin{equation*}
		( \mathscr H^m \restrict M ) \times ( \mathscr H^\mu \restrict
		N ) = \mathscr H^{m+\mu} \restrict ( M \times N )
	\end{equation*}
	and $M \times N$ is $\mathscr H^{m+\mu}$ almost equal to $\{ z \with
	\boldsymbol \Uptheta^{m+\mu} ( \phi \times \chi, z ) > 0 \}$.
\end{theorem}

\begin{proof}
	We let $Z = \{ z \with \boldsymbol \Uptheta^{m+\mu} ( \phi \times
	\chi, z ) > 0 \}$ and recall \cite[2.4.10]{MR41:1976} and
	\cite[3.5\,(1b)]{MR0307015}.  By \cite[3.6\,(1)\,(2)]{MR3625810}, we
	have $\phi \times \chi = \mathscr H^{m+\mu} \restrict \boldsymbol
	\Uptheta^{m+\mu} ( \phi \times \chi, \cdot )$ and
	\begin{equation*}
		\boldsymbol \Uptheta^{m+\mu} ( \phi \times \chi, (x,y)) =
		\boldsymbol \Uptheta^m (\phi,x) \boldsymbol \Uptheta^\mu (
		\chi, y) \quad \text{for $\phi \times \chi$ almost all
		$(x,y)$};
	\end{equation*}
	in particular, $\mathscr H^{m+\mu} ( Z \without ( M \times N)) = 0$
	and, since $\boldsymbol \Uptheta_\ast^{m+\mu} ( \phi \times \chi,
	(x,y)) > 0$ for $(x,y) \in M \times N$ by
	\cite[2.6.2\,(2)]{MR41:1976}, also $\mathscr H^{m+\mu} ( ( M \times N)
	\without Z ) = 0$.  Finally,
	\begin{gather*}
		\big ( \phi \restrict \boldsymbol \Uptheta^m ( \phi,
		\cdot)^{-1} \big ) \times \big ( \chi \restrict \boldsymbol
		\Uptheta^\mu ( \chi, \cdot )^{-1} \big ) = \big ( ( \phi
		\times \chi ) \restrict \boldsymbol \Uptheta^{m+\mu} ( \phi
		\times \chi, \cdot )^{-1} \big )
	\end{gather*}
	may be verified by means of \cite[2.6.2\,(4)]{MR41:1976} and
	\ref{remark:restriction}. 
\end{proof}

\begin{remark} \label{remark:product-rect-var}
	By \cite[3.2.24]{MR41:1976}, for general Borel subsets $M$ of $\mathbf
	R^n$ and $N$ of $\mathbf R^\nu$ which are countably $(\mathscr
	H^m,m)$, respectively $(\mathscr H^\mu,\mu)$, rectifiable, it may
	happen that
	\begin{equation*}
		( \mathscr H^m \restrict M ) \times ( \mathscr H^\mu \restrict
		N ) \neq \mathscr H^{m+\mu} \restrict ( M \times N).
	\end{equation*}
\end{remark}

\section{Rectifiable chains}

\begin{miniremark} \label{miniremark:normed-group-bundle}
	Suppose $m$ and $n$ are nonnegative integers, $n \geq 1$, $G$ is a
	normed commutative group, the function $\beta : \bigwedge_m \mathbf
	R^n \to \bigodot_2 \bigwedge_m \mathbf R^n$ satisfies $\beta ( \zeta )
	= \zeta \odot \zeta/2$ for $\zeta \in \bigwedge_m \mathbf R^n$, and
	the Grassmann manifolds $\mathbf G_{\textup{o}} (n,m)$ and $\mathbf G
	(n,m)$ are canonically isometrically identified with subsets of
	$\bigwedge_m \mathbf R^n$ and $\bigodot_2 \bigwedge_m \mathbf R^n$,
	respectively, as in \cite[3.2.28\,(1)\,(4)]{MR41:1976}.  Abbreviating
	$\alpha = \beta |\mathbf G_{\textup o} (n,m)$, we recall from
	\cite[3.2.28\,(3)\,(4)]{MR41:1976} that $\alpha$ is an open map from
	$\mathbf G_{\textup o} (n,m)$ onto $\mathbf G (n,m)$ and
	\begin{equation*}
		|\alpha ( \zeta ) - \alpha ( \zeta' ) |^2 = |\zeta-\zeta'|^2 (
		1 + \zeta \bullet \zeta' )/2 \quad \text{for $\zeta, \zeta'
		\in \mathbf G_{\textup o} (n,m)$};
	\end{equation*}
	hence, $\Lip \alpha = 1$ if $m \leq n$, $\Lip \alpha = 0$ if $m > n$,
	and $\Lip \big ((\alpha|U)^{-1} \big) < \infty$ whenever $U \subset
	\mathbf G_{\textup o} (n,m)$ with $\diam U < 2$.
	
	We endow the set $B = \mathbf G_{\textup o} (n,m) \times G$ with the
	metric $R$ such that
	\begin{equation*}
		R \big ( (\zeta,g),(\zeta',g') \big ) = |\zeta-\zeta'|+|g-g'|
		\quad \text{for $(\zeta,g),(\zeta',g') \in B$}
	\end{equation*}
	and study the quotient $\Gamma$ of $B$ induced by the action of the
	two-element subgroup $\{ \mathbf 1_B, - \mathbf 1_B \}$ of isometries
	of $B$.  Taking $p : B \to \Gamma$ to be the canonical projection of
	$B$ onto $\Gamma$, we define maps $P : \Gamma \to \mathbf G(n,m)$, $N
	: \Gamma \to \mathbf R$, and
	\begin{equation*}
		+ : ( \Gamma \times \Gamma ) \cap \{ (\gamma,\gamma') \with P
		(\gamma) = P (\gamma') \} \to \Gamma
	\end{equation*}
	by requiring that, whenever $(\zeta,g),(\zeta,g') \in B$, we have
	\begin{equation*}
		P ( p ( \zeta,g) ) = \alpha ( \zeta ), \quad N (p (\zeta,g)) =
		|g|, \quad \text{and} \quad p ( \zeta, g ) + p ( \zeta, g' ) =
		p ( \zeta, g+g').
	\end{equation*}
	For $T \in \mathbf G(n,m)$, the set $H = P^{-1} [ \{ T \} ]$ endowed
	with addition $+|(H \times H)$ and group norm $N|H$ forms a normed
	commutative group. Next, noting that $\inf R [ \gamma \times \gamma' ]
	= \dist (b,\gamma')$ for $b \in \gamma$, we define a metric $\rho :
	\Gamma \times \Gamma \to \mathbf R$ on $\Gamma$ by
	\begin{equation*}
		\rho (\gamma,\gamma') = \inf R [ \gamma \times \gamma' ] \quad
		\text{for $\gamma,\gamma' \in \Gamma$}.
	\end{equation*}
	Clearly, we have $\Lip p \leq 1$, the metric $\rho$ induces the
	quotient topology on $\Gamma$, and $p$ is an open map.  Moreover,
	whenever $f$ maps $\Gamma$ into another metric space, there holds
	$\Lip f = \Lip (f \circ p)$, whence we infer
	\begin{equation*}
		\text{$\Lip p = 1 = \Lip P$ if $m \leq n$}, \quad \text{$\Lip
		p = 0 = \Lip P$ if $m > n$}.
	\end{equation*}
	Whenever $(\zeta,g),(\zeta',g') \in B$, we notice that
	\begin{align*}
		\rho \big ( p (\zeta,g), p(\zeta',g')\big) & = \inf \{ |
		\zeta-\zeta'| + |g-g'|, |\zeta+\zeta'|+|g+g'| \} \\
		& \geq \inf \{ |\zeta-\zeta'|+|g-g'|, 2 - | \zeta - \zeta' |
		\}.
	\end{align*}
	Thus, for $U \subset \mathbf G_{\textup o} (n,m)$ with $\diam U < 2$,
	the map $(p|(U\times G))^{-1}$, whose domain equals $P^{-1} \big [
	\alpha [U] \big ]$, is locally Lipschitzian, whence we infer that the
	Lipschitzian map
	\begin{equation*}
		\text{$\psi_U = p \circ \big ( ( \alpha|U )^{-1} \times
		\mathbf 1_G \big )$ from $\alpha [U] \times G$ onto $P^{-1}
		\big [ \alpha [U] \big ]$}
	\end{equation*}
	possesses a locally Lipschitzian inverse; moreover, we have
	\begin{equation*}
		P ( \psi_U (T,g)) = T, \quad ( N \circ \psi_U ) (T,g) = |g|,
		\quad \psi_U (T,g)+\psi_U (T,g') = \psi_U (T,g+g')
	\end{equation*}
	for $T \in \alpha [U]$ and $g,g' \in G$.  Therefore, the maps $N$ and 
	\begin{equation*}
		+ : ( \Gamma \times \Gamma ) \cap \{ (\gamma,\gamma') \with P
		(\gamma) = P (\gamma') \} \to \Gamma
	\end{equation*}
	are locally Lipschitzian; also, if $G$ is complete, so are the domains
	of these maps.
	
	Whenever $\sigma$ is a $\Gamma$ valued function, we let $|\sigma| = N
	\circ \sigma$.  If also $\tau$ is a $\Gamma$ valued function,
	$\sigma+\tau$ shall be the function with domain
	\begin{equation*}
		( \dmn \sigma ) \cap ( \dmn \tau ) \cap \{ x \with
		P(\sigma(x)) = P (\tau(x)) \}
	\end{equation*}
	and value $\sigma(x)+\tau(x)$ at $x$ in its domain.  Henceforward, we
	will denote the metric space $\Gamma$ by $\mathbf G(n,m,G)$, $P$ by
	$\boldsymbol \uppi_{\mathbf G(n,m,G)}$, and the function $N$ by $|
	\cdot |$.
\end{miniremark}

\begin{miniremark} \label{miniremark:mini-product}
	Suppose $m$ and $n$ are nonnegative integers, $n \geq 1$, $G$ is a
	complete normed commutative group, and, for such $G$, the map $p_G :
	\mathbf G_{\textup o} (n,m) \times G \to \mathbf G(n,m,G)$ denotes the
	quotient map and the product $\mathbf G(n,m,\mathbf Z) \times G$ is
	endowed with the metric whose value at $((\delta,g),(\delta',g')) \in
	( \mathbf G (n,m,\mathbf Z) \times G)^2$ equals
	\begin{equation*}
		\rho (\delta,\delta') + |g-g'|, \quad \text{where $\rho$ is
		the metric on $\mathbf G(n,m,\mathbf Z)$}.
	\end{equation*}
	Then, we define a map $\lambda : \mathbf G (n,m,\mathbf Z ) \times G
	\to \mathbf G (n,m,G)$ by requiring
	\begin{equation*}
		\lambda ( p_{\mathbf Z} (\zeta,d),g) = p_G ( \zeta, d \cdot g)
		\quad \text{for $\zeta \in \mathbf G_{\textup o}(n,m)$, $d \in
		\mathbf Z$, and $g \in G$}.
	\end{equation*}
	Employing the maps $\psi_U$ of \ref{miniremark:normed-group-bundle}
	with $G$ replaced by $\mathbf Z$, we verify that $\lambda$ is locally
	Lipschitzian and that $\lambda | \big ( \boldsymbol \uppi_{\mathbf
	G(n,m,\mathbf Z)}^{-1} [ \{ T \} ] \times G \big ) \to \boldsymbol
	\uppi_{\mathbf G(n,m,G)}^{-1} [ \{ T \} ]$ is bilinear for $T \in
	\mathbf G (n,m)$.  Henceforward, we will denote $\lambda (\delta,g)$
	by $\delta \cdot g$ and, whenever $\sigma$ is a $\mathbf G(n,m,\mathbf
	Z)$ valued function and $g \in G$, we will designate by $\sigma \cdot
	g$ the function with the same domain as $\sigma$ and value $\sigma (x)
	\cdot g$ at $x \in \dmn \sigma$.  We also note that
	\begin{equation*}
		| \delta \cdot g | \leq | \delta | | g | \quad \text{for
		$\delta \in \mathbf G (n,m, \mathbf Z)$ and $g \in G$ with
		equality if $| \delta | \leq 1$}.
	\end{equation*}
\end{miniremark}

\begin{miniremark} \label{miniremark:G-push-and-restriction}
	Suppose $m$ is a nonnegative integer, $n$ is a positive integer, $G$
	is a normed commutative group, and $Y$ is a normed space.  Then, we
	endow
	\begin{equation*}
		\mathbf G(n,m,G,Y) = \{ (\gamma,h) \with \gamma \in \mathbf
		G(n,m,G), h \in \Hom ( \boldsymbol \uppi_{\mathbf G(n,m,G)}
		(\gamma), Y ) \}
	\end{equation*}
	with a metric whose value at $((\gamma,h),(\gamma',h')) \in \mathbf
	G(n,m,G,Y)^2$ equals
	\begin{equation*}
		\rho (\gamma,\gamma') + \| h \circ ( \dmn h )_\natural - h'
		\circ ( \dmn h')_\natural \|,
	\end{equation*}
	where $\rho$ is the metric on $\mathbf G(n,m,G)$.  If $G$ and $Y$ are
	complete, so is $\mathbf G(n,m,G,Y)$.
	
	Whenever $\nu$ is a positive integer, we employ the quotient maps
	\begin{equation*}
		p : \mathbf G_{\textup o} (n,m) \times G \to \mathbf G(n,m,G)
		\quad \text{and} \quad q : \mathbf G_{\textup o} (\nu,m)
		\times G \to \mathbf G(\nu,m,G)
	\end{equation*}
	to define the map
	\begin{equation*}
		P : \mathbf G(n,m,G,\mathbf R^\nu) \cap \big \{ (\gamma,h)
		\with {\textstyle \bigwedge_m h \neq 0} \big \} \to \mathbf
		G(\nu,m,G)
	\end{equation*}
	so that, whenever $T \in \mathbf G(n,m)$, $g \in G$, $h \in \Hom ( T,
	\mathbf R^\nu )$, $\bigwedge_m h \neq 0$, and $T$ is associated with
	$\zeta \in \mathbf G_{\textup o} (n,m)$, we have
	\begin{equation*}
		P ( p ( \zeta,g), h ) = q \left ( \frac{\bigwedge_m ( h \circ
		( \dmn h)_\natural ) ( \zeta )}{\big |\bigwedge_m ( h \circ (
		\dmn h)_\natural ) ( \zeta ) \big |}, g \right ).
	\end{equation*}
	Clearly, $\dmn P$ is an open subset of $\mathbf G(n,m,G,\mathbf
	R^\nu)$.  Employing the maps $\psi_U$ of
	\ref{miniremark:normed-group-bundle}, we readily verify that $P$ is
	locally Lipschitzian.  Henceforth, we will denote $P ( \gamma,h)$ by
	$h_\# \gamma$.  We also note that, whenever $(\gamma,h),(\gamma',h)
	\in \mathbf G(n,m,G, \mathbf R^\nu )$, $\bigwedge_m h \neq 0$, and
	$\boldsymbol \uppi_{\mathbf G(n,m,G)} (\gamma) = \boldsymbol
	\uppi_{\mathbf G(n,m,G)} (\gamma')$, we have
	\begin{gather*}
		| h_\# \gamma | = |\gamma|, \quad \boldsymbol \uppi_{\mathbf
		G( \nu,m,G) } ( h_\# \gamma ) = \im h, \quad h_\#
		(\gamma+\gamma') = h_\# \gamma + h_\# \gamma', \\
		( i \circ h )_\# \gamma = i_\# ( h_\# \gamma ) \quad
		\text{whenever $i \in \Hom ( \im h, \mathbf R^\mu )$ and
		${\textstyle\bigwedge_m i} \neq 0$},
	\end{gather*}
	where $\mu$ is a positive integer,
	\begin{equation*}
		h_\# ( \delta \cdot g ) = ( h_\# \delta ) \cdot g \quad
		\text{whenever $(\delta,h) \in \mathbf G (n,m, \mathbf Z,
		\mathbf R^\nu)$, ${\textstyle \bigwedge_m h \neq 0}$, and $g
		\in G$}.
	\end{equation*}
	
	Similarly, whenever $\kappa$ is a positive integer and $\kappa \leq
	m$, we employ the quotient maps $p : \mathbf G_{\textup o} (n,m)
	\times G \to \mathbf G (n,m,G)$ and $r : \mathbf G_{\textup o}
	(n,m-\kappa) \times G \to \mathbf G (n,m-\kappa,G)$ and recall
	\ref{miniremark:kernel} to define
	\begin{equation*}
		Q : \mathbf G (n,m,G, \mathbf R^\kappa) \cap \big \{
		(\gamma,h) \with \textstyle{\bigwedge^\kappa} h \neq 0 \big \}
		\to \mathbf G (n,m-\kappa, G )
	\end{equation*}
	so that, whenever $T \in \mathbf G (n,m)$, $g \in G$, $h \in \Hom ( T,
	\mathbf R^\kappa )$, $\bigwedge^\kappa h \neq 0$, and $T$ is
	associated with $\zeta \in \mathbf G_{\textup o} (n,m)$, we have
	\begin{equation*}
		Q ( p(\zeta,g), h ) = r \left ( \frac{ \zeta \restrict
		{\textstyle\bigwedge^\kappa} ( h \circ ( \dmn h)_\natural ) (
		\omega )}{ \big | \zeta \restrict {\textstyle\bigwedge^\kappa}
		( h \circ ( \dmn h)_\natural ) ( \omega ) \big |}, g \right )
		{\textstyle \quad \text{whenever $0 \neq \omega \in
		\bigwedge^\kappa \mathbf R^\kappa$}.}
	\end{equation*}
	Clearly, $\dmn Q$ is an open subset of $\mathbf G (n,m,G,\mathbf
	R^\kappa)$.  Employing the maps $\psi_U$ of
	\ref{miniremark:normed-group-bundle}, we readily verify that $Q$ is
	locally Lipschitzian.  Henceforth, we will denote $Q(\gamma,h)$ by
	$\gamma \restrict h$.  We finally note that, whenever
	$(\gamma,h),(\gamma',h) \in \mathbf G(n,m,G,\mathbf R^\kappa )$,
	$\bigwedge^\kappa h \neq 0$, and $\boldsymbol \uppi_{\mathbf G(n,m,G)}
	( \gamma ) = \boldsymbol \uppi_{\mathbf G(n,m,G)} (\gamma')$, we have
	\begin{equation*}
		|\gamma \restrict h| = |\gamma|, \quad  \boldsymbol
		\uppi_{\mathbf G(n,m-\kappa,G)} ( \gamma \restrict h ) = \ker
		h, \quad (\gamma+\gamma') \restrict h = \gamma \restrict h +
		\gamma' \restrict h
	\end{equation*}
	and that, whenever $(\delta,h) \in \mathbf G(n,m,\mathbf Z, \mathbf
	R^\kappa)$, $\bigwedge^\kappa h \neq 0$, and $g \in G$, we have
	\begin{equation*}
		( \delta \cdot g ) \restrict h = ( \delta \restrict h ) \cdot
		g.
	\end{equation*}
\end{miniremark}

\begin{miniremark} \label{miniremark:product-grassmann-bundle}
	Suppose $m$ and $\mu$ are nonnegative integers, $n$ and $\nu$ are
	positive integers,
	\begin{equation*}
		p : \mathbf G_{\textup o} ( n,m ) \times \mathbf Z \to \mathbf
		G(n,m, \mathbf Z ), \quad q : \mathbf G_{\textup o} ( \nu,
		\mu) \times G \to \mathbf G ( \nu, \mu, G ),
	\end{equation*}
	and $r : \mathbf G_{\textup o} ( n+\nu, m+\mu ) \times G \to \mathbf G
	( n+\nu, m+\mu, G )$ are the quotient maps,
	\begin{gather*}
		P : \mathbf R^n \to \mathbf R^n \times \mathbf R^\nu, \quad Q
		: \mathbf R^\nu \to \mathbf R^n \times \mathbf R^\nu, \\
		P (x) = (x,0) \quad \text{and} \quad Q(y) = (0,y) \quad
		\text{for $(x,y) \in \mathbf R^n \times \mathbf R^\nu$},
	\end{gather*}
	and $\mathbf R^n \times \mathbf R^\nu \simeq \mathbf R^{n+\nu}$.
	Then, we define a map
	\begin{equation*}
		\kappa : \mathbf G (n,m,\mathbf Z) \times \mathbf G(\nu,\mu,
		G) \to \mathbf G ( n+\nu, m+\mu, G )
	\end{equation*}
	so that, whenever $\zeta \in \mathbf G_{\textup o} ( n,m )$, $\eta \in
	\mathbf G_{\textup o} ( \nu, \mu )$, $d \in \mathbf Z$, and $g \in G$,
	we have
	\begin{equation*}
		\kappa ( p ( \zeta, d ), q ( \eta, g ) ) = r \big (
		{\textstyle \bigwedge_m} P ( \zeta ) \wedge
		{\textstyle\bigwedge_\mu} Q ( \eta ), d \cdot g \big ).
	\end{equation*}
	Employing the maps $\psi_U$ of \ref{miniremark:normed-group-bundle}
	with $(n,m,G)$ replaced by $(n,m, \mathbf Z)$ and $(\nu,\mu,G)$,
	respectively, we readily verify that $\kappa$ is locally Lipschitzian
	and that
	\begin{equation*}
		\kappa \big | \big ( \boldsymbol \uppi_{\mathbf G (n,m,
		\mathbf Z)}^{-1} [ \{ S \} ] \times \boldsymbol \uppi_{\mathbf
		G ( \nu, \mu, G)}^{-1} [ \{ T \} ] \big ) \to \boldsymbol
		\uppi_{\mathbf G ( n+\nu, m+\mu, G )}^{-1} [ \{ S \times T \}
		]
	\end{equation*}
	is bilinear whenever $S \in \mathbf G(n,m)$ and $T \in \mathbf
	G(\nu,\mu)$.  Henceforward, we will denote $\kappa ( \delta,\gamma )$
	by $\delta \times \gamma$ and, whenever $\sigma$ and $\tau$ are
	$\mathbf G(n,m,\mathbf Z)$ and $\mathbf G(\nu,\mu,G)$ valued
	functions, respectively, we will designate by $\sigma \times \tau$ the
	function with domain $\dmn \sigma \times \dmn \tau$ and value
	$\sigma(x) \times \tau (y)$ at $(x,y) \in \dmn \sigma \times \dmn
	\tau$.  We finally note	
	\begin{gather*}
		\text{$| \delta \times \gamma | \leq | \delta | | \gamma |$
		with equality if $G = \mathbf Z$ or $G = \mathbf R$ or
		$|\delta| \leq 1$}, \\
		\delta \times ( \delta' \cdot g ) = ( \delta \times \delta' )
		\cdot g
	\end{gather*}
	whenever $\delta \in \mathbf G (n,m, \mathbf Z )$, $\delta' \in
	\mathbf G ( \nu, \mu, \mathbf Z )$, $\gamma \in \mathbf G(\nu,\mu,G)$,
	and $g \in G$.  $\big($Defining $\mathbf G(0,0,G) \simeq G$ and
	allowing for $\nu = 0$, the $\cdot$ operation of
	\ref{miniremark:mini-product} could be considered a special case of
	the present $\times$ operation with $\mu = \nu = 0$.$\big)$
\end{miniremark}

\begin{miniremark} \label{miniremark:def-rectifiable-G-chains}
	Suppose $m$ is a nonnegative integer, $n$ is a positive integer, $U$
	is an open subset of $\mathbf R^n$, $\mathscr H^m$ refers to the $m$
	dimensional Hausdorff measure over $U$, and $G$ is a complete normed
	commutative group.  Then, we let $L(U,m,G)$ denote the set of all
	$\mathbf G (n,m,G)$ valued functions $\sigma$ such that the following
	conditions are satisfied: $\dmn \sigma \subset U$, $M = \{ x \with |
	\sigma(x) | > 0 \}$ is $\mathscr H^m$ measurable, $\sigma$ is
	$\mathscr H^m \restrict M$ measurable, for some separable subset $Z$
	of $\mathbf G(n,m,G)$, we have $\sigma(x) \in Z$ for $\mathscr H^m$
	almost all $x \in M$,
	\begin{equation*}
		{\textstyle\int_{K \cap M}} | \sigma | \ud \mathscr H^m <
		\infty \quad \text{whenever $K$ is a compact subset of $U$},
	\end{equation*}
	$U$ is countably $( \phi,m )$ rectifiable, and $\Tan^m ( \phi, x ) =
	\boldsymbol \uppi_{\mathbf G(n,m,G)} ( \sigma (x) )$ for $\phi$ almost
	all $x$, where we abbreviated $\phi = ( \mathscr H^m \restrict M )
	\restrict | \sigma |$; hence, $\sigma$ is a $\phi$ measurable function
	and $\phi$ is the weight of some member of $\mathbf{RV}_m ( U )$ by
	\ref{remark:restriction}, \cite[2.2.3]{MR41:1976}, and
	\ref{remark:rectifiable-varifolds}.  Elements $\sigma$ and $\tau$ of
	$L(U,m,G)$ are termed \emph{equivalent} if and only if the functions
	$\sigma | \{ x \with |\sigma(x)|>0 \}$ and $\tau | \{ x \with |
	\tau(x)|>0 \}$ are $\mathscr H^m$ almost equal; the resulting set of
	equivalence classes is denoted by
	\begin{equation*}
		\mathscr R_m^{\mathrm{loc}} ( U, G )
	\end{equation*}
	and its members are called $m$ dimensional \emph{locally rectifiable
	$G$ chains} in $U$.  For $S \in \mathscr R_m^{\mathrm{loc}} ( U, G )$,
	we denote by $\| S \|$ the Radon measure over $U$, which is equal to
	$( \mathscr H^m \restrict \{ x \with | \sigma(x)|>0 \}) \restrict |
	\sigma |$ for $\sigma \in S$, and we employ \cite[2.8.18,
	2.9.13]{MR41:1976} to define the function
	\begin{equation*}
		\vect S
	\end{equation*}
	to be the member of $S$ characterised by requiring that, for $\sigma
	\in S$ and $a \in U$, we have $a \in \dmn \vect S$ if and only if
	\begin{equation*}
		\boldsymbol \Uptheta^m ( \| S \|, a ) > 0 \quad \text{and}
		\quad ( \| S \|, V ) \aplim_{x \to a} \sigma (x) \in \mathbf G
		(n,m,G)
	\end{equation*}
	and in this case $\vect S (a)$ equals that approximate limit, where
	$V$ is the $\| S \|$ Vitali relation given by $V = \{ (x, \mathbf
	B(x,r)) \with x \in U, 0 < r < \dist (x, \mathbf R^n \without U) \}$.

	In case $m > n$, we have $L(U,m,G) = \{ \varnothing \}$, hence
	$\mathscr R_m^{\textup{loc}} (U,G)$ contains a single element, and, if
	$S \in \mathscr R_m^{\textup{loc}} ( U, G )$, then $\| S \| = 0$ and
	$\vect S = \varnothing$.
	
	The considerations of this paragraph rely on
	\ref{remark:rectifiable-varifolds} and \cite[3.5\,(1b)]{MR0307015}. We
	have $\| S \| = \mathscr H^m \restrict \boldsymbol \Uptheta^m ( \| S
	\|, \cdot )$ for $S \in \mathscr R_m^{\textup{loc}} (U, G )$.  For
	$\sigma \in L (U,m,G)$, we say that \emph{$\sigma$ is adapted to
	$\phi$} if and only if $\phi$ is the weight of some member of
	$\mathbf{RV}_m ( U)$ if $m \leq n$, $\phi$ is the zero measure over
	$U$ if $m > n$, the domain of $\sigma$ is $\mathscr H^m$ almost equal
	to $\{ x \with \boldsymbol \Uptheta^m ( \phi, x ) > 0 \}$, and
	\begin{equation*}
		\boldsymbol \uppi_{\mathbf G (n,m,G)} ( \sigma(x)) = \Tan^m (
		\phi, x) \quad \text{for $\phi$ almost all $x$};
	\end{equation*}
	for instance, $\vect S$ is adapted to $\| S \|$ for $S \in \mathscr
	R_m^{\textup{loc}} ( U, G)$.  If $S \in \mathscr R_m^{\mathrm{loc}} (
	U, G)$ and $A$ is $\| S \|$ measurable, then we define $S \restrict A
	\in \mathscr R_m^{\mathrm{loc}} ( U, G)$ by requiring $\sigma | A \in
	S \restrict A$ whenever $\sigma \in S$; hence, $\| S \restrict A \| =
	\| S \| \restrict A$.  Next, we define the sum
	\begin{equation*}
		S+T \in \mathscr R_m^{\mathrm{loc}} ( U, G )
	\end{equation*}
	of $S$ and $T$ in $\mathscr R_m^{\mathrm{loc}} ( U, G )$ by requiring
	\begin{equation*}
		\rho = \big ( \vect S + \vect T \big ) \cup \big ( \vect S |
		(U \without \dmn \vect T ) \big ) \cup \big ( \vect T | ( U
		\without \dmn \vect S ) \big ) \in S+T.
	\end{equation*}
	We infer that $\| S + T \| = ( \| S \| + \| T \| ) \restrict
	|\rho|/\Theta \leq \| S \| + \| T \|$, where we abbreviated $\Theta =
	\boldsymbol \Uptheta^m ( \| S \|, \cdot ) + \boldsymbol \Uptheta^m (
	\| T \|, \cdot)$, and note that $\sigma + \tau$ belongs to $S + T$ and
	is adapted to $\phi$ whenever $\sigma$ in $S$ and $\tau$ in $T$ are
	both adapted to $\phi$; for instance, to $\phi = \| S \| + \| T \|$.
	It follows that
	\begin{equation*}
		(S+T) \restrict A = S \restrict A + T \restrict A \quad
		\text{whenever $A$ is $\| S \| + \| T \|$ measurable}
	\end{equation*}
	and that $\mathscr R_m^{\mathrm{loc}} ( U, G )$ is a commutative group
	which is a complete topological group when endowed with the group norm
	with value
	\begin{equation*}
		\sum_{i=1}^\infty \inf \{ 2^{-i}, \| S \| (K_i) \} \quad
		\text{at $S \in \mathscr R_m^{\textup{loc}} ( U, G)$},
	\end{equation*}
	where $K_i = U \cap \{ x \with |x| \leq i, \dist (x, \mathbf R^n
	\without U ) \geq 1/i \}$; hence, if $A_1, A_2, A_3, \ldots$ is a
	disjoint sequence of $\| S \|$ measurable sets, then $S \restrict
	\bigcup_{i=1}^\infty A_i = \sum_{i=1}^\infty S \restrict A_i $.  We
	also let $\mathscr R_m ( U, G ) = \mathscr R_m^{\textup{loc}} ( U, G )
	\cap \{ S \with \text{$\spt \| S \|$ is compact} \}$.
	
	Finally, with respect to the group norm whose value at $S$ equals $\|
	S \| ( \mathbf R^n)$,
	\begin{equation*}
		\mathscr R_m^{\mathrm{loc}} ( \mathbf R^n, G ) \cap \{ S \with
		\| S \| ( \mathbf R^n ) < \infty \}
	\end{equation*}
	is a complete normed commutative group; its members correspond to the
	$m$ dimensional rectifiable $G$ chains of \cite[3.6]{MR2876138}.  We
	also notice, if $G$ is equal to a finite direct sum of cyclic groups
	with their standard group norm, then
	\begin{equation*}
		\mathscr R_m^{\mathrm{loc}} ( \mathbf R^n, G ) \cap \big \{ S
		\with \text{both $\| S \|$ and $(\vect S)_\# \| S \|$ have
		compact support} \big \}
	\end{equation*}
	is a subgroup thereof; if $m \leq n$, then its members correspond to
	the $G$ varifolds of dimension~$m$ in $\mathbf R^n$ as defined in
	\cite[2.4, 2.6\,(d)]{MR0225243}.
\end{miniremark}

\begin{miniremark} \label{miniremark:push-forward-G-chain}
	Suppose $m$ is a nonnegative integer, $n$ and $\nu$ are positive
	integers, $U$ and $V$ are open subsets of $\mathbf R^n$ and $\mathbf
	R^\nu$, respectively, $G$ is a complete normed commutative group, $S
	\in \mathscr R_m^{\textup{loc}} ( U, G )$, $f : U \to V$ is locally
	Lipschitzian, $f | \spt \| S \|$ is proper, $\chi = f_\# \big ( \| S
	\| \restrict \| \bigwedge_m (\| S \|, m) \ap \Der f \| \big )$, and $Y
	= \{ y \with \Tan^m ( \chi,y ) \in \mathbf G(\nu,m) \}$.
	
	Then, using \ref{remark:rectifiable-varifolds},
	\ref{thm:push-forward-rectifiable-varifolds},
	\ref{corollary:push-forward-rectifiable-varifolds}, and the maps
	$\psi_U$ of \ref{miniremark:normed-group-bundle} with $n$ replaced by
	$\nu$, we obtain an $\mathscr H^m \restrict Y$ measurable function
	$\tau$, defined on a subset of $Y$ by
	\begin{equation*}
		\tau (y) = \sum_{x \in ( \dmn \vect S) \cap f^{-1} [ \{ y \}
		]} \big ( ( \| S \|, m ) \ap \Der f (x) \big )_\# \big ( \vect
		S (x) \big ) \quad \text{whenever $y \in Y$},
	\end{equation*}
	where the summation is understood to be computed in the complete
	normed commutative group $\boldsymbol \uppi_{\mathbf G(\nu,m,G)}^{-1}
	[ \{ \Tan^m ( \chi, y ) \} ]$, we define $f_\# S$ in $\mathscr
	R_m^{\textup{loc}} ( V, G )$ by requiring that $\tau$ belongs to $f_\#
	S$ and we have
	\begin{equation*}
		\| f_\# S \| = \chi \restrict | \tau | / \boldsymbol
		\Uptheta^m ( \chi, \cdot ) \leq \chi, \quad \spt \| f_\# S \|
		\subset f [ \spt \| S \| ].
	\end{equation*}
	Applying \ref{corollary:push-forward-restricted-rect-var} to the
	characteristic function $u$ of $U \cap \{ x \with f(x) \neq g(x) \}$,
	we see that if also $g : U \to V$ is locally Lipschitzian, $g | \spt
	\| S \|$ is proper, and $g(x)=f(x)$ for $\| S \|$ almost all $x$, then
	$f_\# S = g_\# S$.  Employing \ref{remark:rectifiable-varifolds} and
	\ref{thm:push-forward-rectifiable-varifolds}, we verify that, if
	$\sigma$ in $S$ is adapted to $\phi$, $f | \spt \phi$ is proper, $\psi
	= f_\# ( \phi \restrict \| \bigwedge_m ( \phi,m ) \ap \Der f \| )$,
	and $\Upsilon = \{ \upsilon \with \Tan^m ( \psi, \upsilon ) \in
	\mathbf G ( \nu, m ) \}$, then $Y$ is $\mathscr H^m$ almost contained
	in $\Upsilon$ and the function $\rho$, defined on a subset of
	$\Upsilon$ by
	\begin{equation*}
		\rho (\upsilon) = \sum_{x \in (\dmn \sigma) \cap f^{-1} [ \{
		\upsilon \} ]} \big ( ( \phi, m ) \ap \Der f (x) \big )_\#
		\big ( \sigma (x) \big ) \quad \text{whenever $\upsilon \in
		\Upsilon$},
	\end{equation*}
	where the sum is computed in $\boldsymbol \uppi_{\mathbf
	G(\nu,m,G)}^{-1} [ \{ \Tan^m ( \psi, \upsilon ) \} ]$, belongs to
	$f_\# S$ and is adapted to $\psi$; in particular, $\tau$ is adapted to
	$\chi$.  Therefore, we firstly obtain
	\begin{equation*}
		( f_\# S ) \restrict B = f_\# ( S \restrict f^{-1} [ B ] )
		\quad \text{whenever $B$ is $f_\# \| S \|$ measurable},
	\end{equation*}
	as in this case $f^{-1} [B]$ is $\| S \|$ measurable by
	\cite[2.1.5\,(1)\,(4)]{MR41:1976}, secondly, in view of
	\ref{miniremark:G-push-and-restriction} and
	\ref{miniremark:def-rectifiable-G-chains},
	\begin{equation*}
		f_\# ( S + T ) = f_\# S + f_\# T
	\end{equation*}
	whenever also $T \in \mathscr R_m^{\textup{loc}} ( U, G )$ and $f |
	\spt \| T \|$ is proper, and thirdly, using \ref{remark:summation},
	\ref{thm:push-forward-rectifiable-varifolds},
	\ref{remark:functorial_varifold_push}, and
	\ref{miniremark:G-push-and-restriction},
	\begin{equation*}
		(g \circ f)_\# S = g_\# ( f_\# S )
	\end{equation*}
	whenever also $\mu$ is a positive integer, $W$ is an open subset of
	$\mathbf R^\mu$, $g : V \to W$ is locally Lipschitzian, and $g \circ f
	| \spt \| S \|$ is proper.  Finally, the homomorphism
	\begin{equation*}
		\text{$f_\# : \mathscr R_m^{\textup{loc}} ( U, G ) \cap \{ T
		\with \spt \| T \| \subset C \} \to \mathscr
		R_m^{\textup{loc}} ( V, G )$ is continuous}
	\end{equation*}
	whenever $C$ is a relatively closed subset of $U$ such that $f|C$ is
	proper.
\end{miniremark}

\begin{miniremark} \label{miniremark:product-G-chains}
	Suppose $m$ and $\mu$ are nonnegative integers, $n$ and $\nu$ are
	positive integers, $U$ and $V$ are open subsets of $\mathbf R^n$ and
	$\mathbf R^\nu$, respectively, $G$ is a complete normed commutative
	group, $S \in \mathscr R_m^{\textup{loc}} ( U, \mathbf Z )$, and $T
	\in \mathscr R_\mu^{\textup{loc}} ( V, G )$.  Then, in view of
	\cite[3.5\,(1b)]{MR0307015}, \cite[3.6\,(1)\,(2)]{MR3625810},
	\ref{remark:rectifiable-varifolds}, \ref{thm:product-rect-var}, and
	\ref{miniremark:product-grassmann-bundle}, we define the Cartesian
	product $S \times T \in \mathscr R_{m+\mu}^{\textup{loc}} ( U \times
	V, G )$ by requiring it to contain
	\begin{equation*}
		\vect S \times \vect T,
	\end{equation*}
	we have $\| S \times T \| \leq \| S \| \times \| T \|$ with equality
	if $G = \mathbf Z$ or $G = \mathbf R$ or $\boldsymbol \Uptheta^m ( \|
	S \|, x ) \leq 1$ for $\mathscr H^m$ almost all $x \in U$, and we
	notice that, if $\sigma$ in $S$ is adapted to $\phi$, $\tau$ in $T$ is
	adapted to $\chi$, $M = \{ x \with \boldsymbol \Uptheta^m ( \phi, x )
	> 0 \}$, and $N = \{ y \with \boldsymbol \Uptheta^\mu ( \chi, y ) > 0
	\}$, then, $\dmn ( \vect S \times \vect T )$ is $\mathscr H^{m+\mu}$
	almost contained in $M \times N$ and $( \sigma | M ) \times ( \tau |
	N)$ belongs to $S \times T$ and is adapted to $\phi \times \chi$;
	hence, $\vect S \times \vect T$ is adapted to $\| S \| \times \| T
	\|$. $\big ($Examples with $\mathscr H^{m+\mu} ( \dmn ( \sigma \times
	\tau ) \without (M \times N) ) > 0$ are readily constructed by means
	of \cite[3.2.24]{MR41:1976}.$\big )$ Therefore, we may verify that
	\begin{equation*}
		\times : \mathscr R_m ( U, \mathbf Z ) \times \mathscr R_\mu (
		V, G ) \to \mathscr R_{m+\mu} ( U \times V, G )
	\end{equation*}
	is a continuous bilinear operation using
	\ref{miniremark:product-grassmann-bundle} and
	\ref{miniremark:def-rectifiable-G-chains}.  Finally, if $q : U \times
	V \to V$ satisfies $q(x,y) = y$ for $(x,y) \in U \times V$ and $\spt
	\| S \|$ is compact, then
	\begin{equation*}
		\text{$q_\# ( S \times T ) = 0$ if $m>0$}, \quad \text{$q_\# (
		S \times T ) = {\textstyle \big ( \sum \vect S \big ) \cdot
		T}$ if $m = 0$},
	\end{equation*}
	where the isomorphism $\mathbf G (n,0,\mathbf Z) \simeq \mathbf Z$
	induced by the map $\psi_{\{1\}}$ of
	\ref{miniremark:normed-group-bundle} is used.
\end{miniremark}

\begin{miniremark} \label{miniremark:slicing-G-chains}
	Suppose $\kappa$, $m$, and $n$ are positive integers satisfying
	$\kappa \leq m$, $U$ is an open subset of $\mathbf R^n$, $G$ is a
	complete normed commutative group, $S \in \mathscr R_m^{\textup{loc}}
	( U, G )$, and $f : U \to \mathbf R^\kappa$ is locally Lipschitzian.
	Then, defining a $\mathbf G(n,m-\kappa,G)$ valued function $\tau$ on a
	subset of $U$ by
	\begin{equation*}
		\tau (x) = \vect S(x) \restrict ( \| S \|, m ) \ap \Der f (x)
		\quad \text{whenever $x \in U$},
	\end{equation*}
	\ref{thm:coarea-rectifiable-varifolds} yields that
	\begin{equation*}
		\tau | f^{-1} [ \{ y \} ] \in L (U,m-\kappa,G),
	\end{equation*}
	see \ref{miniremark:def-rectifiable-G-chains}, for $\mathscr L^\kappa$
	almost all $y$; for such $y$, we define $\langle S,f,y \rangle$ to be
	the unique member of $\mathscr R_{m-\kappa}^{\textup{loc}} ( U, G)$
	containing $\tau | f^{-1} [ \{ y \} ]$ so that in particular
	\begin{equation*}
		\Tan^{m-\kappa} ( \| \langle S,f,y \rangle \|, x ) = \ker \, (
		\| S \|, m ) \ap \Der f (x) \quad \text{for $\| \langle S,f,y
		\rangle \|$ almost all $x$}.
	\end{equation*}
	Moreover, for $\mathscr L^\kappa$ almost all $y$, we additionally have
	\begin{equation*}
		\| \langle S,f,y \rangle \| = \big ( \mathscr H^{m-\kappa}
		\restrict f^{-1} [ \{ y \} ] \big ) \restrict \boldsymbol
		\Uptheta^m ( \| S \|, \cdot ).
	\end{equation*}
	
	With the help of \ref{remark:rectifiable-varifolds} and
	\ref{thm:coarea-rectifiable-varifolds}, we next verify that, if
	$\sigma$ in $S$ is adapted to $\phi$, $X = \{ x \with \boldsymbol
	\Uptheta^m ( \phi, x ) > 0 \}$, $\rho$ is a $\mathbf G(n, m-\kappa,
	G)$ valued function, $\dmn \rho \subset X$, and
	\begin{equation*}
		\rho (x) = \sigma (x) \restrict ( \phi, m ) \ap \Der f (x)
		\quad \text{for $\mathscr H^m$ almost all $x \in X$},
	\end{equation*}
	then, for $\mathscr L^\kappa$ almost all $y$, the function $\rho |
	f^{-1} [ \{ y \} ]$ belongs to $\langle S,f,y \rangle$ and is adapted
	to $\big ( \mathscr H^{m-\kappa} \restrict f^{-1} [ \{ y \} ] \big)
	\restrict \boldsymbol \Uptheta^m ( \phi, \cdot)$; in particular, $\tau
	| f^{-1} [ \{ y \} ]$ in $\langle S,f,y \rangle$ is adapted to $\|
	\langle S,f,y \rangle \|$ for $\mathscr L^\kappa$ almost all $y$.  In
	view of \ref{thm:coarea-rectifiable-varifolds},
	\ref{miniremark:G-push-and-restriction}, and
	\ref{miniremark:def-rectifiable-G-chains}, we deduce, if $A$ is $\| S
	\|$ measurable and $T \in \mathscr R_m^{\textup{loc}} ( U, G )$, then,
	for $\mathscr L^\kappa$ almost all $y$,
	\begin{equation*}
		\langle S \restrict A, f, y \rangle = \langle S,f,y \rangle
		\restrict A, \quad \langle S+T, f, y \rangle = \langle S, f, y
		\rangle + \langle T, f, y \rangle.
	\end{equation*}
	Finally, taking $K_i = U \cap \{ x \with |x| \leq i, \dist (x, \mathbf
	R^n \without U ) \geq 1/i \}$, we will show that, whenever
	$\sum_{j=1}^\infty \sum_{i=1}^\infty \inf \{ 2^{-i}, \| S_j \| (K_i)
	\} < \infty$ for some $S_j \in \mathscr R_m^{\textup{loc}} ( U, G )$,
	there holds
	\begin{equation*}
		\lim_{j \to \infty} \langle S_j, f, y \rangle = 0 \quad
		\text{for $\mathscr L^\kappa$ almost all $y$};
	\end{equation*}
	in fact, whenever $\epsilon_i > 0$ satisfy $\epsilon_i \Lip ( f |
	K_i)^\kappa \leq 1$ and $\Phi : \mathbf R^\kappa \to \{ t \with 0 < t
	\leq 1 \}$ satisfies $\int \Phi \ud \mathscr L^\kappa = 1$, applying
	\ref{thm:coarea-rectifiable-varifolds}%
	\,\eqref{item:coarea-rectifiable-varifolds:int} yields
	\begin{equation*}
		{\textstyle\int} \inf \{ 2^{-i}, \epsilon_i \| \langle S_j, f,
		y \rangle \| (K_i) \} \Phi (y) \ud \mathscr L^\kappa \, y \leq
		\inf \{ 2^{-i}, \| S_j \| (K_i) \}.
	\end{equation*}
\end{miniremark}

\section{Integral chains}

\begin{miniremark} \label{miniremark:isomorphism-for-integers}
	Whenever $m$ is a nonnegative integer, $n$ is a positive integer, and
	$U$ is an open subset of $\mathbf R^n$, we employ the quotient map $p
	: \mathbf G_{\textup o} (n,m) \times \mathbf Z \to \mathbf G
	(n,m,\mathbf Z)$ and \cite[4.1.28]{MR41:1976} to define
	\begin{equation*}
		\iota_{U,m} : \mathscr R_m^{\textup{loc}} ( U ) \to \mathscr
		R_m^{\textup{loc}} ( U, \mathbf Z )
	\end{equation*}
	by letting $\iota_{U,m} ( Q ) \in \mathscr R_m^{\textup{loc}} ( U,
	\mathbf Z )$ contain $\tau : X \to \mathbf G (n,m, \mathbf Z )$ given
	by
	\begin{equation*}
		\tau (x) = p \big ( \vec Q(x), \boldsymbol \Uptheta^m ( \| Q
		\|, x ) \big ) \quad \text{for $x \in X$},
	\end{equation*}
	where $X = \{ x \with \text{$\boldsymbol \Uptheta^m ( \| Q \|, x )$ is
	a positive integer, $\vec Q(x) \in \mathbf G_{\textup o} (n,m)$} \}$,
	whenever $Q \in \mathscr R_m^{\textup{loc}} ( U )$; hence, $\|
	\iota_{U,m} ( Q ) \| = \| Q \|$, $\tau$ is adapted to $\| Q \|$, and
	$\iota_{U,m}$ yields an isomorphism of commutative groups.  For such
	$m$ and $U$, these isomorphisms have the following four properties
	whenever $Q \in \mathscr R_m^{\textup{loc}} ( U )$: Firstly, if $A$ is
	$\| Q \|$ measurable, then
	\begin{equation*}
		( \iota_{U,m} (Q) ) \restrict A = \iota_{U,m} ( Q \restrict A)
	\end{equation*}
	by \ref{miniremark:def-rectifiable-G-chains} and
	\cite[4.1.7]{MR41:1976}; secondly, if $\nu$ is a positive integer, $V$
	is an open subset of $\mathbf R^\nu$, $f : U \to V$ is locally
	Lipschitzian, and $f | \spt Q$ is proper, then
	\begin{equation*}
		f_\# ( \iota_{U,m} (Q) ) = \iota_{V,m} ( f_\# Q )
	\end{equation*}
	by the first property, \ref{miniremark:push-forward-G-chain}, and
	\cite[4.1.7, 4.1.14, 4.1.30]{MR41:1976}; thirdly, if $\mu$ and $\nu$
	are nonnegative integers, $\nu \geq 1$, $V$ is an open subset of
	$\mathbf R^\nu$, and $R \in \mathscr R_\mu^{\textup{loc}} ( V )$, then
	\begin{equation*}
		\iota_{U,m} ( Q ) \times \iota_{V,\mu} (R) = \iota_{U \times
		V, m+\mu} ( Q \times R )
	\end{equation*}
	by \ref{miniremark:product-G-chains} and \cite[4.1.8]{MR41:1976}; and,
	finally, if $\kappa$ is a positive integer, $\kappa \leq m$, and $f :
	U \to \mathbf R^\kappa$ is locally Lipschitzian, then
	\begin{equation*}
		\langle \iota_{U,m} (Q), f , y \rangle = \iota_{U,m-\kappa} (
		\langle Q,f,y \rangle ) \quad \text{for $\mathscr L^\kappa$
		almost all $y$}
	\end{equation*}
	by the first property, \ref{miniremark:slicing-G-chains}, and
	\cite[4.1.7, 4.3.1, 4.3.2\,(2), 4.3.8]{MR41:1976}.
	
	Next, whenever $m$ and $U$ are as before, we let
	\begin{equation*}
		\mathbf I_m^{\textup{loc}} ( U, \mathbf Z ) = \iota_{U,m} [
		\mathbf I_m^{\textup{loc}} ( U ) ], \quad \mathscr P_m ( U,
		\mathbf Z ) = \iota_{U,m} [ \mathscr P_m ( U ) ]
	\end{equation*}
	and, if $m \geq 1$, we employ $\boundary : \mathbf I_m^{\textup{loc}}
	( U ) \to \mathbf I_{m-1}^{\textup{loc}} ( U )$ to define the
	homomorphism
	\begin{equation*}
		{\boundary_{\mathbf Z}} : \mathbf I_m^{\textup{loc}} ( U,
		\mathbf Z ) \to \mathbf I_{m-1}^{\textup{loc}} ( U, \mathbf Z)
	\end{equation*}
	such that ${\boundary_{\mathbf Z}} \circ \iota_{U,m} = \iota_{U,m-1}
	\circ {\boundary}$.  Requiring the monomorphisms mapping $S \in
	\mathbf I_m^{\textup{loc}} ( U, \mathbf Z )$ onto
	\begin{equation*}
		\text{$(S,\boundary_{\mathbf Z}S) \in \mathscr
		R_m^{\textup{loc}} ( U, \mathbf Z ) \times \mathscr
		R_{m-1}^{\textup{loc}} ( U, \mathbf Z )$ if $m \geq 1$}, \quad
		\text{$S \in \mathscr R_m^{\textup{loc}} ( U, \mathbf Z )$ if
		$m=0$}
	\end{equation*}
	to be isometric, the groups $\mathbf I_m^{\textup{loc}} ( U, \mathbf
	Z)$ are endowed with the structure of complete normed commutative
	groups.  For $S \in \mathbf I_m^{\textup{loc}} ( U, \mathbf Z )$ with
	$m \geq 1$, we have
	\begin{equation*}
		\text{$\boundary_{\mathbf Z} ( \boundary_{\mathbf Z} S ) = 0$
		if $m \geq 2$}, \quad \spt \| \boundary_{\mathbf Z} S \|
		\subset \spt \| S \|,
	\end{equation*}
	if $\nu$ is a positive integer, $V$ is an open subset of $\mathbf
	R^\nu$, $f : U \to V$ is a locally Lipschitzian map, and $f | \spt \|
	S \|$ is proper, then we have $f_\# S \in \mathbf I_m^{\textup{loc}}
	(V, \mathbf Z )$ with $\boundary_{\mathbf Z} (f_\# S) = f_\#
	(\boundary_{\mathbf Z} S)$ by \cite[4.1.7]{MR41:1976}, and, if $f : U
	\to \mathbf R$ is locally Lipschitzian, then there holds
	\begin{gather*}
		S \restrict \{ x \with f(x) > y \} \in \mathbf
		I_m^{\textup{loc}} ( U, \mathbf Z ), \\
		\boundary_{\mathbf Z} ( S \restrict \{ x \with f(x)>y \} ) =
		\langle S,f,y \rangle + ( \boundary_{\mathbf Z} S ) \restrict
		\{ x \with f(x)>y \}
	\end{gather*}
	for $\mathscr L^1$ almost all $y$ by \cite[4.2.1, 4.3.1,
	4.3.4]{MR41:1976}.  From \cite[4.1.8]{MR41:1976}, we infer that, if
	additionally $\mu$ is a nonnegative integer, $\nu$ is a positive
	integer, and $V$ is an open subset of $\mathbf R^\nu$, we have $S
	\times T \in \mathbf I_{m+\mu}^{\textup{loc}} ( U \times V, \mathbf
	Z)$ and
	\begin{gather*}
		\text{$\boundary_{\mathbf Z} ( S \times T ) = (
		\boundary_{\mathbf Z} S ) \times T + (-1)^m \cdot ( S \times
		\boundary_{\mathbf Z} T )$ if $m > 0 < \mu$}, \\
		\text{$\boundary_{\mathbf Z} ( S \times T ) = (
		\boundary_{\mathbf Z} S ) \times T$ if $m > \mu = 0$}, \quad
		\text{$\boundary_{\mathbf Z} ( S \times T ) = S \times
		\boundary_{\mathbf Z} T$ if $m = 0 < \mu$}
	\end{gather*}
	whenever $S \in \mathbf I_m^{\textup{loc}} ( U, \mathbf Z )$ and $T
	\in \mathbf I_\mu^{\textup{loc}} ( V, \mathbf Z )$.  Finally, we let
	\begin{equation*}
		\mathbf I_m ( U, \mathbf Z ) = \mathbf I_m^{\textup{loc}} ( U,
		\mathbf Z ) \cap \{ S \with \text{$\spt \| S \|$ is compact}
		\}.
	\end{equation*}
\end{miniremark}

\begin{lemma} \label{lemma:tensor-G}
	Suppose $f : B \to A$ is a homomorphism of commutative groups.
	
	Then, the following two conditions are equivalent.
	\begin{enumerate}
		\item \label{item:tensor-G:cycle} Whenever $d$ is a
		nonnegative integer, the homomorphism
		\begin{equation*}
			f_d : B/d B \to A/d A,
		\end{equation*}
		induced by $f$, is univalent.
		\item \label{item:tensor-G:all} Whenever $G$ is a commutative
		group, the homomorphism $f \otimes \mathbf 1_G$ is univalent.
	\end{enumerate}
\end{lemma}

\begin{proof}
	The homomorphisms $f_d$ correspond to $f \otimes \mathbf 1_{\mathbf Z/
	d \mathbf Z}$ via the canonical isomorphisms $B / d B \simeq B \otimes
	( \mathbf Z / d \mathbf Z )$ and $A / d A \simeq A \otimes ( \mathbf Z
	/ d \mathbf Z )$ noted in \cite[Chapter II, \S\,3.6, Corollary 2 to
	Proposition 6]{MR1727844}.  If these homomorphisms are univalent, then
	so are the homomorphisms $f \otimes \mathbf 1_G$ whenever $G$ is a
	finitely generated commutative group by \cite[Chapter II, \S\,3.7,
	Proposition 7]{MR1727844} and \cite[Chapter VII, p.\,19, Theorem
	2]{MR1994218}, whence we deduce the validity of
	\eqref{item:tensor-G:all} by \cite[Chapter II, \S\,6.3, Corollary 4 to
	Proposition 7]{MR1727844}.
\end{proof}

\begin{remark} \label{remark:tensor-G}
	The conditions imply that $f$ is univalent but the converse does not
	hold; in fact, one may take $A = \mathbf Z$, $B = 2 \mathbf Z$, $f$
	the inclusion map, and $G = \mathbf Z / 2 \mathbf Z$ by \cite[Chapter
	II, \S\,3.6, Remark to Corollary to Proposition 5]{MR1727844}.
\end{remark}

\begin{example} \label{example:monomorphisms}
	If $f$ is the inclusion map of a pair $(A,B)$ and the conditions of
	\ref{lemma:tensor-G} hold, then we shall identify $B \otimes G$ with
	the subset $(f \otimes \mathbf 1_G) [ B \otimes G]$ of $A \otimes G$.
	We will prove that, whenever $m$ is a nonnegative integer, $C \subset
	U \subset \mathbf R^n$, $U$ is open, $C$ is closed relative to $U$, we
	may take the pair $(A,B)$ to equal
	\begin{gather*}
		( \mathbf I_m ( U, \mathbf Z ), \mathscr P_m ( U, \mathbf Z)),
		\quad \big ( \mathbf I_m^{\textup{loc}} ( U, \mathbf Z ),
		\mathbf I_m ( U, \mathbf Z ) \big), \quad \big ( \mathscr
		R_m^{\textup{loc}} ( U, \mathbf Z ), \mathscr R_m ( U, \mathbf
		Z) \big ), \\
		\big ( \mathbf I_m^{\textup{loc}} ( U, \mathbf Z ), \mathbf
		I_m^{\textup{loc}} ( U, \mathbf Z ) \cap \{ S \with \spt \| S
		\| \subset C \} \big ), \\
		( \mathscr R_m ( U, \mathbf Z ), \mathbf I_m ( U, \mathbf Z)),
		\quad \text{or} \quad \big ( \mathscr R_m^{\textup{loc}} ( U,
		\mathbf Z), \mathbf I_m^{\textup{loc}} ( U, \mathbf Z) \big).
	\end{gather*}
	We recall \ref{miniremark:isomorphism-for-integers}.  Then, concerning
	the first pair, we verify that, if $d$ is a positive integer, $Q \in
	\mathbf I_m ( U )$, and $d Q \in \mathscr P_m ( U )$, then $Q \in
	\mathscr P_m ( U )$, by employing the representation of the image of
	$d Q$ in $\mathscr P_m ( \mathbf R^n )$ by oriented convex cells
	obtained in \cite[4.1.32]{MR41:1976}; concerning the last two pairs,
	localising by means of slicing in the case of the last pair, we
	similarly make use of \cite[4.2.16\,(2)]{MR41:1976}; and the remaining
	pairs trivially satisfy the conditions.
\end{example}

\begin{miniremark} \label{miniremark:bilinear}
	Suppose $m$ is a nonnegative integer, $n$ is a positive integer, $U$
	is an open subset of $\mathbf R^n$, and $G$ is a complete normed
	commutative group.  Then, we may define a bilinear operation from
	$\mathscr R_m^{\mathrm{loc}} ( U, \mathbf Z ) \times G$ into $\mathscr
	R_m^{\mathrm{loc}} ( U, G )$ by requiring that
	\begin{equation*}
		S \cdot g \in \mathscr R_m^{\mathrm{loc}} ( U, G)
	\end{equation*}
	contains $\sigma \cdot g$ whenever $\sigma \in S \in \mathscr
	R_m^{\textup{loc}} ( U, \mathbf Z )$ and $g \in G$; hence, $\| S \cdot
	g \| \leq | g | \| S \|$ with equality if $\boldsymbol \Uptheta^m ( \|
	S \|, x ) = 1$ for $\| S \|$ almost all $x$.  Whenever $S \in \mathscr
	R_m^{\textup{loc}} ( U, \mathbf Z )$ and $g \in G$, we notice the
	following four properties. If $A$ is $\| S \|$ measurable, then $( S
	\cdot g ) \restrict A = ( S \restrict A) \cdot g$; if $\nu$ is a
	positive integer, $V$ is an open subset of $\mathbf R^\nu$, and $f : U
	\to V$ is a locally Lipschitzian map such that $f | \spt \| S \|$ is
	proper, then $f_\# ( S \cdot g ) = ( f_\# S ) \cdot g$; if $\mu$ is a
	nonnegative integer, $\nu$ is a positive integer, $V$ is an open
	subset of $\mathbf R^\nu$, and $T \in \mathscr R_\mu^{\textup{loc}} (
	V, \mathbf Z )$, then $( S \times T ) \cdot g = S \times ( T \cdot
	g)$; and, if $\kappa$ is a positive integer, $\kappa \leq m$, and $f :
	U \to \mathbf R^\kappa$ is locally Lipschitzian, then $\langle S \cdot
	g, f , y \rangle = \langle S,f,y \rangle \cdot g$ for $\mathscr
	L^\kappa$ almost all $y$.
	
	Next, recalling \ref{example:monomorphisms}, we will study the
	homomorphism
	\begin{equation*}
		\rho_{U,m,G} : \mathscr R_m^{\textup{loc}} ( U, \mathbf Z)
		\otimes G \to \mathscr R_m^{\textup{loc}} ( U, G),
	\end{equation*}
	characterised by $\rho_{U,m,G} ( S \otimes g ) = S \cdot g$ for $S \in
	\mathscr R_m^{\textup{loc}} ( U, \mathbf Z )$ and $g \in G$.  Clearly,
	$\rho_{U,m,\mathbf Z}$ is the canonical isomorphism $\mathscr
	R_m^{\textup{loc}} ( U, \mathbf Z ) \otimes \mathbf Z \simeq \mathscr
	R_m^{\textup{loc}} ( U, \mathbf Z )$,
	\begin{equation*}
		\rho_{U,m,G} [ \mathscr R_m^{\textup{loc}} ( U, \mathbf Z )
		\otimes G ] = \mathscr R_m^{\textup{loc}} ( U, G ) \quad
		\text{if $G$ is finite},
	\end{equation*}
	and $\rho_{U,0,G} [ \mathscr R_0 ( U, \mathbf Z ) \otimes G ] =
	\mathscr R_0 ( U, G )$.  In general, noting $\rho_{U,m,G} [ \mathscr
	R_m^{\textup{loc}} ( U, \mathbf Z) \otimes G ]$ is dense in $\mathscr
	R_m^{\textup{loc}} ( U, G)$ and that $\mathbf I_m ( U, \mathbf Z )$ is
	dense in $\mathscr R_m^{\textup{loc}} ( U , \mathbf Z)$, we obtain
	that
	\begin{equation*}
		\text{$\rho_{U,m,G} [ \mathbf I_m ( U, \mathbf Z ) \otimes G
		]$ is dense in $\mathscr R_m^{\textup{loc}} ( U, G)$}.
	\end{equation*}
	Whenever $S \in \rho_{U,m,G} [ \mathbf I_m^{\textup{loc}} ( U, \mathbf
	Z ) \otimes G ]$ and $C$ is a relatively closed neighbourhood of $\spt
	\| S \|$ in $U$, there holds $S \in \rho_{U,m,G} \big [ ( \mathbf
	I_m^{\textup{loc}} ( U, \mathbf Z ) \cap \{ T \with \spt \| T \|
	\subset C \}) \otimes G \big ]$; in fact, this is readily verified
	using \ref{corollary:separation-smooth-Tietze} and
	\ref{miniremark:isomorphism-for-integers}.  Finally, we let
	\begin{equation*}
		\mathscr P_m ( U, G ) = \rho_{U,m,G} [ \mathscr P_m ( U,
		\mathbf Z ) \otimes G ].
	\end{equation*}
\end{miniremark}

\begin{theorem} \label{thm:rho-mono}
	Suppose $m$ is a nonnegative integer, $n$ is a positive integer, $U$
	is an open subset of $\mathbf R^n$, $G$ is a complete normed
	commutative group, and $\rho_{U,m,G}$ is as in
	\ref{miniremark:bilinear}.
	
	Then, $\rho_{U,m,G}$ is a monomorphism.
\end{theorem}

\begin{proof}
	Suppose $\xi \in \ker \rho_{U,m,G}$.  Then, for some finitely
	generated subgroup $H$ of $G$, we have $\xi \in \im ( \mathbf
	1_{\mathscr R_m^{\textup{loc}} ( U, \mathbf Z )} \otimes i)$, where $i
	: H \to G$ is the inclusion map.  In view of \cite[Chapter VII,
	p.\,19, Theorem 2]{MR1994218}, there exist integers $r$ and $s$ with
	$0 \leq r \leq s$ and integers $d_t$ with $d_t \geq 2$ such that
	\begin{equation*}
		A = \bigoplus_{t=1}^r ( \mathbf Z / d_t \mathbf Z ) \oplus
		\bigoplus_{t=r+1}^s \mathbf Z \simeq H \quad \text{(as
		commutative groups)}.
	\end{equation*}
	Choosing $h_t \in H$ corresponding to a generator of the $t$-th
	summand of $A$ under this isomorphism for $t = 1, \ldots, s$, we
	express
	\begin{equation*}
		\xi = \sum_{t=1}^s S_t \otimes i(h_t) \quad \text{for some
		$S_t \in \mathscr R_m^{\textup{loc}} ( U, \mathbf Z)$}.
	\end{equation*}
	Selecting $\sigma_t$ in $S_t$ adapted to $\phi = \sum_{t=1}^s \| S_t
	\|$, we infer that $\sum_{t=1}^s \sigma_t \cdot i(h_t)$ in $0 \in
	\mathscr R_m^{\textup{loc}} ( U, \mathbf Z )$ is adapted to $\phi$ by
	\ref{miniremark:def-rectifiable-G-chains}.  The preceding isomorphism
	then yields
	\begin{equation*}
		\sigma_t(x) \cdot i(h_t) = 0 \in \boldsymbol \uppi_{\mathbf
		G(n,m,G)}^{-1} [ \{ \Tan^m ( \phi, x ) \} ] \quad \text{for
		$t=1,\ldots,s$}
	\end{equation*}
	for $\phi$ almost all $x$; hence, we have
	\begin{gather*}
		\sigma_t(x) \in d_t \boldsymbol \uppi_{\mathbf G(n,m, \mathbf
		Z)}^{-1} [ \{ \Tan^m ( \phi, x ) \} ] \quad \text{for $t = 1,
		\ldots, r$}, \\
		\sigma_t(x) = 0 \in \boldsymbol \uppi_{\mathbf G(n,m, \mathbf
		Z)}^{-1} [ \{ \Tan^m ( \phi, x ) \} ] \quad \text{for $t =
		r+1, \ldots, s$}
	\end{gather*}
	for such $x$.  Finally, for $t=1, \ldots, r$, we infer $S_t \otimes
	i(h_t) = 0$, as $S_t \in d_t \mathscr R_m^{\textup{loc}} ( U, \mathbf
	Z )$ and $d_t \cdot h_t = 0$, whereas, for $t=r+1, \ldots, s$, we
	clearly have $S_t = 0$.
\end{proof}

\begin{remark}
	Since $S = 0$ or $d = 0$ whenever $S \in \mathscr R_m^{\textup{loc}} (
	U, \mathbf Z )$, $d \in \mathbf Z$, and $d \cdot S = 0$, we infer from
	\cite[Chapter I, \S\,2.3, Proposition 1]{MR1727221} that $\mathbf
	1_{\mathscr R_m^{\textup{loc}} ( U, \mathbf Z)} \otimes i$ is a
	monomorphism.  In view of \ref{example:torus-subgroup}, we also note
	that $H$ may fail to be isomorphic to $A$ as normed group, when $A$ is
	endowed with the standard group norm.
\end{remark}

\begin{remark} \label{remark:isomorphisms-finite-group}
	If $G$ is finite, then $\rho_{n,m,G}$ accordingly induces isomorphisms
	\begin{equation*}
		\mathscr R_m ( U, \mathbf Z ) \otimes G \simeq \mathscr R_m (
		U, G ) \quad \text{and} \quad \mathscr P_m ( U, \mathbf Z )
		\otimes G \simeq \mathscr P_m ( U, G )
	\end{equation*}
	by \ref{miniremark:def-rectifiable-G-chains},
	\ref{example:monomorphisms}, and \ref{miniremark:bilinear}.
\end{remark}

\begin{corollary} \label{corollary:boundary-operator-dense-subset}
	Whenever $m$ is a positive integer, there exists a unique homomorphism
	(see \ref{miniremark:isomorphism-for-integers} and
	\ref{miniremark:bilinear})
	\begin{equation*}
		\boundary_G : \rho_{U,m,G} [ \mathbf I_m^{\textup{loc}} ( U,
		\mathbf Z ) \otimes G ] \to \rho_{U,m-1,G} [ \mathbf
		I_{m-1}^{\textup{loc}} ( U, \mathbf Z ) \otimes G ]
	\end{equation*}
	such that $\boundary_G ( S \cdot g ) = ( \boundary_{\mathbf Z} S )
	\cdot g$ for $S \in \mathbf I_m^{\textup{loc}} ( U, \mathbf Z )$ and
	$g \in G$.  Moreover, in the case $G = \mathbf Z$, this homomorphism
	agrees with
	\begin{equation*}
		\boundary_{\mathbf Z} : \mathbf I_m^{\textup{loc}} ( U,
		\mathbf Z ) \to \mathbf I_{m-1}^{\textup{loc}} ( U, \mathbf
		Z).
	\end{equation*}
\end{corollary}

\begin{proof}
	In view of \ref{thm:rho-mono}, it is sufficient for the principal
	conclusion to recall from \ref{example:monomorphisms} that the
	canonical homomorphism from $\mathbf I_m^{\textup{loc}} ( U, \mathbf
	Z) \otimes G$ into $\mathscr R_m^{\textup{loc}} ( U, \mathbf Z )
	\otimes G$ is univalent.  The postscript is then readily verified by
	means of \ref{miniremark:bilinear}.
\end{proof}

\begin{remark} \label{remark:boundary-dense-subset}
	Employing \ref{miniremark:isomorphism-for-integers} and
	\ref{miniremark:bilinear} with suitable relatively closed
	neighbourhoods $C$ of $\spt \| S \|$ in $U$, we verify that, whenever
	$S \in \rho_{U,m,G} [ \mathbf I_m^{\textup{loc}} ( U, \mathbf Z )
	\otimes G ]$, we have $\boundary_G ( \boundary_G S ) = 0$ if $m \geq
	2$, $\spt \| \boundary_G S \| \subset \spt \| S \|$, and
	\begin{equation*}
		f_\# S \in \rho_{V,m,G} [ \mathbf I_m^{\textup{loc}} ( V,
		\mathbf Z ) \otimes G ] \quad \text{with} \quad \boundary_G
		(f_\# S) = f_\# (\boundary_G S)
	\end{equation*}
	whenever $\nu$ is a positive integer, $V$ is an open subset of
	$\mathbf R^\nu$, and $f : U \to V$ is a locally Lipschitzian map such
	that $f | \spt \| S \|$ is proper; in fact, for the last equation,
	suitability of $C$ amounts to properness of $f|C$.  Finally, from
	\ref{miniremark:isomorphism-for-integers} and
	\ref{miniremark:bilinear}, we infer that, if $\mu$ is a nonnegative
	integer, $\nu$ is a positive integer, $V$ is an open subset of
	$\mathbf R^\nu$, then $S \times T \in \rho_{U \times V, m+\mu, G } [
	\mathbf I_{m+\mu}^{\textup{loc}} ( U \times V, \mathbf Z) \otimes G ]$
	and
	\begin{gather*}
		\text{$\boundary_G ( S \times T ) = ( \boundary_{\mathbf Z} S)
		\times T + (-1)^m \cdot ( S \times \boundary_G T )$ if $m > 0
		< \mu$}, \\
		\text{$\boundary_G ( S \times T ) = ( \boundary_{\mathbf Z} S)
		\times T$ if $m > \mu = 0$}, \quad \text{$\boundary_G ( S
		\times T ) = S \times \boundary_G T$ if $m = 0 < \mu$}
	\end{gather*}
	for $S \in \mathbf I_m^{\textup{loc}} ( U, \mathbf Z )$ and $T \in
	\rho_{V,\mu,G} [ \mathbf I_\mu^{\textup{loc}} ( V, \mathbf Z ) \otimes
	G ]$, and, recalling \ref{miniremark:def-rectifiable-G-chains} and
	\ref{miniremark:slicing-G-chains}, that, if $f : U \to \mathbf R$ is
	locally Lipschitzian and $S \in \rho_{U,m,G} [ \mathbf
	I_m^{\textup{loc}} ( U, \mathbf Z ) \otimes G ]$, then
	\begin{gather*}
		S \restrict \{ x \with f(x) > y \} \in \rho_{U,m,G} [ \mathbf
		I_m^{\textup{loc}} (U, \mathbf Z ) \otimes G ], \\
		\boundary_G ( S \restrict \{ x \with f(x)>y \} ) = \langle
		S,f,y \rangle + ( \boundary_G S ) \restrict \{ x \with f(x)>y
		\}
	\end{gather*}
	for $\mathscr L^1$ almost all $y$.
\end{remark}

\begin{definition} \label{definition:loc-integral-chain}
	Suppose $m$ is a nonnegative integer, $n$ is a positive integer, $U$
	is an open subset of $\mathbf R^n$, and $G$ is a complete normed
	commutative group.
	
	Then, we define the complete normed commutative group $\mathbf
	I_m^{\textup{loc}} ( U, G )$ of $m$ dimensional \emph{locally integral
	$G$ chains} in $U$ to be the subgroup (see \ref{miniremark:bilinear})
	\begin{equation*}
		\Clos \big \{ ( S, \boundary_G S ) \with S \in \rho_{U,m,G} [
		\mathbf I_m^{\textup{loc}} ( U, \mathbf Z ) \otimes G ] \big
		\}
	\end{equation*}
	of $\mathscr R_m^{\textup{loc}} ( U, G ) \times \mathscr
	R_{m-1}^{\textup{loc}} ( U, G )$, if $m \geq 1$, and $\mathbf
	I_m^{\textup{loc}} ( U, G ) = \mathscr R_m^{\textup{loc}} ( U, G )$,
	if $m = 0$.  In the case $m \geq 1$, we recall
	\ref{corollary:boundary-operator-dense-subset} and
	\ref{remark:boundary-dense-subset} to define the continuous
	homomorphism
	\begin{equation*}
		\boundary_G : \mathbf I_m^{\textup{loc}} ( U, G ) \to \mathbf
		I_{m-1}^{\textup{loc}} ( U, G )
	\end{equation*}
	by $\boundary_G ( S,T ) = (T,0)$ if $m \geq 2$ and $\boundary_G ( S,T)
	= T$ if $m = 1$ for $(S,T) \in \mathbf I_m^{\textup{loc}} ( U, G )$.
\end{definition}

\begin{remark} \label{remark:locally-integral-chains}
	Defining the monomorphisms (see
	\ref{corollary:boundary-operator-dense-subset})
	\begin{gather*}
		\kappa_{U,m,G} : \rho_{U,m,G} [ \mathbf I_m^{\textup{loc}} (
		U, \mathbf Z ) \otimes G ] \to \mathbf I_m^{\textup{loc}} ( U,
		G ), \\
		\text{$\kappa_{U,m,G} ( S ) = ( S, \boundary_G S )$ if $m \geq
		1$}, \quad \text{$\kappa_{U,m,G} (S) = S$ if $m=0$},
	\end{gather*}
	whenever $S \in \rho_{U,m,G} [ \mathbf I_m^{\textup{loc}} ( U, \mathbf
	Z ) \otimes G ]$, we employ \ref{remark:boundary-dense-subset} to
	verify, if $m \geq 1$, then
	\begin{equation*}
		\boundary_G \kappa_{U,m,G} (S) = \kappa_{U,m-1,G} (
		\boundary_G S) \quad \text{for $S \in \rho_{U,m,G} [ \mathbf
		I_m^{\textup{loc}} ( U, \mathbf Z ) \otimes G ]$}.
	\end{equation*}
	In the case $G = \mathbf Z$, the map $\kappa_{U,m,G}$ is an isometry
	between $\mathbf I_m^{\textup{loc}} ( U, \mathbf Z )$, as defined in
	\ref{miniremark:isomorphism-for-integers}, and $\mathbf
	I_m^{\textup{loc}} ( U, G )$, as defined in
	\ref{definition:loc-integral-chain}, by
	\ref{corollary:boundary-operator-dense-subset}.  In the general case,
	the next theorem will allow us to subsequently identify $\mathbf
	I_m^{\textup{loc}} ( U, G)$ with a dense subgroup of $\mathscr
	R_m^{\textup{loc}} ( U, G)$ by showing that the homomorphism mapping
	$(S,T) \in \mathbf I_m^{\textup{loc}} ( U, G )$ onto $S \in \mathscr
	R_m^{\textup{loc}} ( U, G )$ is univalent; the isometric isomorphism
	$\kappa_{U,m,\mathbf Z}$ will then correspond to the identity map on
	$\mathbf I_m^{\textup{loc}} ( U, \mathbf Z )$ justifying our notation
	for $G = \mathbf Z$.
\end{remark}

\begin{theorem} \label{thm:integral-chains}
	Suppose $m$ and $n$ are positive integers, $U$ is an open subset
	of~$\mathbf R^n$, and $G$ is a complete normed commutative group.
	
	Then, the following eight statements hold.
	\begin{enumerate}
		\item \label{item:integral-chains:mono} If $(0,T) \in \mathbf
		I_m^{\textup{loc}} ( U, G )$, then $T = 0$.
		\setcounter{enumi_memory}{\value{enumi}}
	\end{enumerate}
	Henceforward, \emph{we will identify $\mathbf I_m^{\textup{loc}} ( U,
	G )$ with the image of the monomorphism mapping $(S,T) \in \mathbf
	I_m^{\textup{loc}} ( U, G )$ onto $S \in \mathscr R_m^{\textup{loc}} (
	U, G )$.}
	\begin{enumerate}
		\setcounter{enumi}{\value{enumi_memory}}
		\item \label{item:integral-chains:dense} The subgroup $\mathbf
		I_m^{\textup{loc}} ( U, G )$ is dense in $\mathscr
		R_m^{\textup{loc}} (U,G)$.
		\item \label{item:integral-chains:extension} The subgroup
		$\rho_{U,m,G} [ \mathbf I_m^{\textup{loc}} ( U, \mathbf Z )
		\otimes G ]$, see \ref{miniremark:bilinear}, is dense in
		$\mathbf I_m^{\textup{loc}} ( U, G )$ and $\boundary_G :
		\mathbf I_m^{\textup{loc}} (U,G) \to \mathbf
		I_{m-1}^{\textup{loc}} ( U,G )$ yields the unique continuous
		extension of $\boundary_G : \rho_{U,m,G} [ \mathbf
		I_m^{\textup{loc}} ( U, \mathbf Z ) \otimes G ] \to
		\rho_{U,m-1,G} [ \mathbf I_{m-1}^{\textup{loc}} ( U, \mathbf
		Z) \otimes G ]$, see
		\ref{corollary:boundary-operator-dense-subset}.
		\item \label{item:integral-chains:complex} If $m \geq 2$, then
		$\boundary_G ( \boundary_G S ) = 0$ for $S \in \mathbf
		I_m^{\textup{loc}} ( U,G )$.
		\item \label{item:integral-chains:spt} For $S \in \mathbf
		I_m^{\textup{loc}} ( U,G )$, we have $\spt \| \boundary_G S \|
		\subset \spt \| S \|$.
		\item \label{item:integral-chains:push-forward} If $\nu$ is a
		positive integer, $V$ is an open subset of $\mathbf R^\nu$, $f
		: U \to V$ is locally Lipschitzian, $S \in \mathbf
		I_m^{\textup{loc}} ( U, G )$, and $f | \spt \| S \|$ is
		proper, then $f_\# S \in \mathbf I_m^{\textup{loc}} ( V,G )$
		and $\boundary_G ( f_\# S ) = f_\# ( \boundary_G S)$.
		\item \label{item:integral-chains:product} If $S \in \mathbf
		I_m^{\textup{loc}} ( U, \mathbf Z )$, $\mu$ is a nonnegative
		integer, $\nu$ is a positive integer, $V$ is an open subset of
		$\mathbf R^\nu$, and $T \in \mathbf I_\mu^{\textup{loc}} ( V,
		G )$, then $S \times T \in \mathbf I_{m+\mu}^{\textup{loc}} (
		U \times V, G )$ and
		\begin{gather*}
			\text{$\boundary_G ( S \times T ) = (
			\boundary_{\mathbf Z} S) \times T + (-1)^m \cdot ( S
			\times \boundary_G T )$ if $m > 0 < \mu$}, \\
			\text{$\boundary_G ( S \times T ) = (
			\boundary_{\mathbf Z} S) \times T$ if $m > \mu = 0$},
			\\
			\text{$\boundary_G ( S \times T ) = S \times
			\boundary_G T$ if $m = 0 < \mu$}.
		\end{gather*}
		\item \label{item:integral-chains:slice} If $S \in \mathbf
		I_m^{\textup{loc}} ( U, G )$ and $f : U \to \mathbf R$ is
		locally Lipschitzian, then there holds
		\begin{gather*}
			S \restrict \{ x \with f(x)>y \} \in \mathbf
			I_m^{\textup{loc}} ( U, G ), \\
			\boundary_G ( S \restrict \{ x \with f(x)>y \} ) =
			\langle S,f,y \rangle + ( \boundary_G S ) \restrict \{
			x \with f(x)>y \}
		\end{gather*}
		for $\mathscr L^1$ almost all $y$.
		\setcounter{enumi_memory}{\value{enumi}}
	\end{enumerate}
\end{theorem}

\begin{proof}
	\eqref{item:integral-chains:complex} is trivial.  We will prove the
	following three assertions.
	\begin{enumerate}
		\setcounter{enumi}{\value{enumi_memory}}
		\item \label{item:integral-chains:pre-slice} If $(S,T) \in
		\mathbf I_m^{\textup{loc}} ( U, G )$ and $f : U \to \mathbf R$
		is locally Lipschitzian, then
		\begin{equation*}
			( S \restrict \{ x \with f(x)>y \}, \langle S,f,y
			\rangle + T \restrict \{ x \with f(x)>y \} ) \in
			\mathbf I_m^{\textup{loc}} ( U, G )
		\end{equation*}
		for $\mathscr L^1$ almost all $y$.
		\item \label{item:integral-chains:neighbourhood} If $(S,T) \in
		\mathbf I_m^{\textup{loc}} (U,G)$ and $C$ is a neighbourhood
		of $\spt ( \| S \| + \| T \| )$ which is relatively closed in
		$U$, then
		\begin{equation*}
			(S,T) \in \Clos \big \{ (R,\boundary_G R) \with R \in
			\rho_{U,m,G} [ \mathbf I_m^{\textup{loc}} ( U, \mathbf
			Z ) \otimes G ], \spt \| R \| \subset C \big \}.
		\end{equation*}
		\item \label{item:integral-chains:pre-push-forward} If $(S,T)
		\in \mathbf I_m^{\textup{loc}} ( U, G )$, $\nu$ is a positive
		integer, $V$ is an open subset of $\mathbf R^\nu$, $f : U \to
		V$ is locally Lipschitzian, and $f | \spt ( \| S \| + \| T
		\|)$ is proper, then $(f_\#S, f_\#T) \in \mathbf
		I_m^{\textup{loc}} ( V, G )$.
	\end{enumerate}
	For \eqref{item:integral-chains:pre-slice} and
	\eqref{item:integral-chains:pre-push-forward}, the special case that
	$S \in \rho_{U,m,G} [ \mathbf I_m^{\textup{loc}} ( U, \mathbf Z )
	\otimes G]$ and $T = \boundary_G S$ was noted in
	\ref{remark:boundary-dense-subset}.  The general case of
	\eqref{item:integral-chains:pre-slice} then follows by approximation
	by means of \ref{miniremark:def-rectifiable-G-chains} and
	\ref{miniremark:slicing-G-chains}.  Applying
	\ref{corollary:separation-smooth-Tietze} with $E_0 = U \without \Int
	C$ and $E_1 = \spt ( \| S \| + \| T \| )$, we deduce
	\eqref{item:integral-chains:neighbourhood} from
	\ref{remark:boundary-dense-subset}, again using
	\ref{miniremark:def-rectifiable-G-chains} and
	\ref{miniremark:slicing-G-chains}.  Selecting a relatively closed
	neighbourhood $C$ of $\spt ( \| S \| + \| T \| )$ in $U$ such that $f
	| C$ is proper, the general case of
	\eqref{item:integral-chains:pre-push-forward} follows from the special
	case by approximation based on \ref{miniremark:push-forward-G-chain},
	\ref{remark:boundary-dense-subset}, and
	\eqref{item:integral-chains:neighbourhood}.
	
	To prove \eqref{item:integral-chains:mono} in the case $m > n$, it
	suffices to note that $\mathbf I_m^{\textup{loc}} ( U, \mathbf Z )$
	and thus also $\im \kappa_{U,m,G}$ and its closure $\mathbf
	I_m^{\textup{loc}} ( U,G )$ are trivial groups.

	To prove \eqref{item:integral-chains:mono} in the case $m \leq n$, we
	suppose there were $(0,T) \in \mathbf I_m^{\textup{loc}} ( U, G )$
	such that $T \neq 0$.  Six further properties could be assumed.
	Firstly, compactness of $\spt \| T \|$  by
	\eqref{item:integral-chains:pre-slice}; secondly, $U = \mathbf R^n$ by
	\eqref{item:integral-chains:pre-push-forward}; thirdly, $\| T \| ( M )
	> 0$, where $M = \{ x + g(x) \with x \in P \}$ for some $P \in \mathbf
	G (n,m-1)$ and $g : P \to P^\perp$ with $\Lip g < \infty$, see
	\cite[2.10.43, 3.1.19\,(5), 3.2.29]{MR41:1976}; fourthly, $M = P$ by
	\eqref{item:integral-chains:pre-push-forward}; fifthly, $\| T \| ( M)
	> \| T \| ( \mathbf R^n \without M )$ by applying
	\eqref{item:integral-chains:pre-slice} with $f(x)$ replaced by $|x-a|$
	for some $a \in \mathbf R^n$ with
	\begin{equation*}
		\lim_{r \to 0+} \frac{\| T \| ( \mathbf B (a,r) \without
		M)}{\| T \| \, \mathbf B (a,r)} = 0,
	\end{equation*}
	see \cite[2.8.18, 2.9.11]{MR41:1976}; and, sixthly, $\spt \| T \|
	\subset M$ by applying \eqref{item:integral-chains:pre-push-forward}
	with $f$ replaced by $M_\natural$, as $(M_\natural)_\# T \neq 0$ would
	be ensured by the facts
	\begin{equation*}
		(M_\natural)_\# ( T \restrict M) = T \restrict M \quad
		\text{and} \quad \| ( M_\natural)_\# ( T \restrict \mathbf R^n
		\without M ) \| ( \mathbf R^n ) \leq \| T \| ( \mathbf R^n
		\without M),
	\end{equation*}
	see \ref{miniremark:def-rectifiable-G-chains} and
	\ref{miniremark:push-forward-G-chain}.  Then, taking a compact
	neighbourhood $C$ of $\spt \| T \|$, we could employ
	\eqref{item:integral-chains:neighbourhood} to secure $R_i \in
	\rho_{U,m,G} [ \mathbf I_m^{\textup{loc}} ( \mathbf R^n, \mathbf Z )
	\otimes G ]$ with $\spt \| R_i \| \subset C$ for every positive
	integer $i$ such that
	\begin{equation*}
		\lim_{i \to \infty} (R_i, \boundary_G R_i) = (0,T) \quad
		\text{in $\mathscr R_m^{\textup{loc}} ( \mathbf R^n, G )
		\times \mathscr R_{m-1}^{\textup{loc}} ( \mathbf R^n, G )$}
	\end{equation*}
	and, recalling \ref{miniremark:push-forward-G-chain}, additionally
	require $\spt \| R_i \| \subset M$ by applying
	\ref{miniremark:push-forward-G-chain} and
	\ref{remark:boundary-dense-subset} with $f$ replaced by $M_\natural$.
	Since $M \in \mathbf G (n,m-1)$, this would imply $R_i=0$ and thus
	$\boundary_G R_i = 0$ for every positive integer $i$, in contradiction
	to $T \neq 0$.
	
	Having established \eqref{item:integral-chains:mono}, the remaining
	conclusions pose no difficulty: \eqref{item:integral-chains:dense} is
	implied by \ref{miniremark:bilinear};
	\ref{remark:locally-integral-chains} yields
	\eqref{item:integral-chains:extension};
	\eqref{item:integral-chains:spt} may be inferred from
	\eqref{item:integral-chains:pre-slice};
	\eqref{item:integral-chains:push-forward} follows from
	\eqref{item:integral-chains:spt} and
	\eqref{item:integral-chains:pre-push-forward};
	\eqref{item:integral-chains:product} is deduced from
	\ref{miniremark:product-G-chains}, \ref{remark:boundary-dense-subset},
	and \eqref{item:integral-chains:extension}; finally,
	\eqref{item:integral-chains:slice} reduces to
	\eqref{item:integral-chains:pre-slice}.
\end{proof}

\begin{remark} \label{remark:Fleming-inspiration}
	The proof of \eqref{item:integral-chains:mono} was inspired by
	\cite[p.\ 163]{MR185084}.
\end{remark}

\begin{remark} \label{remark:integral-chains}
	We moreover record that, if $S \in \mathbf I_m^{\textup{loc}} ( U, G)$
	and $C$ is a neighbourhood of $\spt \| S \|$ which is relatively
	closed in $U$, then $S$ belongs to the closure of $\rho_{U,m,G} \big [
	\mathbf I_m^{\textup{loc}} ( U, \mathbf Z ) \otimes G \big ] \cap \{ T
	\with \spt \| T \| \subset C \} $ in $\mathbf I_m^{\textup{loc}} ( U,
	G)$ by \eqref{item:integral-chains:spt} and
	\eqref{item:integral-chains:neighbourhood}.
\end{remark}

\begin{corollary} \label{corollary:integral-chains:homotopy-formula}
	Suppose additionally $\nu$ is a positive integer, $V$ is an open
	subset of $\mathbf R^\nu$, $A$ is an open subinterval of $\mathbf R$,
	$0 \in A$, $t \in A$, $h : A \times U \to V$ is locally Lipschitzian,
	$f : U \to V$ and $g : U \to V$ satisfy
	\begin{equation*}
		f(x) = h(0,x) \quad \text{and} \quad g(x) = h (t,x) \quad
		\text{for $x \in U$},
	\end{equation*}
	$S \in \mathbf I_m^{\textup{loc}} ( U, G )$, and $h | ( \spt
	\boldsymbol [ 0,t \boldsymbol ] \times \spt \| S \|)$ is proper.
	
	Then, there holds $h_\# ( \boldsymbol [ 0,t \boldsymbol ] \times S )
	\in \mathbf I_{m+1}^{\textup{loc}} (V,G)$ and
	\begin{gather*}
		\text{$g_\# S - f_\# S = \boundary_G h_\# ( \boldsymbol [ 0,t
		\boldsymbol ] \times S ) + h_\# ( \boldsymbol [ 0,t
		\boldsymbol ] \times \boundary_G S )$ if $m \geq 1$}, \\
		\text{$g_\# S - f_\# S = \boundary_G h_\# ( \boldsymbol [ 0,t
		\boldsymbol ] \times S )$ if $m=0$};
	\end{gather*}
	here, $\boldsymbol [ 0,t \boldsymbol ]$ in $\mathbf I_1^{\textup{loc}}
	( A )$ is identified with $\iota_{A,1} ( \boldsymbol [ 0,t \boldsymbol
	])$ in $\mathbf I_1^{\textup{loc}} ( A, \mathbf Z )$, see
	\ref{miniremark:isomorphism-for-integers}.
\end{corollary}

\begin{proof}
	Recalling \ref{miniremark:push-forward-G-chain} and
	\ref{miniremark:product-G-chains}, it suffices to compute the boundary
	of the chain $h_\# ( \boldsymbol [ 0,t \boldsymbol ] \times S ) \in
	\mathbf I_{m+1}^{\textup{loc}} ( V, G )$ by means of
	\ref{thm:integral-chains}\,\eqref{item:integral-chains:spt}%
	\,\eqref{item:integral-chains:push-forward}%
	\,\eqref{item:integral-chains:product}.
\end{proof}

\begin{definition} \label{definition:integral-chain}
	Suppose $m$ and $n$ are nonnegative integers, $n \geq 1$, $U$ is an
	open subset of $\mathbf R^n$, and $G$ is a complete normed commutative
	group.
	
	Then, we let $\mathbf I_m ( U,G ) = \mathbf I_m^{\textup{loc}} ( U, G)
	\cap \{ S \with \text{$\spt \| S \|$ is compact} \}$.
\end{definition}

\begin{theorem} \label{thm:restriction-homomorphism}
	Suppose $n$ is a positive integer, $U \subset V \subset \mathbf R^n$,
	$i : U \to V$ is the inclusion map, $U$ and $V$ are open, $G$ is a
	complete normed commutative group, and
	\begin{equation*}
		r_m : \mathscr R_m^{\textup{loc}} ( V, G ) \to \mathscr
		R_m^{\textup{loc}} ( U,G ), \quad \text{whenever $m$ is a
		nonnegative integer},
	\end{equation*}
	are characterised by $\vect T | U \in r_m (T)$ for $T \in \mathscr
	R_m^{\textup{loc}} ( V, G )$.

	Then, there holds
	\begin{equation*}
		r_m \big [ \mathbf I_m^{\textup{loc}} ( V, G ) \big ] \subset
		\mathbf I_m^{\textup{loc}} ( U, G )
	\end{equation*}
	and, in case $m \geq 1$, also $\boundary_G r_m(T) = r_{m-1}
	(\boundary_G T)$ whenever $T \in \mathbf I_m^{\textup{loc}} ( V, G )$.
	In particular, we have $i_\# [ \mathbf I_m ( U,G ) ] = \mathbf I_m (
	V,G) \cap \{ T \with \spt \| T \| \subset U \}$.
\end{theorem}

\begin{proof}
	Clearly, $r_m$ are continuous homomorphisms.  Assuming $m \geq 1$, we
	thus define the closed subgroup $H$ of $\mathbf I_m^{\textup{loc}} (
	V, G )$ to consist of those $T \in \mathbf I_m^{\textup{loc}} ( V, G)$
	with
	\begin{equation*}
		(r_m(T),r_{m-1} ( \boundary_GT )) \in \Clos \{ (S,\boundary_G
		S) \with S \in \mathbf I_m ( U, G ) \},
	\end{equation*}
	where the closure is taken in $\mathscr R_m^{\textup{loc}} ( U,G )
	\times \mathscr R_{m-1}^{\textup{loc}} (U,G)$.  Employing
	\ref{thm:integral-chains}\,\eqref{item:integral-chains:push-forward},
	we readily verify $i_\# [ \mathbf I_m ( U,G ) ] \subset H$.  Recalling
	\ref{miniremark:isomorphism-for-integers}, we infer
	\begin{equation*}
		\big ( \rho_{V,m,G} [ \mathbf I_m^{\textup{loc}} ( V, \mathbf
		Z ) \otimes G ] \big ) \cap \{ T \with \text{$\spt \| T \|$ is
		a compact subset of $U$} \} \subset i_\# [ \mathbf I_m (U,G) ]
	\end{equation*}
	from \ref{miniremark:bilinear}; hence, $\mathbf I_m ( V,G ) \cap \{ T
	\with \spt \| T \| \subset U \} \subset H$ by
	\ref{remark:integral-chains}.

	To prove $H = \mathbf I_m^{\textup{loc}} (V,G )$, we next suppose $T
	\in \mathbf I_m^{\textup{loc}} ( V, G )$ and obtain a locally
	Lipschitzian map $f : V \to \mathbf R$ such that $U = \{ x \with f(x)
	> 0 \}$ and $\{ x \with f(x) \geq y \}$ is compact for $y > 0$ from
	\ref{thm:prescribed-zero-set}.  By
	\ref{thm:integral-chains}\,\eqref{item:integral-chains:slice},
	$\mathscr L^1$ almost all $y$ belong to the set $Y$ of $y \in \mathbf
	R$ with $T \restrict \{ x \with f(x)>y \} \in \mathbf
	I_m^{\textup{loc}} ( V, G )$ and
	\begin{equation*}
		\boundary_G ( T \restrict \{ x \with f(x)>y \} ) = \langle
		T,f,y \rangle + ( \boundary_G T ) \restrict \{ x \with f(x)>y
		\};
	\end{equation*}
	therefore, we conclude $T \in H$ because $0 \in \Clos ( Y \cap \{ y
	\with y > 0 \} )$.  As $\mathbf I_m^{\textup{loc}} ( U,G )$ is
	isometric to the subgroup $\{ (S,\boundary_G S) \with S \in \mathbf
	I_m^{\textup{loc}} ( U,G ) \}$ of $\mathscr R_m^{\textup{loc}} (U,G)
	\times \mathscr R_{m-1}^{\textup{loc}} (U,G)$ and $\mathbf
	I_m^{\textup{loc}} (U,G)$ is complete, the conclusion follows.
\end{proof}

\section{Classical coefficient groups}

\begin{example} \label{example:real-coefficients}
	Whenever $m$ is a nonnegative integer and $n$ is a positive integer,
	we employ the quotient map $p : \mathbf G_{\textup o} ( n,m ) \times
	\mathbf R \to \mathbf G (n,m, \mathbf R )$ and \cite[\S\,1]{MR833403}
	to define isomorphisms (of commutative groups)
	\begin{equation*}
		\lambda_{n,m} : \mathbf F_m^{\textup{loc}} ( \mathbf R^n )
		\cap \{ Q \with \text{$Q$ has positive densities} \} \to
		\mathscr R_m^{\textup{loc}} ( \mathbf R^n, \mathbf R )
	\end{equation*}
	by letting $\lambda_{n,m} (Q) \in \mathscr R_m^{\textup{loc}} (
	\mathbf R^n, \mathbf R )$ contain $\tau : X \to \mathbf G (n,m,\mathbf
	R )$ given by
	\begin{equation*}
		\tau (x) = p \big ( \vec Q(x), \boldsymbol \Uptheta^m ( \| Q
		\|, x ) \big ) \quad \text{for $x \in X$},
	\end{equation*}
	where $X = \{ x \with \text{$0 < \boldsymbol \Uptheta^m ( \| Q \|, x )
	< \infty$ and $\vec Q(x) \in \mathbf G_{\textup o} (n,m)$} \}$,
	whenever $Q$ is an $m$ dimensional locally flat chain in $\mathbf R^n$
	with positive densities.  The formal properties of $\iota_{\mathbf
	R^n,m}$ listed in the first paragraph of
	\ref{miniremark:isomorphism-for-integers} are shared by
	$\lambda_{n,m}$ and have the same proof; moreover, we have
	$\lambda_{n,m} ( rQ ) = \iota_{\mathbf R^n, m} ( Q ) \cdot r$ for $Q
	\in \mathbf I_m^{\textup{loc}} ( \mathbf R^n )$ and $r \in \mathbf R$.
	Recalling \ref{miniremark:bilinear} and
	\ref{corollary:boundary-operator-dense-subset}, we infer that
	\begin{gather*}
		\mathscr P_m ( \mathbf R^n, \mathbf R ) = \lambda_{n,m} [
		\mathbf P_m ( \mathbf R^n ) ], \\
		\rho_{\mathbf R^n,m,\mathbf R} [ \mathbf I_m^{\textup{loc}} (
		\mathbf R^n, \mathbf Z) \otimes \mathbf R ] = \lambda_{n,m} [
		D_{n,m} ], \\
		\quad \text{where $D_{n,m}$ is the real linear span of
		$\mathbf I_m^{\textup{loc}} ( \mathbf R^n )$ in $\mathbf
		F_m^{\textup{loc}} ( \mathbf R^n ) $},
	\end{gather*}
	and, if $m \geq 1$, that $\boundary_{\mathbf R} \lambda_{n,m} (Q) =
	\lambda_{n,m-1} ( \boundary Q )$ for $Q \in D_{n,m}$.
	
	Next, we define commutative groups by
	\begin{align*}
		I_{n,m} & = \mathbf F_m^{\textup{loc}} ( \mathbf R^n ) \cap \{
		Q \with \text{$Q$ and $\boundary Q$ have positive densities}
		\} \quad \text{if $m \geq 1$}, \\
		I_{n,0} & = \mathbf F_0^{\textup{loc}} ( \mathbf R^n ) \cap \{
		Q \with \text{$Q$ has positive densities} \}.
	\end{align*}
	Clearly, we have $\mathbf I_0^{\textup{loc}} ( \mathbf R^n, \mathbf R)
	= \lambda_{n,0} [ I_{n,0} ]$.  For $m \geq 1$, we will establish that
	\begin{equation*}
		\mathbf I_m^{\textup{loc}} ( \mathbf R^n, \mathbf R) =
		\lambda_{n,m} [ I_{n,m} ] \quad \text{with} \quad
		\text{$\boundary_{\mathbf R} \lambda_{n,m} (Q) =
		\lambda_{n,m-1} ( \boundary Q)$ for $Q \in I_{n,m}$}.
	\end{equation*}
	Defining the group norm $\sigma$ on $I_{n,m}$ by
	\begin{equation*}
		\sigma (Q) = \sum_{i=1}^\infty \big ( \inf \{ 2^{-i}, \| Q \|
		\, \mathbf B(0,i) \} + \inf \{ 2^{-i}, \| \boundary Q \| \,
		\mathbf B (0,i) \} \big ) \quad \text{for $Q \in I_{n,m}$},
	\end{equation*}
	$\lambda_{n,m}$ maps $D_{n,m}$ isometrically onto $\rho_{\mathbf R^n,
	m, \mathbf R} [ \mathbf I_m^{\textup{loc}} ( \mathbf R^n, \mathbf Z )
	\otimes \mathbf R ]$.  Recalling that $\rho_{\mathbf R^n, m, \mathbf
	R} [ \mathbf I_m^{\textup{loc}} ( \mathbf R^n, \mathbf Z ) \otimes
	\mathbf R ]$ is dense in $\mathbf I_m^{\textup{loc}} ( \mathbf R^n,
	\mathbf R )$ and $\mathbf I_m^{\textup{loc}} ( \mathbf R^n, \mathbf
	R)$ is complete, it therefore suffices to prove that $D_{n,m}$ is
	$\sigma$ dense in $I_{n,m}$, as $I_{n,m}$ is $\sigma$ complete.
	Observing that $\mathbf F_m ( \mathbf R^n ) \cap I_{n,m}$ is $\sigma$
	dense in $I_{n,m}$ by \cite[4.2.1, 4.3.1, 4.3.8]{MR41:1976}, this is a
	consequence of the deformation theorem obtained in
	\cite[\S\,4]{MR833403}; in fact, employing \cite[\S\,4]{MR833403}
	instead of \cite[4.2.9]{MR41:1976} in its derivation, one readily
	verifies that one may replace $\mathbf I_m ( \mathbf R^n)$ and
	$\mathscr P_m ( \mathbf R^n )$ by $\mathbf F_m ( \mathbf R^n ) \cap
	I_{n,m}$ and $\mathbf P_m ( \mathbf R^n )$, respectively, in the
	approximation theorem \cite[4.2.20]{MR41:1976} and clearly we have
	$g_\# P \in D_{n,m}$ whenever $P \in \mathbf P_m ( \mathbf R^n )$ and
	$g : \mathbf R^n \to \mathbf R^n$ is Lipschitzian.
\end{example}

\begin{example} \label{example:integers-mod-lambda}
	Here, we relate the present treatment to \cite[4.2.26]{MR41:1976}.
	Whenever $m$ is a non\-neg\-a\-tive integer, $n$ and $d$ are positive
	integers, and $U$ is an open subset of $\mathbf R^n$, we let $p :
	\mathbf G_{\textup o} ( n,m ) \times ( \mathbf Z / d \mathbf Z ) \to
	\mathbf G(n,m, \mathbf Z / d \mathbf Z)$ denote the quotient map and
	define $\mu_{U,m,d}$ to be the composition of isomorphisms
	\begin{align*}
		\mathscr R_m^d ( U ) & \simeq \mathscr R_m ( U ) / d \mathscr
		R_m ( U ) \\
		& \simeq \mathscr R_m ( U ) \otimes ( \mathbf Z / d \mathbf Z)
		\simeq \mathscr R_m ( U, \mathbf Z ) \otimes ( \mathbf Z / d
		\mathbf Z ) \simeq \mathscr R_m ( U, \mathbf Z / d \mathbf Z),
	\end{align*}
	where the inverse of the first isomorphism thereof is induced by the
	homomorphism mapping $Q \in \mathscr R_m ( U )$ onto $(Q)^d \in
	\mathscr R_m^d ( U )$, the second is canonical, the third is
	$\iota_{U,m} \otimes \mathbf 1_{\mathbf Z/ d \mathbf Z}$, and the
	fourth is induced by $\rho_{U,m,\mathbf Z/d \mathbf Z}$; see
	\cite[4.2.26, p.\ 426, p.\ 430]{MR41:1976}, \cite[Chapter II, \S\,3.6,
	Corollary 2 to Proposition 6]{MR1727844},
	\ref{miniremark:isomorphism-for-integers}, and
	\ref{remark:isomorphisms-finite-group}, respectively.  If $Q \in
	\mathscr R_m ( U )$ and $S = \mu_{U,m,d} \big ( ( Q )^d \big )$, then
	the function $\sigma : X \to \mathbf G ( n,m, \mathbf Z/ d \mathbf
	Z)$, defined by
	\begin{equation*}
		\sigma ( x ) = p \big ( \vec Q (x), \boldsymbol \Uptheta^m (
		\| Q \|, x ) \cdot 1 \big ) \quad \text{for $x \in X$},
	\end{equation*}
	where $X = \{ x \with \text{$\vec Q(x) \in \mathbf G_{\textup o}
	(n,m)$ and $0 < \boldsymbol \Uptheta^m ( \| Q \|, x ) \in \mathbf Z$}
	\}$ and $1 \in \mathbf Z / d \mathbf Z$, belongs to $S$.  We readily
	verify the following three properties for $Q \in \mathscr R_m ( U )$
	with $q = (Q)^d$.  Firstly, $\| \mu_{U,m,d} (q) \| = \| Q \|^d$;
	secondly,
	\begin{equation*}
		(\mu_{U,m,d} (q) ) \restrict A = \mu_{U,m,d} (q \restrict A )
		\quad \text{whenever $A$ is $\| Q \|^d$ measurable};
	\end{equation*}
	and, thirdly,
	\begin{equation*}
		f_\# ( \mu_{U,m,d} ( q ) ) = \mu_{U,m,d} ( f_\# q)
	\end{equation*}
	whenever $\nu$ is a positive integer, $V$ is an open subset of
	$\mathbf R^\nu$, and $f : U \to V$ is locally Lipschitzian.  Recalling
	\ref{miniremark:bilinear} and
	\ref{corollary:boundary-operator-dense-subset}, we obtain that
	\begin{gather*}
		\mathscr P_m ( U, \mathbf Z/ d \mathbf Z ) = \mu_{U,m,d} \big
		[ \mathscr P_m^d ( U) \big ], \\
		\rho_{U,m,\mathbf Z/d \mathbf Z} \big [ \mathbf I_m ( U,
		\mathbf Z) \otimes ( \mathbf Z / d \mathbf Z ) \big ] =
		\mu_{U,m,d} \big [ \{ ( Q)^d \with Q \in \mathbf I_m ( U ) \}
		\big ], \\
		\boundary_{\mathbf Z/ d \mathbf Z} \mu_{U,m,d} (q) =
		\mu_{U,m-1,d} ( \boundary q ) \quad \text{whenever $q =
		(Q)^d$, $Q \in \mathbf I_m ( U )$, and $m \geq 1$}.
	\end{gather*}
	
	Clearly, $\mathbf I_0 ( U, \mathbf Z / d \mathbf Z ) = \mu_{U,0,d} [
	\mathbf I_0^d ( U ) ]$.  For $m \geq 1$, we will establish that
	\begin{gather*}
		\mathbf I_m ( U, \mathbf Z/ d \mathbf Z ) = \mu_{U,m,d} [
		\mathbf I_m^d ( U ) ], \\
		\boundary_{\mathbf Z/ d \mathbf Z} \mu_{U,m,d} ( q) =
		\mu_{U,m-1,d} ( \boundary q ) \quad \text{for $q \in \mathbf
		I_m^d ( U )$}.
	\end{gather*}
	We abbreviate $D = \{ ( Q)^d \with Q \in \mathbf I_m (U) \}$ and $E =
	\rho_{U,m,\mathbf Z/d \mathbf Z} [ \mathbf I_m ( U, \mathbf Z )
	\otimes ( \mathbf Z / d \mathbf Z ) ]$, define the group norm $\tau$
	on $\mathbf I_m ( U, \mathbf Z / d \mathbf Z)$ by
	\begin{equation*}
		\tau ( S ) = ( \| S \| + \| \boundary_{\mathbf Z/ d \mathbf Z}
		S \| ) ( U ) \quad \text{for $S \in \mathbf I_m ( U, \mathbf Z
		/ d \mathbf Z )$},
	\end{equation*}
	and notice that $\mu_{U,m,d}|D$ is an isometry (onto $E$) with respect
	to $\mathbf N^d$ and $\tau$.  One may readily infer the conclusion by
	combining the following four assertions, to be shown for every compact
	subset $K$ of $U$.  Firstly, $\mathbf I_m^d (U) \cap \{ r \with \spt^d
	r \subset K \}$ is $\mathbf N^d$ complete; secondly, $\mathbf I_m (U,
	\mathbf Z/ d \mathbf Z ) \cap \{ S \with \spt \| S \| \subset K \}$ is
	$\tau$ complete; thirdly, each $q \in \mathbf I_m^d (U)$ with $\spt^d
	q \subset \Int K$ belongs to the $\mathbf N^d$ closure of $D \cap \{ r
	\with \spt^d r \subset K \}$; and, fourthly, each $S \in \mathbf I_m
	(U, \mathbf Z/d \mathbf Z)$ with $\spt \| S \| \subset \Int K$ belongs
	to the $\tau$ closure of $E \cap \{ T \with \spt \| T \| \subset K
	\}$. The first two are elementary, the fourth follows from
	\ref{remark:integral-chains}, and the special case $U = \mathbf R^n$
	of the third is a consequence of \cite[(4.2.20)$^d$]{MR41:1976}.
	Observing that $i_\#$ induced by the inclusion map $i : U \to \mathbf
	R^n$, maps $D$ onto $\{ (Q)^d \with Q \in \mathbf I_m (\mathbf R^n),
	\spt^d Q \subset U \}$ and $\mathbf I_m^d ( U )$ onto $\mathbf I_m^d (
	\mathbf R^n ) \cap \{ r \with \spt^d r \subset U \}$, the general case
	of the third reduces to the special case thereof.

	Finally, we record the structural theorem
	\begin{equation*}
		\mathbf I_m^d ( U ) = \{ ( Q )^d \with Q \in \mathbf I_m ( U )
		\}
	\end{equation*}
	from \cite[Corollary 1.5]{MR3743699}; in fact, the cited source treats
	the main case $U = \mathbf R^n$ to which the case $U \neq \mathbf R^n$
	is readily reduced (for instance, by means of
	\ref{miniremark:bilinear}).
\end{example}

\begin{remark} \label{remark:Federer-correction}
	In line with the list of corrections to \cite{MR41:1976} that Federer
	maintained and distributed, we mention that, in contrast with the
	structural theorem, it had been known that
	\begin{equation*}
		\{ (Q)_K^d \with Q \in \mathbf I_{m,K} ( U ) \}
	\end{equation*}
	is in general a proper subset of $\mathbf I_{m,K}^d ( U )$; in fact,
	we will show that, taking $S$ and $f$ as in \cite[p.~426]{MR41:1976},
	$m = d = 2$, $U = \mathbf R^6$, and $K = \spt S$, an example is
	furnished by $(S)_K^2$.  Clearly, we have $S \in \mathscr R_{2,K} (
	\mathbf R^6 )$ and $\mathscr F_K^2 ( \boundary S ) = 0$ so that we
	have $(S)_K^2 \in \mathbf I_{2,K}^2 ( \mathbf R^6)$.  Moreover, if $Q
	\in \mathscr F_{2,K} ( \mathbf R^6)$ and $(Q)_K^2 = (S)_K^2$, then we
	conclude
	\begin{equation*}
		Q-S \in 2 \mathscr R_{2,K} ( \mathbf R^6 ),
	\end{equation*}
	since $\mathscr R_{3,K} ( \mathbf R^6 ) = \{ 0 \}$ so that $\mathscr
	F_{2,K} ( \mathbf R^6 ) = \mathscr R_{2,K} ( \mathbf R^6 )$, $\mathscr
	F_K ( R) = \mathbf M ( R )$ whenever $R \in \mathscr F_{2,K} ( \mathbf
	R^6)$, and $2 \mathscr F_{2,K} ( \mathbf R^6 )$ is $\mathscr F_K$
	closed in $\mathscr F_{2,K} ( \mathbf R^6 )$, whence we infer that
	\begin{equation*}
		\text{$\boldsymbol \Uptheta^2 ( \| Q \|, x )$ is an odd
		integer for $\mathscr H^2$ almost all $x \in K$}.
	\end{equation*}
	Using that $B = f [ \mathbf S^2 ]$ is a nonorientable connected
	compact two-dimensional submanifold of class $\infty$ of $\mathbf
	R^6$, we construct $c>0$ such that $\mathbf M ( \boundary ( Q
	\restrict C_j ) ) \geq c j^{-1}$ for every positive integer $j$, where
	$C_j = f \big [ \mathbf R^3 \cap \{ x \with |x| = j^{-1/2} \} \big ]$,
	so that
	\begin{equation*}
		\mathbf M ( \boundary Q ) \geq \sum_{j=1}^\infty \mathbf M (
		\boundary ( Q \restrict C_j ) ) = \infty;
	\end{equation*}
	in fact, since we have $\mathscr F_{3,B} ( \mathbf R^6 ) \subset
	\mathbf F_{3,B} ( \mathbf R^6 ) = \{ 0 \}$ by
	\cite[4.1.15]{MR41:1976}, we notice $\mathbf F_B ( R) = \mathbf M (R)$
	for $R \in \mathbf F_{2,B} ( \mathbf R^6 )$ by
	\cite[4.1.24]{MR41:1976} as well as $\mathscr F_B ( R ) = \mathbf M
	(R)$ for $R \in \mathscr F_{2,B} ( \mathbf R^6 )$, whence we firstly
	deduce
	\begin{equation*}
		\inf \{ \mathbf M ( \boundary R ) \with R \in \mathbf F_{2,B}
		( \mathbf R^6), \mathbf M ( R ) = 1 \} > 0
	\end{equation*}
	by \cite[4.1.31\,(2), 4.2.17\,(1)]{MR41:1976} and then that we may
	take $c$ to be the infimum of the set of numbers $\mathbf M (
	\boundary R )$, corresponding to all $R \in \mathscr R_{2,B} ( \mathbf
	R^6)$ such that $\boldsymbol \Uptheta^2 ( \| R \|, x)$ is an odd
	integer for $\mathscr H^2$ almost all $x \in B$ because $c>0$ by
	\cite[4.2.16\,(2), 4.2.17\,(2)]{MR41:1976}.
\end{remark}

\begin{remark} \label{remark:earlier-correction-attempt}
	The preceding remark shall complete the literature as Federer's list
	of corrections---which does not contain proofs of the above
	statements---remains unpublished and the corrections in
	\cite[2.5]{MR482509} and \cite[\S\,4]{MR3819529} are not entirely
	satisfactory.  Regarding \cite{MR3819529}, a revised version will
	appear in \cite{MR3819529-corrigendum}.
\end{remark}

\section{Constancy theorem}

\begin{theorem} \label{thm:constancy-theorem}
	Suppose $m$ and $n$ are positive integers, $U$ is an open subset
	of~$\mathbf R^n$, $M$ is a connected $m$ dimensional submanifold of
	class $1$ of $U$, $\zeta$ is an $m$ vector field orienting $M$, $G$ is
	a complete normed commutative group, $S \in \mathbf I_m^{\textup{loc}}
	( U, G )$,
	\begin{equation*}
		M \cap \spt \| S \| \neq \varnothing, \quad \spt \|
		\boundary_G S \| \subset U \without M,
	\end{equation*}
	and $( \spt \| S \| ) \without M$ is closed relative to $U$.

	Then, for some nonzero member $g$ of $G$, there holds (see
	\ref{miniremark:isomorphism-for-integers})
	\begin{equation*}
		\spt \| S - \iota_{U,m} ( ( \mathscr H^m \restrict M ) \wedge
		\zeta ) \cdot g \| \subset U \without M.
	\end{equation*}
\end{theorem}

\begin{proof}
	We firstly establish that if either $m=1$ or $m>1$ and the statement
	of the theorem holds with $m$ replaced by $m-1$, then the following
	assertion holds: \emph{If $- \infty < a_i < b_i < \infty$ for $i = 1,
	\ldots, m$,
	\begin{equation*}
		C = \mathbf R^m \cap \{ x \with \textup{$a_i \leq x_i \leq
		b_i$ for $i = 1, \ldots, m$} \},
	\end{equation*}
	and $T \in \mathbf I_m ( \mathbf R^m, G )$ satisfies $\spt \|
	\boundary_G T \| \subset \Bdry C$, then there exists $g$ in $G$ with
	$T = \iota_{\mathbf R^m,m} ( \boldsymbol [ a_1, b_1 \boldsymbol ]
	\times \cdots \times \boldsymbol [ a_m, b_m \boldsymbol ] ) \cdot g$.}
	We abbreviate $B = \mathbf R^m \cap \{ x \with x_1 = b_1 \}$, $I =
	\iota_{\mathbf R,1} ( \boldsymbol [0,1] )$, and define $f : \mathbf
	R^m \to \mathbf R^m$ and $h : \mathbf R \times \mathbf R^m \to \mathbf
	R^m$ by
	\begin{equation*}
		\begin{gathered}
			f(x) = (a_1, x_2, \ldots, x_m), \\
			h(t,x) = (1-t)f(x) + tx = \big ( (1-t) a_1 + t x_1,
			x_2, \ldots, x_m \big )
		\end{gathered}
	\end{equation*}
	for $(t,x) \in \mathbf R \times \mathbf R^m$.  We see from
	\ref{corollary:integral-chains:homotopy-formula} that
	\begin{equation*}
		T = h_\# ( I \times (( \boundary_G T) \restrict B ) );
	\end{equation*}
	in fact, in view of \ref{miniremark:push-forward-G-chain}, we conclude
	that $f_\# T = 0$, that $h_\# ( I \times T) = 0$, and, since $\mathscr
	H^m ( h [ \mathbf R \times (( \Bdry C) \without B )]) = 0$, also that
	$h_\# \big ( I \times (( \boundary_G T ) \restrict ( \mathbf R^m
	\without B)) \big ) = 0$.  If $m = 1$, then $B = \{ b_1 \}$ and there
	exists $g \in G$ with $( \boundary_G T ) \restrict B = \iota_{\mathbf
	R, 0} ( \boldsymbol [ b_1 \boldsymbol ] ) \cdot g$ and
	\begin{equation*}
		T = \iota_{\mathbf R,1} \big ( h_\# ( \boldsymbol [ 0,1
		\boldsymbol ] \times \boldsymbol [ b_1 \boldsymbol ] ) \big )
		\cdot g = \iota_{\mathbf R, 1} ( \boldsymbol [ a_1, b_1 ] )
		\cdot g,
	\end{equation*}
	both by \ref{miniremark:isomorphism-for-integers} and
	\ref{miniremark:bilinear}.  If $m > 1$, then, assuming $T \neq 0$,
	denoting the standard basis of $\mathbf R^m$ by $e_1, \ldots, e_m$,
	and taking $n=m$,
	\begin{equation*}
		\begin{gathered}
			U = \mathbf R^m \cap \{ x \with \text{$x_1 > a_1$,
			$a_i < x_i < b_i$ for $i = 2, \ldots, m$} \}, \quad M
			= B \cap U, \\
			\text{$\zeta(y)= e_2 \wedge \cdots \wedge e_m$ for $y
			\in M$}, \quad \text{$S = r_{m-1} ( \boundary_G T )
			\in \mathbf I_{m-1}^{\textup{loc}} ( U, G )$, see
			\ref{thm:restriction-homomorphism}},
		\end{gathered}
	\end{equation*}
	hence $B \cap \Bdry C$ and $M = U \cap \Bdry C$ are $\mathscr H^{m-1}$
	almost equal, $\boundary_G S = 0$ by
	\ref{thm:integral-chains}\,\eqref{item:integral-chains:complex}, and
	$\varnothing \neq \spt \| S \| \subset M$ since $T \neq 0$, we apply
	the statement of the theorem with $m$ replaced by $m-1$ to infer, for
	some $g \in G$, that
	\begin{equation*}
		S = \iota_{U,m-1} \big ( ( \mathscr H^{m-1} \restrict M )
		\wedge \zeta \big ) \cdot g;
	\end{equation*}
	thus,
	\begin{equation*}
		\begin{aligned}
			( \boundary_G T ) \restrict B & = \iota_{\mathbf
			R^m,m-1} \big ( ( \mathscr H^{m-1} \restrict M )
			\wedge \zeta \big ) \cdot g \\
			& = \iota_{\mathbf R^m,m-1} \big ( \xi_\# (
			\boldsymbol [a_2,b_2 \boldsymbol ] \times \cdots
			\times \boldsymbol [ a_m,b_m \boldsymbol ] ) \big )
			\cdot g,
		\end{aligned}
	\end{equation*}
	where $\xi : \mathbf R^{m-1} \to \mathbf R^m$ is defined by $\xi (y) =
	(b_1, y_1, \ldots, y_{m-1} )$ for $y \in \mathbf R^{m-1}$, so that
	\ref{miniremark:isomorphism-for-integers} and
	\ref{miniremark:bilinear} yield
	\begin{align*}
		T & = \iota_{\mathbf R^m,m} \big ( h_\# ( \boldsymbol [0,1
		\boldsymbol ] \times \xi_\# ( \boldsymbol [ a_2, b_2
		\boldsymbol ] \times \cdots \times \boldsymbol [ a_m, b_m
		\boldsymbol ] ) ) \big ) \cdot g \\
		& = \iota_{\mathbf R^m,m} ( \boldsymbol [ a_1, b_1 \boldsymbol
		] \times \cdots \times \boldsymbol [ a_m, b_m \boldsymbol ] )
		\cdot g,
	\end{align*}
	because $h(t,\xi(y)) = ((1-t)a_1 + tb_1, y_1, \ldots, y_{m-1} )$ for
	$(t,y) \in \mathbf R \times \mathbf R^{m-1}$.

	Secondly, we will show that the special case $n=m$ and $U = M =
	\mathbf R^m$ of the statement of the theorem holds for $m$ whenever
	the assertion of the first paragraph holds for this $m$.  Abbreviating
	\begin{equation*}
		C_r = \mathbf R^m \cap \{ x \with \text{$|x_i| \leq r$ for $i
		= 1, \ldots, m$} \},
	\end{equation*}
	we see that $\mathscr L^1$ almost all positive real numbers belong to
	the set $A$ of $0 < r < \infty$ satisfying $S \restrict C_r \in
	\mathbf I_m ( \mathbf R^n, G )$ and $\spt \| \boundary_G ( S \restrict
	C_r ) \| \subset \Bdry C_r$ by
	\ref{thm:integral-chains}\,\eqref{item:integral-chains:slice}.  The
	assertion of the first paragraph then yields $\alpha : A \to G$ with
	\begin{equation*}
		S \restrict C_r = \iota_{\mathbf R^m,m} \big ( (\mathscr L^m
		\restrict C_r) \wedge \zeta \big ) \cdot \alpha (r) \quad
		\text{for $r \in A$}.
	\end{equation*}
	Recalling \ref{miniremark:def-rectifiable-G-chains} and
	\ref{miniremark:bilinear}, we compute
	\begin{align*}
		& \iota_{\mathbf R^m,m} \big ( ( \mathscr L^m \restrict C_r )
		\wedge \zeta ) \cdot ( \alpha(s)-\alpha(r)) \\
		& \qquad = S \restrict ( C_s \without C_r ) - \iota_{\mathbf
		R^m,m} \big ( ( \mathscr L^m \restrict (C_s \without C_r) )
		\wedge \zeta \big ) \cdot \alpha (s)
	\end{align*}
	whenever $r,s \in A$ and $r < s$ and conclude that $\im \alpha$
	consists of a single point $g \in G$, that $S = \iota_{\mathbf R^m,m}
	( \mathscr L^m \wedge \zeta ) \cdot g$, and that $g \neq 0$.

	It remains to prove that the statement of the theorem holds for $m$
	whenever the assertion of the first paragraph holds for this $m$.  For
	this purpose, we consider the class $\Omega$ of all $(V,h)$ such that
	$V$ is an open subset $U$, $h \in G$, and
	\begin{gather*}
		M \cap V \neq \varnothing, \quad V \cap \spt \| S \| \subset
		M, \quad ( \mathscr H^m \restrict M \cap V ) \wedge \zeta \in
		\mathscr R_m^{\textup{loc}} ( U ), \\
		V \cap \spt \big \| S - \iota_{U,m} \big ( ( \mathscr H^m
		\restrict M \cap V ) \wedge \zeta \big ) \cdot h \big \| =
		\varnothing.
	\end{gather*}

	To establish $M = \bigcup \{ M \cap V \with (V,h) \in \Omega \}$, we
	suppose $a \in M$, recall that nontrivial open balls in $\mathbf R^m$
	are diffeomorphic to $\mathbf R^m$, and employ
	\cite[3.1.19\,(4)]{MR41:1976} in constructing an open subset $V$ of
	$U$ with $a \in V$, $V \cap \spt \| S \| \subset M$, and $\mathscr H^m
	( M \cap V) < \infty$ as well as maps $\phi : V \to \mathbf R^m$ and
	$\psi : \mathbf R^m \to V$ of class $1$ such that
	\begin{equation*}
		\phi \circ \psi = \mathbf 1_{\mathbf R^m}, \quad M \cap V =
		\im \psi.
	\end{equation*}
	We observe that $\psi$ and $\phi | \im \psi$ are proper and choose an
	orientation $\eta$ of $\mathbf R^m$ such that
	\begin{equation*}
		\psi_\# ( \mathscr L^m \wedge \eta ) = ( \mathscr H^m
		\restrict M \cap V ) \wedge \zeta \in \mathscr
		R_m^{\textup{loc}} ( V )
	\end{equation*}
	by means of \cite[4.1.31]{MR41:1976} and
	\ref{miniremark:push-forward-G-chain}.  As $V \cap \spt \| \boundary_G
	S \| = \varnothing$ by our hypothesis and
	\ref{thm:integral-chains}\,\eqref{item:integral-chains:complex}, we
	may apply \ref{thm:restriction-homomorphism} (with the roles of $U$
	and $V$ exchanged) to obtain $S' = r_m ( S ) \in \mathbf
	I_m^{\textup{loc}} ( V,G )$ with $\boundary_G S' = 0$.  Defining $T =
	\phi_\# S' \in \mathbf I_m^{\textup{loc}} ( \mathbf R^m, G )$ with
	$\boundary_G T = 0$ by
	\ref{thm:integral-chains}\,\eqref{item:integral-chains:push-forward},
	the special case yields $h \in G$ with
	\begin{equation*}
		T = \iota_{\mathbf R^m,m} ( \mathscr L^m \wedge \eta ) \cdot
		h.
	\end{equation*}
	Noting $\psi \circ \phi | \im \psi = \mathbf 1_{\im \psi}$, we
	conclude $S' = \psi_\# T = \iota_{V,m} \big ( ( \mathscr H^m \restrict
	M \cap V ) \wedge \zeta \big ) \cdot h$ by
	\ref{miniremark:push-forward-G-chain} and \ref{miniremark:bilinear},
	whence we infer $(V,h) \in \Omega$.

	If $(V_1,h_1)$ and $(V_2,h_2)$ belong to $\Omega$ and $M \cap V_1 \cap
	V_2 \neq \varnothing$, then $h_1 = h_2$; in fact, $\mathscr H^m ( M
	\cap V_1 \cap V_2 ) > 0$ and, by
	\ref{miniremark:def-rectifiable-G-chains} and
	\ref{miniremark:bilinear}, we have
	\begin{align*}
		& V_1 \cap V_2 \cap \spt \big \| \iota_{U,m} \big ( ( \mathscr
		H^m \restrict M \cap V_1 \cap V_2 ) \wedge \zeta \big) \cdot (
		h_1 - h_2 ) \big \| \\
		& \qquad \subset \spt \big \| \iota_{U,m} \big ( ( \mathscr
		H^m \restrict M \cap V_1 ) \wedge \zeta \big ) \cdot h_1 -
		\iota_{U,m} \big ( ( \mathscr H^m \restrict M \cap V_2 )
		\wedge \zeta \big ) \cdot h_2 \big \|
	\end{align*}
	with the latter set not meeting $V_1 \cap V_2$.  Since $M \cap \spt \|
	S \| \neq \varnothing$, we can select a nonzero $g \in G$ with
	\begin{equation*}
		\Upsilon = \Omega \cap \{ (V,h) \with h=g \} \neq \varnothing
	\end{equation*}
	and infer that $\bigcup \{ M \cap V \with (V,g) \in \Omega \}$ is a
	nonempty, relatively open and relatively closed subset of $M$, hence
	equals $M$.  We conclude $\Upsilon = \Omega$.  Since $g \neq 0$ and
	\begin{equation*}
		\| S \| \restrict V = ( \mathscr H^m \restrict M \cap V ) | g
		| \quad \text{for $(V,g) \in \Omega$}
	\end{equation*}
	by \ref{miniremark:bilinear}, we also have $( \mathscr H^m \restrict
	M) \wedge \zeta \in \mathscr R_m^{\textup{loc}} ( U)$ and the
	conclusion follows.
\end{proof}

\begin{remark} \label{remark:constacy-theorem-model}
	The preceding theorem is modelled on \cite[4.1.31\,(2)]{MR41:1976}
	where the case $G = \mathbf R$ is treated.  In the present case,
	orientability of $M$ must be part of the hypotheses (rather than the
	conclusion) by \cite[4.2.26, p.\,432]{MR41:1976} and
	\ref{example:integers-mod-lambda}.
\end{remark}

\begin{remark} \label{remark:constancy-theorem-discussion}
	The special case $U = \mathbf R^n$ and $M$ an $m$ dimensional cube is
	a basic ingredient of deformation theorems.  For $G = \mathbf Z$ or $G
	= \mathbf R$, this follows from the constancy theorem derived in
	\cite[4.1.4, 4.1.7]{MR41:1976}.  An alternative approach for $G =
	\mathbf R$, designed to be extendable to $G = \mathbf Z / d \mathbf
	Z$, is given in \cite[4.2.3]{MR41:1976}.  However, the flat chain $X
	\restrict \mathbf R^m \without H$ constructed in that proof does not
	belong to the domain of $\psi_\#$ as claimed; this is easily
	circumvented for $G = \mathbf R$ but requires some further arguments
	for $G = \mathbf Z / d \mathbf Z$.  To avoid this difficulty and as
	our boundary operator $\boundary_G$ is not yet defined on all of
	$\mathscr R_m^{\textup{loc}} ( U,G )$, the present proof (applicable
	to locally integral chains only) merges the extensions of
	\cite[4.1.31\,(2)]{MR41:1976} and \cite[4.2.3]{MR41:1976} to general
	$G$ in a simultaneous inductive argument by means of the restriction
	operators $r_m$ constructed in \ref{thm:restriction-homomorphism}.
	Finally, in the context of the flat $G$ chains of \cite{MR2876138}, a
	different approach to the special case is chosen in
	\cite[6.3]{MR3206697} on which a constancy theorem for chains in
	\emph{Lipschitz submanifolds} of complete separable metric spaces (see
	\cite[7.6]{MR3206697}) is based.
\end{remark}

\section{Flat chains}

\begin{miniremark} \label{miniremark:locally-flat-G-chains}
	Suppose $n$ is a positive integer, $U$ is an open subset of $\mathbf
	R^n$, and $G$ is a complete normed commutative group.  Whenever $m$ is
	a nonnegative integer, we note
	\begin{equation*}
		H_m = \big ( \mathbf I_m^{\textup{loc}} ( U,G ) \times \mathbf
		I_{m+1}^{\textup{loc}} ( U,G ) \big ) \cap \{ (S,T) \with S +
		\boundary_G T = 0 \}
	\end{equation*}
	is a closed subgroup of $\mathscr R_m^{\textup{loc}} ( U, G ) \times
	\mathscr R_{m+1}^{\textup{loc}} ( U,G )$ by
	\ref{thm:integral-chains}\,\eqref{item:integral-chains:mono}%
	\,\eqref{item:integral-chains:complex} and recall
	\ref{remark:normed-group} to define the complete normed commutative
	group $\mathscr F_m^{\textup{loc}} (U,G)$ of $m$ dimensional
	\emph{locally flat $G$ chains} in $U$ to be the quotient
	\begin{equation*}
		\big ( \mathscr R_m^{\textup{loc}} (U,G) \times \mathscr
		R_{m+1}^{\textup{loc}} (U,G) \big ) \big /H_m.
	\end{equation*}
	Since the composition of canonical continuous homomorphisms
	\begin{equation*}
		\mathscr R_m^{\textup{loc}} (U,G) \to \mathscr
		R_m^{\textup{loc}} ( U, G ) \times \mathscr
		R_{m+1}^{\textup{loc}} ( U, G ) \to \mathscr
		F_m^{\textup{loc}} ( U, G )
	\end{equation*}
	is univalent, \emph{we will henceforth identify $\mathscr
	R_m^{\textup{loc}} (U,G)$ with its image in $\mathscr
	F_m^{\textup{loc}} (U,G)$.}  Whenever $m$ is a positive integer,
	noting that the continuous homomorphisms
	\begin{equation*}
		b_m : \mathscr R_m^{\textup{loc}} ( U, G ) \times \mathscr
		R_{m+1}^{\textup{loc}} ( U,G ) \to \mathscr
		R_{m-1}^{\textup{loc}} ( U, G ) \times \mathscr
		R_m^{\textup{loc}} ( U,G ),
	\end{equation*}
	defined by $b_m (S,T) = (0,S)$ for $(S,T) \in \mathscr
	R_m^{\textup{loc}} (U,G) \times \mathscr R_{m+1}^{\textup{loc}}
	(U,G)$, satisfy the conditions $b_m [H_m] \subset H_{m-1}$ by
	\ref{thm:integral-chains}\,\eqref{item:integral-chains:complex}, $b_m
	\circ b_{m+1} = 0$, and $(\boundary_G S,-S) \in H_{m-1}$ for $S \in
	\mathbf I_m^{\textup{loc}} (U,G)$, they induce continuous quotient
	homomorphisms
	\begin{equation*}
		\boundary_G : \mathscr F_m^{\textup{loc}} ( U,G ) \to \mathscr
		F_{m-1}^{\textup{loc}} ( U, G )
	\end{equation*}
	with $\boundary_G ( \boundary_G S ) = 0$ for $S \in \mathscr
	F_{m+1}^{\textup{loc}} (U,G)$ such that the diagram
	\begin{equation*}
		\begin{xy}
			\xymatrix{
				\mathbf I_m^{\textup{loc}} (U,G) \ar[r]
				\ar[d]^{\boundary_G} & \mathscr
				R_m^{\textup{loc}} (U,G) \ar[r] & \mathscr
				F_m^{\textup{loc}} (U,G) \ar[d]^{\boundary_G}
				\\
				\mathbf I_{m-1}^{\textup{loc}} (U,G) \ar[r] &
				\mathscr R_{m-1}^{\textup{loc}} (U,G) \ar[r] &
				\mathscr F_{m-1}^{\textup{loc}} ( U,G) }
		\end{xy}
	\end{equation*}
	commutes.  Whenever $m$ is a nonnegative integer, we notice that the
	quotient homomorphism of $\mathscr R_m^{\textup{loc}} ( U,G ) \times
	\mathscr R_{m+1}^{\textup{loc}} (U,G)$ onto $\mathscr
	F_m^{\textup{loc}} ( U, G)$ maps
	\begin{equation*}
		\text{$(S,T) \in \mathscr R_m^{\textup{loc}} ( U, G ) \times
		\mathscr R_{m+1}^{\textup{loc}} ( U,G)$ onto $S + \boundary_G
		T \in \mathscr F_m^{\textup{loc}} ( U,G )$},
	\end{equation*}
	in particular $\mathscr F_m^{\textup{loc}} (U,G) = \big \{ S +
	\boundary_G T \with S \in \mathscr R_m^{\textup{loc}} (U,G), T \in
	\mathscr R_{m+1}^{\textup{loc}} (U,G) \big \}$, and we record that
	\begin{equation*}
		\text{$\mathscr P_m ( U,G)$ is dense in $\mathscr
		F_m^{\textup{loc}} (U,G)$};
	\end{equation*}
	in fact, to prove the second assertion, recalling
	\ref{miniremark:isomorphism-for-integers}, \ref{miniremark:bilinear},
	and \ref{corollary:boundary-operator-dense-subset}, we notice that the
	subgroup $\rho_{U,m,G} [ \mathbf I_m ( U, \mathbf Z ) \otimes G ]$ is
	dense in $\mathscr F_m^{\textup{loc}} (U,G)$ and observe that $\mathbf
	R^n$ may be replaced by $U$ in \cite[4.2.21]{MR41:1976}.
\end{miniremark}

\begin{example} \label{example:locally-integral-flat-chains}
	Suppose $m$ is a nonnegative integer, $n$ is a positive integer, and
	$U$ is an open subset of $\mathbf R^n$.  Recalling
	\ref{remark:classical-integral-flat-chains} and
	\ref{miniremark:isomorphism-for-integers}, the commutative groups
	$\mathscr F_m^{\textup{loc}} ( U )$ and $\mathscr F_m^{\textup{loc}} (
	U, \mathbf Z)$ are isomorphic via
	\begin{equation*}
		\mathscr F_m^{\textup{loc}} ( U ) \simeq \big ( \mathscr
		R_m^{\textup{loc}} (U) \times \mathscr R_{m+1}^{\textup{loc}}
		(U ) \big ) \big / \ker \eta \simeq \mathscr
		F_m^{\textup{loc}} ( U, \mathbf Z ),
	\end{equation*}
	where the second isomorphism is induced by $\iota_{U,m} \times
	\iota_{U,m+1}$; in fact, $\iota_{U,m} \times \iota_{U,m+1}$ maps
	\begin{equation*}
		\ker \eta = \big ( \mathbf I_m^{\textup{loc}} (U) \times
		\mathbf I_{m+1}^{\textup{loc}} ( U ) \big ) \cap \{ ( Q,R )
		\with Q + \boundary R = 0 \}
	\end{equation*}
	onto $\big ( \mathbf I_m^{\textup{loc}} ( U, \mathbf Z ) \times
	\mathbf I_{m+1}^{\textup{loc}} ( U, \mathbf Z ) \big) \cap \{ (S,T)
	\with S + \boundary_{\mathbf Z} T = 0 \}$.  The preceding isomorphisms
	$\mathscr F_m^{\textup{loc}} ( U ) \simeq \mathscr F_m^{\textup{loc}}
	( U, \mathbf Z )$ commute with the boundary operators $\boundary$ and
	$\boundary_{\mathbf Z}$.  Finally, topologising the commutative group
	$\mathscr F_m^{\textup{loc}} (U )$ as in \cite[4.3.16]{MR41:1976},
	they also are homeomorphisms because
	\ref{corollary:separation-smooth-Tietze} and
	\ref{thm:locally-integral-flat-chains} allow us to employ slicing to
	verify that basic neighbourhoods of $0$ in $\mathscr
	F_m^{\textup{loc}} ( U )$ are given by the family of sets
	\begin{equation*}
		\eta \big [ (\mathscr R_m^{\textup{loc}} (U) \times \mathscr
		R_{m+1}^{\textup{loc}} (U)) \cap \{ (Q,R) \with ( \| Q \| + \|
		R \| ) (W) < \delta \} \big ]
	\end{equation*}
	corresponding to all pairs $(W,\delta)$ such that $W$ is open, $\Clos
	W$ is a compact subset of $U$, and $\delta > 0$.
\end{example}

\begin{example} \label{example:locally-flat-real-chains}
	Proceeding as in \ref{example:locally-integral-flat-chains} with
	\ref{remark:classical-integral-flat-chains} and
	\ref{miniremark:isomorphism-for-integers} replaced by
	\ref{remark:classical-real-flat-chains} and
	\ref{example:real-coefficients}, we obtain an isomorphism of chain
	complexes with $\mathbf F_m^{\textup{loc}} ( \mathbf R^n ) \simeq
	\mathscr F_m^{\textup{loc}} ( \mathbf R^n, \mathbf R )$.
\end{example}

\medskip \noindent \textsc{Affiliations}

\medskip \noindent Ulrich Menne \smallskip \newline
Department of Mathematics \\
National Taiwan Normal University \\
No.88, Sec.4, Tingzhou Rd. \\
Wenshan Dist., \textsc{Taipei City 116059 \\
	Taiwan(R.\ O.\ C.)}

\medskip \noindent Christian Scharrer \smallskip \newline Max Planck Institute for Mathematics \newline 
Vivatsgasse 7 \newline 
\textsc{53111 Bonn} \\ \textsc{Germany}

\medskip \noindent \textsc{Email addresses}

\medskip \noindent
\href{mailto:Ulrich.Menne@math.ntnu.edu.tw}{Ulrich.Menne@math.ntnu.edu.tw}
\quad
\href{mailto:Scharrer@mpim-bonn.mpg.de}{Scharrer@mpim-bonn.mpg.de}


\begin{thebibliography}{Whi99b}

\bibitem[All72]{MR0307015}
William~K. Allard.
\newblock On the first variation of a varifold.
\newblock {\em Ann. of Math. (2)}, 95:417--491, 1972.
\newblock URL: \url{https://doi.org/10.2307/1970868}.

\bibitem[Alm68]{MR0225243}
F.~J. Almgren, Jr.
\newblock Existence and regularity almost everywhere of solutions to elliptic
  variational problems among surfaces of varying topological type and
  singularity structure.
\newblock {\em Ann. of Math. (2)}, 87:321--391, 1968.
\newblock URL: \url{https://doi.org/10.2307/1970587}.

\bibitem[AM22]{CVGMT-AlbMar22}
Giovanni Alberti and Andrea Marchese.
\newblock On the structure of flat chains with finite mass, 2022.
\newblock CVGMT preprint.
\newblock URL: \url{http://cvgmt.sns.it/paper/5434/}.

\bibitem[Bou89]{MR979294}
Nicolas Bourbaki.
\newblock {\em General topology. {C}hapters 1--4}.
\newblock Elements of Mathematics (Berlin). Springer-Verlag, Berlin, 1989.
\newblock Translated from the French, Reprint of the 1966 edition.
\newblock URL: \url{https://doi.org/10.1007/978-3-642-61701-0}.

\bibitem[Bou98a]{MR1727844}
Nicolas Bourbaki.
\newblock {\em Algebra {I}. {C}hapters 1--3}.
\newblock Elements of Mathematics (Berlin). Springer-Verlag, Berlin, 1998.
\newblock Translated from the French, Reprint of the 1989 English translation.

\bibitem[Bou98b]{MR1727221}
Nicolas Bourbaki.
\newblock {\em Commutative algebra. {C}hapters 1--7}.
\newblock Elements of Mathematics (Berlin). Springer-Verlag, Berlin, 1998.
\newblock Translated from the French, Reprint of the 1989 English translation.

\bibitem[Bou98c]{MR1726872}
Nicolas Bourbaki.
\newblock {\em General topology. {C}hapters 5--10}.
\newblock Elements of Mathematics (Berlin). Springer-Verlag, Berlin, 1998.
\newblock Translated from the French, Reprint of the 1989 English translation.

\bibitem[Bou03]{MR1994218}
Nicolas Bourbaki.
\newblock {\em Algebra {II}. {C}hapters 4--7}.
\newblock Elements of Mathematics (Berlin). Springer-Verlag, Berlin, 2003.
\newblock Translated from the 1981 French edition by P. M. Cohn and J. Howie,
  Reprint of the 1990 English edition.
\newblock URL: \url{https://doi.org/10.1007/978-3-642-61698-3}.

\bibitem[DPH12]{MR2876138}
Thierry De~Pauw and Robert Hardt.
\newblock Rectifiable and flat {$G$} chains in a metric space.
\newblock {\em Amer. J. Math.}, 134(1):1--69, 2012.
\newblock URL: \url{https://doi.org/10.1353/ajm.2012.0004}.

\bibitem[DPH14]{MR3206697}
Thierry De~Pauw and Robert Hardt.
\newblock Some basic theorems on flat {$G$} chains.
\newblock {\em J. Math. Anal. Appl.}, 418(2):1047--1061, 2014.
\newblock URL: \url{https://doi.org/10.1016/j.jmaa.2014.04.024}.

\bibitem[Fed69]{MR41:1976}
Herbert Federer.
\newblock {\em Geometric measure theory}.
\newblock Die Grundlehren der ma\-the\-ma\-ti\-schen Wissenschaften, Band 153.
  Springer-Verlag New York Inc., New York, 1969.
\newblock URL: \url{https://doi.org/10.1007/978-3-642-62010-2}.

\bibitem[Fed86]{MR833403}
Herbert Federer.
\newblock Flat chains with positive densities.
\newblock {\em Indiana Univ. Math. J.}, 35(2):413--424, 1986.
\newblock URL: \url{https://doi.org/10.1512/iumj.1986.35.35025}.

\bibitem[Fle66]{MR185084}
Wendell~H. Fleming.
\newblock Flat chains over a finite coefficient group.
\newblock {\em Trans. Amer. Math. Soc.}, 121:160--186, 1966.
\newblock URL: \url{https://doi.org/10.2307/1994337}.

\bibitem[HS75]{MR0367121}
Edwin Hewitt and Karl Stromberg.
\newblock {\em Real and abstract analysis}.
\newblock Springer-Verlag, New York-Heidelberg, 1975.
\newblock A modern treatment of the theory of functions of a real variable,
  Third printing, Graduate Texts in Mathematics, No. 25.
\newblock URL: \url{https://doi.org/10.1007/978-3-642-88044-5}.

\bibitem[KM17]{MR3625810}
S{\l}awomir Kolasi\'nski and Ulrich Menne.
\newblock Decay rates for the quadratic and super-quadratic tilt-excess of
  integral varifolds.
\newblock {\em NoDEA Nonlinear Differential Equations Appl.}, 24(2):Art. 17,
  56, 2017.
\newblock URL: \url{https://doi.org/10.1007/s00030-017-0436-z}.

\bibitem[Men12]{MR2898736}
Ulrich Menne.
\newblock Decay estimates for the quadratic tilt-excess of integral varifolds.
\newblock {\em Arch. Ration. Mech. Anal.}, 204(1):1--83, 2012.
\newblock URL: \url{https://doi.org/10.1007/s00205-011-0468-1}.

\bibitem[Men16]{MR3528825}
Ulrich Menne.
\newblock Weakly differentiable functions on varifolds.
\newblock {\em Indiana Univ. Math. J.}, 65(3):977--1088, {\noopsort{a}}2016.
\newblock URL: \url{https://doi.org/10.1512/iumj.2016.65.5829}.

\bibitem[MS18]{MR3819529}
Andrea Marchese and Salvatore Stuvard.
\newblock On the structure of flat chains modulo {$p$}.
\newblock {\em Adv. Calc. Var.}, 11(3):309--323, 2018.
\newblock URL: \url{https://doi.org/10.1515/acv-2016-0040}.

\bibitem[MS22a]{MR3819529-corrigendum}
Andrea Marchese and Salvatore Stuvard.
\newblock {O}n the structure of flat chains modulo {$p$} -- revised version,
  2022.
\newblock In preparation.

\bibitem[MS22b]{arXiv:2209.05955v1}
Ulrich Menne and Christian Scharrer.
\newblock A priori bounds for geodesic diameter. {P}art {II}. {F}ine
  connectedness properties of varifolds, 2022.
\newblock \href {http://arxiv.org/abs/2209.05955v1}
  {\path{arXiv:2209.05955v1}}.

\bibitem[MS22c]{MenneScharrer2-3}
Ulrich Menne and Christian Scharrer.
\newblock A priori bounds for geodesic diameter. {P}art {III}. {A} novel type
  of {S}obolev-{P}oincaré inequality and applications to a variety of
  {P}lateau problems, 2022.
\newblock In preparation.

\bibitem[Pau77]{MR482509}
Sandra~O. Paur.
\newblock Stokes' theorem for integral currents modulo {$\nu $}.
\newblock {\em Amer. J. Math.}, 99(2):379--388, 1977.
\newblock URL: \url{https://doi.org/10.2307/2373825}.

\bibitem[Sim83]{MR756417}
Leon Simon.
\newblock {\em Lectures on geometric measure theory}, volume~3 of {\em
  Proceedings of the Centre for Mathematical Analysis, Australian National
  University}.
\newblock Australian National University Centre for Mathematical Analysis,
  Canberra, 1983.
\newblock URL:
  \url{https://maths-proceedings.anu.edu.au/CMAProcVol3/CMAProcVol3-Complete.pdf}.

\bibitem[Whi99a]{MR1738045}
Brian White.
\newblock The deformation theorem for flat chains.
\newblock {\em Acta Math.}, 183(2):255--271, 1999.
\newblock URL: \url{https://doi.org/10.1007/BF02392829}.

\bibitem[Whi99b]{MR1715323}
Brian White.
\newblock Rectifiability of flat chains.
\newblock {\em Ann. of Math. (2)}, 150(1):\allowbreak{}165--184, 1999.
\newblock URL: \url{https://doi.org/10.2307/121100}.

\bibitem[You18]{MR3743699}
Robert Young.
\newblock Quantitative nonorientability of embedded cycles.
\newblock {\em Duke Math. J.}, 167(1):41--108, 2018.
\newblock URL: \url{https://doi.org/10.1215/00127094-2017-0035}.

\end{thebibliography}
\end{document}